\documentclass[reqno]{amsart}

\usepackage{mathrsfs}
\usepackage[all]{xy}
\usepackage{mathtools}
\usepackage{mathbbol}
\usepackage{wasysym}
\usepackage{amssymb}

\usepackage{MnSymbol}
\usepackage{comment}
\usepackage{enumerate}
\usepackage{xcolor}
\usepackage{enumitem}
\usepackage{ushort}

\usepackage{changepage}
\usepackage{algorithm}
\usepackage[noend]{algpseudocode}
\makeatletter
\def\BState{\State\hskip-\ALG@thistlm}
\makeatother

\usepackage{ifpdf}
\ifpdf %if using pdfLaTeX in PDF mode
  \usepackage[pdftex]{graphicx}
  \DeclareGraphicsExtensions{.pdf,.png,.jpg,.jpeg,.mps}
  \usepackage{pgf}
\else %if using LaTeX or pdfLaTeX in DVI mode
  \usepackage{graphicx}
  \DeclareGraphicsExtensions{.eps,.bmp}
  \usepackage{pgf}
  \usepackage{pstricks}
\fi
\usepackage{epic,bez123}
\usepackage{wrapfig}% package for wrapfigure environment

\usepackage{soul}
\usepackage{url}
\usepackage{tikz}
\usepackage{tikz-cd}
\usetikzlibrary{arrows.meta}

\usepackage[margin=1.2in]{geometry}
\newtheorem{thm}{Theorem}[section]
\newtheorem{prop}[thm]{Proposition}
\newtheorem{cor}[thm]{Corollary}
\newtheorem{lem}[thm]{Lemma}

\newtheorem{mainthm}{}

\theoremstyle{definition}
\newtheorem{defn}[thm]{Definition}

\theoremstyle{remark}

\newtheorem{rem}[thm]{Remark}
\newtheorem{example}[thm]{Example}

\newtheorem{claim}[thm]{Claim}
\newtheorem*{ack}{Acknowledgments}

\newcommand{\U}{X}

\newcommand{\act}{\curvearrowright}

\newcommand{\f}{\mathfrak F}

\newcommand{\PC}{\mathcal P_K(\f)}
\newcommand{\LS}{\ell^{\mathrm{std}}}
\newcommand{\po}{\mathfrak o}
\newcommand{\ax}{\mathrm{Ax}}

\newcommand{\proj}{\textbf{d}}

\newcommand{\len }{\mathrm{Len}}

\usepackage[bookmarks=true, pdfauthor={DING Lihuang}]{hyperref}

\usepackage[capitalize]{cleveref} 

\makeatletter
\newcommand{\nameditem}[2]{%
  \item[(#1)]%
  \def\@currentlabel{(#1)}%
  \label{#2}%
}
\makeatother

\numberwithin{equation}{section}
\title[Growth Gaps and Exponential Genericity of WPD Elements]{Growth Gaps and Exponential Genericity in Acylindrically Hyperbolic Groups}
 \author{Lihuang Ding}	
 \address{Beijing International Center for Mathematical Research\\
Peking University\\
 Beijing 100871, China
P.R.}
 \email{shanquan2@stu.pku.edu.cn}
 
 \author{Wenyuan Yang}

\thanks{W.Y. was supported by National Key R \& D Program of China (2025YFA1017500) and  NSFC (no. 12131009, no.12326601).}



\date{September 22, 2026}

\address{School of Mathematical Sciences and Beijing International Center for Mathematical Research\\
Peking University\\
 Beijing 100871, China
P.R.}
\email{wyang@math.pku.edu.cn}

\begin{document}

\begin{abstract}
Let $G$ be a finitely generated group acting by isometries on a
geodesic metric space with two independent strongly contracting WPD
elements. We prove that, for every finite  generating set
of $G$, strongly contracting WPD elements for this action are
exponentially generic in word-metric balls. The proof establishes
a growth gap for elements with small displacement relative to word
length in an auxiliary projection complex. For finitely generated
acylindrically hyperbolic groups, we establish growth tightness and
cogrowth tightness for every finite generating set. We also show that
the number of conjugacy classes meeting a ball of radius $n$ is
comparable to the cardinality of that ball divided by $n$.
\end{abstract}
 
\maketitle
 
\section{Introduction}

\subsection{Main results} Acylindrically hyperbolic groups introduced by Osin \cite{Osin6} form a broad class of groups exhibiting negatively curved behavior. This class includes non-elementary hyperbolic and relatively hyperbolic groups, mapping class groups, $\mathrm{Out}(F_n)$ for $n\ge 2$, and many groups acting on CAT(0) spaces. Acylindrical hyperbolicity has proved to be a flexible and effective geometric framework for studying this broad range of examples using a common set of tools; see \cite{OsinICM} for an overview. In this paper, we study growth gaps, genericity phenomena and their applications in these groups with respect to counting in word metric balls.

Let $G$ be a finitely generated group, and let $S$ be a finite symmetric generating set. We denote by $\mathrm{Cay}(G,S)$ the Cayley graph equipped with the word metric $d_S$, and write
\[
B_n=\{g\in G:d_S(1,g)\le n\}
\]
for the ball of radius $n$. For a subset $A\subseteq G$, define its exponential growth rate with respect to $S$ by
\[
\omega(A,S)=\limsup_{n\to\infty}\frac{\log |A\cap B_n|}{n}.
\]
Then $\omega(A,S)\le \omega(G,S)$. We say that $A$ has a \emph{growth gap}, or is \emph{growth tight}, if $\omega(A,S)<\omega(G,S).$
Let $\mu_n$ denote the uniform probability measure on $B_n$. A subset $A\subseteq G$ is \emph{exponentially generic} in balls if there exist constants $C>0$ and $0<\lambda<1$ such that for all $n\ge 1$,
\[
\mu_n(A)=\frac{|A\cap B_n|}{|B_n|}\ge 1-C\lambda^n.
\]
Equivalently, $A$ is exponentially generic in balls if and only if its complement has a growth gap.

By definition, an acylindrically hyperbolic group admits a non-elementary acylindrical action on a hyperbolic space. Work of Dahmani--Guirardel--Osin \cite{DGO} and Osin \cite{Osin6} provides equivalent characterizations through hyperbolically embedded subgroups and loxodromic elements satisfying the weak proper discontinuity (WPD) condition. Introduced by Bestvina--Fujiwara \cite{BF2}, this condition expresses a weak form of properness along the orbit of an element; see Definition~\ref{defn:WPD}. We work in the more general setting of actions on geodesic metric spaces admitting two independent strongly contracting WPD elements.

\begin{mainthm}[Exponential genericity of  WPD elements]
\label{MainThmWPDGeneric}
Let $G$ be a finitely generated group acting by isometries on a geodesic metric
space $X$. Suppose that the action admits two independent strongly contracting WPD
elements. Then, for every finite symmetric generating set $S$ of $G$, the set of
strongly contracting WPD elements is exponentially generic in balls. 

Equivalently, there
exist constants $C>0$ and $0<\lambda<1$ such that
\[
\mu_n\bigl(\{g\in G:g\text{ is a strongly contracting WPD element for }G\curvearrowright X\}\bigr)
\ge 1-C\lambda^n
\]
for all $n\ge 1$.
\end{mainthm}

%Equivalently, the set of non-WPD elements has growth gap for every finite generating set. 
Since strongly contracting WPD elements are Morse in the
word metric by a theorem of Sisto \cite{Sisto16}, ~\ref{MainThmWPDGeneric}
also gives exponential genericity of Morse elements.
\medskip

The proof gives more than WPD genericity. It shows that generic elements have
almost maximal stable length, both in the word metric and in the auxiliary
space. For $g\in G$, let
\[
\tau_S(g)=\lim_{n\to\infty}\frac{d_S(1,g^n)}{n},
\qquad
\tau_X(g)=\lim_{n\to\infty}\frac{d(o,g^n o)}{n}
\]
where $o\in X$ is a base point. The quantity $\tau_X(g)$ is independent of the
choice of $o$.

\begin{mainthm}[Stable length exponential deviations]
\label{MainThmStableLength}
Under the assumptions of ~\ref{MainThmWPDGeneric}, for every
$\varepsilon>0$, the set of elements $g\in G$ satisfying
\[
\tau_S(g)\ge (1-\varepsilon)d_S(1,g)
\quad\text{and}\quad
\tau_X(g)\ge (1-\varepsilon)d(o,go)
\]
is exponentially generic in balls.
\end{mainthm}
For mapping class groups, an analogous conclusion for stable word length was observed by Choi \cite[Remark~5.4]{Choi25}, without an exponential rate of convergence. Here we obtain exponential deviation bounds in a broader class of groups, both for stable word length and for stable translation length in the auxiliary space. These estimates are counting
analogues of deviation estimates in the random-walk setting; see
\cite{BMSS,Gou22,ChoiRW25II}.

One may also consider exponential genericity in \emph{spheres} instead of balls by replacing $B_n$ with $S_n=B_n\setminus B_{n-1}$. It is easy to see that these two formulations are equivalent.
One may then derive from \ref{MainThmStableLength} a convergence in spherical mean for the stable
word length:
\[
\lim_{n\to\infty}
\frac{1}{|S_n|}
\sum_{g\in S_n}
\frac{\tau_S(g)}{n}
=1.
\]
Thus generic elements have almost maximal stable word length.  

\medskip

\noindent{\textbf{Applications to mapping class groups}.}
In \cite{ChoiPAExp}, Choi
proved exponential genericity of pseudo-Anosov elements for
infinitely many generating sets obtained from random walks. Ding--Mart\'inez--Granado--Zalloum \cite{DMGZ} obtained exponential genericity for another large
class of generating sets arising from injective spaces. More recently, Choi
\cite{Choi25} proved that pseudo-Anosov elements are  generic in every Cayley
graph of the mapping class group at  a supper-polynomial convergence rate, but exponential genericity for arbitrary
finite generating sets remained open.

\begin{cor}[Mapping class groups]
\label{CorMCGTranslationLength}
Let $\Sigma$ be a closed orientable surface of genus at least $2$, and let
$G\le \mathrm{Mod}(\Sigma)$ be a finitely generated subgroup containing two
independent pseudo-Anosov elements. Let $\mathcal C(\Sigma)$ be the curve
complex. Then, for every finite symmetric generating set $S$ of $G$, there exist constants
$c>0$, $C>0$, and $0<\lambda<1$ such that
\[
\mu_n\left(
\left\{
\phi\in G:
\phi\text{ is pseudo-Anosov and }
\tau_{\mathcal C(\Sigma)}(\phi)\ge c\,d_S(1,\phi)
\right\}
\right)
\ge 1-C\lambda^n
\]
for all $n\ge 1$.
\end{cor}

Thus, for the full mapping class group, Corollary~\ref{CorMCGTranslationLength}
answers \cite[Question~1.4]{ChoiAH} and \cite[Question~1.4]{DMGZ}
affirmatively. We emphasize that the statement applies more generally to every finitely generated subgroup containing two independent pseudo-Anosov elements. The analogous statement in the random walk model is well known. The genericity of pseudo-Anosovs in such subgroups, but without exponential rate, was recently proved by Choi \cite{ChoiAH}.  Our result here gives the stronger translation length conclusion for generic pseudo-Anosov elements on the curve complex.   

\medskip

\noindent{\textbf{Applications to $\mathrm{Out}(F_n)$}.} The free
factor complex $\mathcal{FF}_n$ is hyperbolic by Bestvina--Feighn
\cite{BF14}; fully irreducible elements in $\mathrm{Out}(F_n)$ act loxodromically on
$\mathcal{FF}_n$ and satisfy the WPD condition. Recently, Choi \cite{ChoiAH} proved genericity of fully irreducible elements in $\mathrm{Out}(F_n)$, but without an
exponential rate. His result also applies to general acylindrically hyperbolic groups.

\begin{cor}[$\mathrm{Out}(F_n)$]
\label{CorOutFnTranslationLength}
Let $n\ge 3$, and let $G\le \mathrm{Out}(F_n)$ be a finitely generated subgroup
whose action on the free factor complex $\mathcal{FF}_n$ contains two
independent fully irreducible elements.
Then, for every finite symmetric generating set $S$ of $G$, there exist constants
$c>0$, $C>0$, and $0<\lambda<1$ such that
\[
\mu_n\left(
\left\{
\phi\in G:
\phi\text{ is fully irreducible and }
\tau_{\mathcal{FF}_n}(\phi)\ge c\,d_S(1,\phi)
\right\}
\right)
\ge 1-C\lambda^n
\]
for all $n\ge 1$.
\end{cor}

For the full group $\mathrm{Out}(F_n)$, this answers
\cite[Question~1.5]{ChoiAH} affirmatively by upgrading genericity of fully
irreducible elements to exponential genericity for every finite generating set.
Moreover, the linear lower bound answers the translation length question \cite[Question~1.6]{ChoiAH}.

\subsection{Applications to growth and conjugacy}
Our counting results have applications to three related problems:
growth tightness, cogrowth tightness, and counting conjugacy classes.
We begin with the effect of passing to an infinite normal quotient
on the exponential growth rate.
\medskip

\noindent{\textbf{1. Growth tightness}.}
Following Grigorchuk and
de la Harpe \cite{GriH}, a finitely generated group $G$ is called
\emph{growth tight} if, for every finite generating set $S$ and every infinite
normal subgroup $H\lhd G$,
\[
\omega(G/H,\bar S)<\omega(G,S),
\]
where $\bar S$ denotes the image of $S$ in $G/H$. Growth tightness was first
established for non elementary hyperbolic groups by Arzhantseva--Lysenok
\cite{AL}. It was subsequently developed for many negatively curved actions and
spaces; see Sambusetti \cite{Sam2}, Dal'Bo--Peign\'e--Picaud--Sambusetti
\cite{DPPS}, Yang \cite{YANG6}, Arzhantseva--Cashen--Tao \cite{ACTao}, and
Ding--Mart\'inez-Granado--Zalloum \cite{DMGZ}.

For acylindrically hyperbolic groups we prove the following.

\begin{mainthm}[Growth tightness]
\label{MainThmGrowthTight}
Let $G$ be a finitely generated acylindrically hyperbolic group. Then $G$ is
growth tight. That is, for every finite symmetric generating set $S$ and every infinite
normal subgroup $H\lhd G$,
\[
\omega(G/H,\bar S)<\omega(G,S).
\]
\end{mainthm}

This extends the previous results to the full class of finitely generated acylindrically hyperbolic groups, with respect to arbitrary finite generating sets. The gap in ~\ref{MainThmGrowthTight} cannot be made uniform. Indeed, the
companion paper \cite{DY26A} constructs, for any finite generating set $S$, a
sequence of infinite normal subgroups $H_n\lhd G$ such that
\[
\omega(G/H_n,\bar S_n)\longrightarrow\omega(G,S),
\]
where $\bar S_n$ is the image of $S$ in $G/H_n$.

\begin{cor}
\label{CorMCGGrowthTight}
Mapping class groups of  finite type surfaces are growth tight with
respect to every finite symmetric generating set.
\end{cor}

\begin{cor}
\label{CorOutFnGrowthTight}
For $n\ge 3$, the group $\mathrm{Out}(F_n)$ is growth tight with respect to
every finite symmetric generating set.
\end{cor}

These corollaries answer questions of Arzhantseva--Cashen--Tao for every finite generating set. More precisely, for mapping class groups, Corollary~\ref{CorMCGGrowthTight} settles the Cayley graph case of \cite[Question~3]{ACTao}; the same proof, with the word metric replaced by the corresponding orbit metric, also applies to the marking graph. Similarly, for $\mathrm{Out}(F_n)$, Corollary~\ref{CorOutFnGrowthTight} settles the Cayley graph case of \cite[Question~4]{ACTao}; the same argument also applies to the spine of Outer space. Ding--Mart'inez-Granado--Zalloum \cite{DMGZ} previously constructed infinitely many generating sets of mapping class groups with tight quotient growth, whereas the corresponding result for $\mathrm{Out}(F_n)$ appears to be entirely new.

\medskip

\noindent{\textbf{2. Cogrowth tightness}.}
Combining ~\ref{MainThmGrowthTight} with the growth--cogrowth inequality
proved in the companion paper \cite{DY26A} gives a strict cogrowth bound. 

If
$H\lhd G$ is an infinite normal subgroup, then $\omega(H,S)$, computed using the
restriction of the word metric on $G$, is called the \emph{cogrowth} of $G/H$.
For free groups, Grigorchuk's cogrowth formula \cite{Gri77} relates this quantity to the
spectral radius of simple random walk on $G/H$. Grigorchuk
observed that
\[
\omega(H,S)>\frac{1}{2}\omega(G,S)
\]
for free groups and it was asked in \cite{GriH} whether the same inequality holds for non-elementary
hyperbolic groups. This was answered positively by Jaerisch--Matsuzaki--Yabuki
\cite{MYJ20}. Further generalizations were obtained by
Arzhantseva--Cashen \cite{AC20}, Coulon \cite{Coulon22}, Yang \cite{YANG22},
Choi--Gekhtman--Yang \cite{CGYZ}, and Ding--Mart\'inez-Granado--Zalloum
\cite{DMGZ}. In \cite[Theorem D]{DY26A} we prove that, for every infinite normal
subgroup $H\lhd G$ of an acylindrically hyperbolic group,  
\[
\omega(H,S)+\frac{1}{2}\omega(G/H,\bar S)\ge \omega(G,S).
\]
We therefore obtain the following   corollary by ~\ref{MainThmGrowthTight}. 

\begin{cor}[Cogrowth tightness]
\label{CorCogrowthTight}
Let $G$ be a finitely generated acylindrically hyperbolic group. Then, for every
finite symmetric generating set $S$ and every infinite normal subgroup $H\lhd G$,
\[
\omega(H,S)>\frac{1}{2}\omega(G,S).
\]
\end{cor}

Thus quotient growth and cogrowth fit into a single growth theoretic picture
for acylindrically hyperbolic groups. We point out that the growth and cogrowth tightness results
also holds for a larger class of \textit{confined subgroups} (see Section \ref{SecGrowthTight} for definitions).

We emphasize that the results of this paper and the companion paper \cite{DY26A} are closely related but logically independent, with the sole exception of Corollary~\ref{CorCogrowthTight} here.

\medskip

\noindent{\textbf{3. Counting conjugacy classes}.}
Denote by
\[
\mathcal C_n=\{[g]: [g]\cap B_n\neq\varnothing\}
\]
the set of conjugacy classes admitting a representative of $d_S$-length at most $n$.
The problem of counting conjugacy classes in finitely generated groups has been studied from both geometric and formal viewpoints.    For hyperbolic groups, precise estimates for conjugacy growth were obtained by Coornaert--Knieper and Antol'in--Ciobanu \cite{CoK1,CoK2,AntCio}.  More generally, for statistically convex-cocompact actions with strongly contracting elements, Gekhtman--Yang \cite{GYANG} proved
\[
|\mathcal C_n| \asymp \frac{\mathrm e^{\omega n}}{n}.
\] 
Geometrically, conjugacy classes play the role of free homotopy classes of closed curves, and the familiar factor $1/n$ in conjugacy counting is analogous to the factor $1/t$ in the prime geodesic theorem. 

The preceding conjugacy-counting results are obtained in metrics with
\emph{purely exponential growth}. For a word metric, this means
\[
|B_n|\asymp \mathrm e^{\omega n},
\qquad \omega=\omega(G,S).
\]
For arbitrary finite generating sets of acylindrically hyperbolic
groups, purely exponential growth is not known (but see related results \cite[Theorem B]{DY26A}). Our result does not
require this assumption and compares conjugacy growth directly with
the ball growth function.

\begin{mainthm}[Conjugacy classes]
\label{MainThmConjugacyClasses}
Let $G$ be a finitely generated acylindrically hyperbolic group.  Then, for every finite symmetric generating set $S$, there exist constants $c_1,c_2>0$ such that
\[
c_1\frac{|B_n|}{n}
\le
|\mathcal C_n|
\le
c_2\frac{|B_n|}{n}
\]
for every $n\ge 1$.
\end{mainthm}

This gives a quantitative refinement of Hull--Osin's theorem
\cite{HOsin} that acylindrically hyperbolic groups have exponential
conjugacy growth. In particular, the ordinary and conjugacy
exponential growth rates coincide for every finite symmetric
generating set. This answers affirmatively the question raised by
Ciobanu--Hermiller--Mercier \cite[p.~429]{CHM23} of whether equality
of these growth rates holds for acylindrically hyperbolic groups.

The theorem also determines the decay of the \emph{conjugacy ratio},
introduced by Ciobanu--Cox--Martino \cite{CCM19}:
\[
\operatorname{cr}_S(G)
:=
\limsup_{n\to\infty}
\frac{|\mathcal C_n|}{|B_n|}.
\]
For a finite group, the corresponding ratio equals the probability
that two uniformly chosen elements commute. Motivated by this
connection, Ciobanu--Cox--Martino conjectured that
$\operatorname{cr}_S(G)>0$ if and only if $G$ is virtually abelian.
By \ref{MainThmConjugacyClasses},
\[
\frac{|\mathcal C_n|}{|B_n|}\asymp\frac{1}{n},
\]
so the ratio converges to zero with precisely this order of decay.
This verifies their conjecture for finitely generated acylindrically
hyperbolic groups, extending the hyperbolic case
\cite[Theorem~4.1]{CCM19}, and answers
\cite[Question~1]{CCM19} affirmatively for this class by showing
that the defining limsup is a limit.

There is a further consequence for formal conjugacy growth.
Rivin conjectured that the conjugacy growth series of a hyperbolic
group is rational if and only if the group is virtually cyclic.
Antol\'in--Ciobanu \cite{AntCio} completed the proof by establishing
the stronger conclusion that, for every non-elementary hyperbolic
group and every finite symmetric generating set, this series is
transcendental over $\mathbb Q(z)$. They further conjectured that
transcendence holds for every finitely generated acylindrically
hyperbolic group. By \ref{MainThmConjugacyClasses}, this follows
whenever the chosen word metric has purely exponential growth.

\begin{cor}
\label{CorConjugacyTranscendence}
Let $G$ be a finitely generated acylindrically hyperbolic group and let $S$ be a finite symmetric generating set.  Suppose that $G$ has purely exponential growth with respect to $S$.  Then the conjugacy growth series
\[
\mathcal G_{\mathrm{cj},S}(z)
:=
\sum_{n\ge 1} |\mathcal C_n|z^n
\]
is transcendental over $\mathbb Q(z)$.
\end{cor}  

\subsection{The main geometric mechanism}

Random walks provide another model for genericity, using the distribution
$\nu^{*n}$ of an $n$-step walk on $G$.  %Under suitable hypotheses on the measure and action, a substantial body of work establishes exponential genericity of loxodromic elements in hyperbolic spaces, including pseudo-Anosov elements in mapping class groups and fully irreducible elements in $\mathrm{Out}(F_n)$, as well as weakly contracting elements in more general metric spaces \cite{Rivin,Maher2,MT,CM15,Sisto,KMPT22,Gou22,ChoiRW25II}. A prominent feature of this model is that sample paths typically make linear progress and sublinearly track geodesics in the relevant hyperbolic space.
Under suitable hypotheses,
loxodromic or contracting elements are exponentially generic, and sample
paths make linear progress and track geodesics sublinearly; see, for
example, \cite{MT,Sisto,ChoiRW25II}. Word-metric counting instead uses the uniform measure $\mu_n$ on Cayley
graph balls. The central difficulty is that word geodesics need not map
to quasigeodesics in the auxiliary space, and random-walk genericity
does not automatically transfer to uniform counting. We overcome this
difficulty by establishing a counting analogue of linear progress in
a projection complex, leading to exponential genericity of strongly
contracting WPD elements for every finite generating set.

Let $G\curvearrowright X$ be as in
~\ref{MainThmWPDGeneric}. Fix a finite non-empty set $F$ of independent
strongly contracting WPD elements, and let
\[
\mathfrak F=\{g\ax(f):g\in G,\ f\in F\}
\]
be the associated $G$-invariant system of contracting axes. Let $\mathcal P_K(\mathfrak F)$ be the corresponding projection complex of Bestvina--Bromberg--Fujiwara \cite{BBF}, whose vertices are the axes in $\mathfrak F$. It is known that, for $K\gg 1$, $\mathcal P_K(\mathfrak F)$ is an unbounded quasi-tree on which $G$ acts acylindrically \cite{BBF,BBFS}.

Choose a base axis $\mathfrak o\in \mathfrak F$ and a base point
$o\in \mathfrak o$. The orbit maps to $X$ and to the projection complex are
related by the following diagram:
\begin{equation}\label{dia:threespaces}
\begin{tikzcd}[column sep=large,row sep=large]
\mathrm{Cay}(G,S) \arrow[rr, "{g\mapsto go}"]
\arrow[dr, swap, "{g\mapsto g\mathfrak o}"]
& &
\bigcup_{Y\in\mathfrak F}Y\subseteq X \arrow[dl, "\Psi"] \\
& \mathcal P_K(\mathfrak F) &
\end{tikzcd}
\end{equation}
Here $\Psi(x)$ is a choice of an axis containing $x$. The bounded
projection property makes this choice well defined up to uniformly
bounded distance in the projection complex, so the diagram commutes
up to uniformly bounded error.  

An arbitrary word geodesic need not project to a quasigeodesic in $X$. Nevertheless, we prove that, for a generic element counted in word metric balls, every word geodesic from $1$ to that element coarsely passes through linearly many WPD axes. Moreover, these axes occur in the order prescribed by a \emph{standard path} in the projection complex (see Definition \ref{def:stdpath}). %In this sense, this provides a counting analogue of linear progress for the word metric model.  

\begin{mainthm}[Positive-density WPD-axis recurrence]
\label{MainThmRecurrence}
Let $G\curvearrowright X$ be as in ~\ref{MainThmWPDGeneric}, and let
$\mathfrak F$ be the $G$-invariant system of strongly contracting WPD axes defined above.
For every finite symmetric generating set $S$, there exist constants $\varepsilon>0$ and
$R>0$, and an exponentially generic subset $G_{\mathrm{rec}}\subseteq G$, such
that the following holds.

For every $g\in G_{\mathrm{rec}}$ and every word geodesic $\gamma$ from $1$ to
$g$ in $\mathrm{Cay}(G,S)$, there are at least $\varepsilon d_S(1,g)$ distinct
vertices
\[
Y_1<Y_2<\cdots<Y_m
\]
on the standard path from a fixed point $\mathfrak o$ to $g\mathfrak o$ in
$\mathcal P_K(\mathfrak F)$ such that each $Y_i$ is represented by a coset
$g_iE(f_i)$, with $f_i\in F$, and
\[
d_S(g_i,\gamma)\le R.
\]
\end{mainthm}

%Thus the proof does not require word geodesics to shadow quasigeodesics in the auxiliary space. Instead, it proves a positive density recurrence phenomenon: generic word geodesics make linearly many coarse returns to WPD axes. 
%This statistical recurrence replaces the weak contraction or shadowing hypotheses used in earlier approaches \cite{Choi25,Wiest,GTT}.

The recurrence theorem is  derived from a more basic growth gap
statement. It says that elements whose displacement in the auxiliary space is
too small compared with their word length form an exponentially negligible set.

\begin{mainthm}[Growth gap for short displacement]
\label{MainThmShortElems}
Let $G\curvearrowright X$ be as in ~\ref{MainThmWPDGeneric} and fix a base point $o\in X$. Then, for every finite symmetric generating set
$S$ of $G$, there exists $\varepsilon=\varepsilon(S,o)>0$ such that
\[
A=\{g\in G:d(o,go)\le \varepsilon d_S(1,g)\}
\]
has a growth gap in $G$.
\end{mainthm}

The proof first establishes an analogous growth gap for short displacement in
the projection complex, and then the coarse Lipschitz estimate in
Lemma~\ref{lem:coarseLip2Projcplx} allows us to transfer the  
bound  to $X$ via the commuting diagram depicted in (\ref{dia:threespaces}). 

\subsection{Previous work and comparison}

%We place the results in context. As alluded to above, the genericity problem in the random walk model is well understood following the work of many authors.  In the word metric model, by contrast, one counts uniformly in balls of a Cayley graph. The absence of a counting analogue of linear progress is a major obstruction to proving genericity in this setting.

For arbitrary finite generating sets, Choi \cite{Choi25} proved
genericity of pseudo-Anosov elements with exceptional proportion
decaying faster than every inverse power, using weak contraction in
the Cayley graph. This property is available for mapping class groups
through subsurface projections and distance formulas, but is not
known to hold for acylindrically hyperbolic groups in general. His
subsequent work \cite{ChoiAH} establishes genericity of Morse elements
in this broader setting without a quantitative convergence rate.

Exponential genericity of pseudo-Anosov elements was previously
established for infinitely many finite generating sets by Choi
\cite{ChoiPAExp}, using random-walk methods, and by
Ding--Mart\'inez-Granado--Zalloum \cite{DMGZ}, using generating sets
arising from sufficiently large balls in injective spaces.
Here we prove exponential genericity for every finite generating set.
Our argument uses positive-density recurrence to WPD axes associated
with an auxiliary action and does not assume weak contraction in
the Cayley graph.

%Instead, we use WPD axes for an auxiliary action and build a projection complex from them. The positive-density WPD-axis recurrence theorem shows that generic word geodesics encounter coarsely linearly many of these axes, both in the projection complex and in the Cayley graph. This statistical recurrence is a counting analogue of linear progress for the word metric model and replaces weak contraction in the Cayley graph.

Another approach, due to Wiest, uses geodesic automatic structures together with a shadowing hypothesis for actions on hyperbolic spaces \cite{Wiest}. This applies to certain special generating sets, for instance in Garside-type settings; see also \cite{CaWi17A,CaWi17B}. For hyperbolic groups, Gekhtman--Taylor--Tiozzo adapted ideas from random walks to regular geodesic languages and proved exponential genericity for arbitrary non-elementary actions on hyperbolic spaces \cite{GTT}; see also \cite{GTT20} for an axiomatized framework covering a broader class of groups. Our argument uses recurrence to WPD axes in place of geodesic shadowing.

Finally, earlier growth-gap and genericity results for statistically convex-cocompact actions were obtained for orbit metrics associated to actions with strongly contracting elements \cite{YANG10,YANG11,GYANG}.  Here strong contraction is required
only for an auxiliary action. The counting metric is an arbitrary
word metric, and the argument does not require strongly contracting
elements in the corresponding Cayley graph.

\subsection{Proof strategy}
We outline the geometric and counting arguments that lead from the anchored length estimate to exponential genericity.

The main difficulty is to control the multiplicity of an insertion
argument when word geodesics may backtrack in the auxiliary space.
We address this by measuring the total standard-path length between
successive anchor points in a projection complex. Our anchored length
estimate bounds the excess over the endpoint standard-path length by
an arbitrarily small multiple of the word length, together with an
error proportional to the number of anchors. This estimate controls
the possible locations of inserted contracting pieces. The resulting
counting argument gives a growth gap for small displacement. Further, a guard
decomposition then yields recurrence to contracting WPD axes, and a
conjugacy-shortening argument controls the remaining exceptions.

We next explain these steps in more detail. 

%Let $\mathcal P_K(\mathfrak F)$ be the projection complex associated to the system of contracting WPD axes with respect to the action of $G\act X$. See Figure~\ref{dia:threespaces} for an illustration of the relationships among the three spaces $\mathrm{Cay}(G,S)$, $X$, and $\mathcal P_K(\mathfrak F)$. 

\medskip

\noindent \textbf{(1).} The first ingredient is an \textit{anchored length estimate}, proved in Section~\ref{sec:anchoredlength}. Let $\gamma$ be a word geodesic from $g$ to $h$, and let
\[
Q=\{a_1<\cdots<a_m\}\subseteq\gamma
\]
be a finite ordered set of anchor points. Setting $a_0=g$ and $a_{m+1}=h$, we define in Definition~\ref{def:anchored_set} the anchored length
\[
\LS_Q(g\mathfrak o,h\mathfrak o)
:=
\sum_{i=0}^{m}
\LS(a_i\mathfrak o,a_{i+1}\mathfrak o),
\]
where each summand is the length of the standard path from $a_i\mathfrak o$ to $a_{i+1}\mathfrak o$.

The key estimate, Proposition~\ref{prop:word_geod_on_proj_cplx}, compares this anchored length with the ordinary standard-path length. It asserts that, for every $\varepsilon>0$, there exists $E>0$ such that
\[
\LS_Q(g\mathfrak o,h\mathfrak o)
\le
\LS(g\mathfrak o,h\mathfrak o)
+\varepsilon d_S(g,h)+E|Q|.
\]
The proof combines the projection-complex machinery of Bestvina--Bromberg--Fujiwara~\cite{BBF}, its refinement in~\cite{BBFS}, divergence properties of WPD elements in the word metric~\cite{Sisto16,MS20,GS22}, and the theory of admissible paths and the extension lemma for strongly contracting elements~\cite{YANG6,YANG10}. These tools are recalled in Section~\ref{SecPrelim}.
\medskip

\noindent  \textbf{(2).} The second ingredient is a counting argument, given in Section~\ref{SecShortPC}. Suppose that
\[
A=\{g\in G:\ell^{\mathrm{std}}(\mathfrak o,g\mathfrak o)\le
\varepsilon d_S(1,g)\}.
\]
For $g\in A\cap S_n$, write $g=s_1s_2\cdots s_n$ as a geodesic word. Choose a
set of anchor positions $Q\subseteq\{1,\ldots,n\}$ of size proportional to $n$,
and insert long WPD pieces at these positions. This defines an \textit{insertion map}
from pairs $(g,Q)$ to new group elements in \textsection \ref{subsec:insertion_map}. The main point is to control the
multiplicity of this map, which is done in Lemma \ref{lem:size_of_preimage}. The fellow traveling properties of admissible paths
force the inserted WPD axes to appear in the standard path in the projection
complex, while the anchored length estimate controls the possible positions of
these axes. This gives a growth gap for elements with short projection complex
displacement, proving Theorem \ref{thm:grow_tight_PC} and hence ~\ref{MainThmShortElems}.

\medskip

\noindent  \textbf{(3).} The third ingredient treats elements with \textit{large projection complex displacement} and is carried out in Section~\ref{SecExpGenWPD}.
For such elements $g\in G\setminus A$, every word geodesic admits a guard decomposition with
linearly many guarded blocks. After discarding a small proportion of bad
blocks, good guarded blocks force coarse returns to WPD axes; see Lemma \ref{lem:good_block} for a precise statement. With Theorem \ref{thm:grow_tight_PC}, this proves \ref{MainThmRecurrence} on the
positive-density recurrence to WPD axes for generic elements. Finally, a dichotomy
is  proved in Lemma \ref{lem:split_into_conjugacy}: either many good blocks
survive and one detects WPD behavior, or one obtains a conjugacy shortening
\[
g=vuv^{-1}
\]
with a definite saving in word length. A standard counting argument shows that
the latter exceptional set has a growth gap. Therefore, strongly contracting WPD elements are exponentially generic, completing the proof of ~\ref{MainThmWPDGeneric}.

\medskip

\noindent\textbf{Organization.}
Section~\ref{SecPrelim} recalls the geometric tools. The anchored length
estimate and the growth gap for short projection-complex displacement
are proved in Sections~\ref{sec:anchoredlength} and~\ref{SecShortPC}.
Section~\ref{SecExpGenWPD} establishes recurrence and exponential
genericity, and Section~\ref{SecWPDGen} proves the stable-length estimates.
The applications to growth tightness and conjugacy counting are proved
in Sections~\ref{SecGrowthTight} and~\ref{SecConjClas}, respectively.
%The cogrowth corollary additionally uses the growth--cogrowth inequality from the companion paper~\cite{DY26A}. 

\medskip
\noindent\textbf{AI disclosure.}
No AI tools were used in the development of the mathematical ideas or proofs in this work. ChatGPT  were used solely for final proofreading and language editing.

\ack
We thank Inhyeok Choi and Abdul Zalloum for helpful feedback.

\section{Preliminaries}\label{SecPrelim}

\subsection{Shortest projection maps}
Let $(X,d)$ be a complete metric space and $Y \subseteq X$ be a closed subset.  Given a point $x \in X$, the \textit{shortest projection} of $x$ to $Y$ is defined  as $$\pi_Y(x)=\{y\in Y: d(x, y)
= d(x, Y)\},$$ and  $\pi_Y(U) = \bigcup_{u \in U} \pi_Y(u)$ for a subset
$U \subseteq X$.  

For any subsets $U,V$ in $X$, denote $$\proj_{Y}^\pi(U,V) := \mathrm{diam}(\pi_{Y}(U) \cup \pi_{Y}(V))$$ The following triangle inequality holds 
$$
\proj_{Y}^\pi(U,W)\le \proj_{Y}^\pi(U,V)+\proj_{Y}^\pi(V,W).
$$

Let  $G$ be a group with a finite symmetric generating  set  $S$, i.e.  $S=S^{-1}$.  Let $\mathrm{Cay}(G,S)$ denote the Cayley graph of $G$ with the word metric $d_S$. Assume that $G$ acts by isometry  on $X$. Fix a basepoint $o\in X$ and the orbital map  is defined by $$\Pi:G\longrightarrow X,\qquad g\longmapsto go$$ Let $H$ be a subset in $G$.  We may define two projection maps $G\to H$. In word metric, for any $g\in G$, define the shortest projection map 
$$\pi_H(g)=\{h\in G:\, d_S(g,h)=d_S(g, H)\}.$$ 
We may also push back the shortest projection to $Ho$ in $X$.  Namely, for any $g\in G$, define  
$$\pi_H^X(g)=\{h\in G:\, d(go,ho)=d(go, Ho)\}.$$ 
This is called  $X$-projection in \cite[Definition 3.2]{GS22}.
Given $x,y\in G$, define $$\proj_H^\pi(x,y)=\mathrm{diam}_S(\pi_H(x)\cup\pi_H(y)),\qquad \proj_H^X(x,y)=\mathrm{diam}_S(\pi_H^X(x)\cup\pi_H^X(y))$$
where the two diameters are both taken in word metric.

A map $\Pi: X\to Y$ between two metric spaces is called \textit{$\lambda$-quasi-isometric embedding} for some $\lambda>1$ if 
$$
\forall x,y\in X:\quad \lambda^{-1} d_X(x,y)-\lambda\le d_Y(\Pi(x),\Pi(y))\le \lambda d_X(x,y)+\lambda
$$
If, in addition, $Y$ is contained in the $D$-neighborhood  $N_D(\Pi(X))$  for some $D>0$, then $\Pi$ is called \textit{$\lambda$-quasi-isometry} and $X$ is    \textit{quasi-isometric} to $Y$. If only the right-hand side of the inequality holds, we call $\Pi$ a  \textit{coarsely $\lambda$-Lipschitz} map. 

In general, a  path  in $(X,d)$ is a continuous map $p:I\subseteq \mathbb R\to X$. If $I=[a,b]$ for $a,b\in \mathbb R$, we denote by $p_-=p(a)$ and $p_+=p(b)$ the initial and terminal endpoints. The length of $p$ is denoted by $\ell(p)$ if it is rectifiable. If $\ell(p)\le \lambda d_X(p_-,p_+)+\lambda$  for some $\lambda\ge 1$, then $p$ is called \textit{coarsely $\lambda$-Lipschitz path}. A path $p$ is called \textit{$\lambda$-quasi-geodesic} if any subpath is coarsely $\lambda$-Lipschitz.  

As usual, we often refer to the \emph{image} of $p$ as a path and pick up two points $x,y$ (with  parameters $x=p(s), y=p(t)$ implicit  in context) on $p$ to specify the subpath from $x$ to $y$, denoted by $[x,y]_p$. When $X$ is a graph with combinatorial metric, we also understand a path $p$ as an ordered sequence of \textit{adjacent} vertices $(v_0,v_1,\cdots, v_n)$ on $p$.   

%Let $S_n=\{g\in G: d_S(1,g)=n\}$ and $B_{n}=\{g\in G: d_S(1,g)\le n\}$. Given $\Delta\ge 0$, the $\Delta$-thicken sphere is defined as follows $$S^n(\Delta)=\{g\in G: |d_S(1,g)-n|\le \Delta\}$$ If $\Delta=0$, then $S^n(\Delta)=S^n$.  

%Given an element $g\in G$, a \textit{product decomposition} means a sequence $(h_1, \cdots, h_m)$ where $1\ne h_i \in G$ for $1\le i\le m$ such that $g =  h_1 \cdots h_m$. Set $g_0=1$ and $g_i=s_1\cdots s_i$ for $1\le i\le m$. If $s_i\in S$, a product decomposition of $g$ produces a unique path in $\mathrm{Cay}(G,S)$ from any starting point. Conversely, reading  in order the labels on a path gives a word representing an element. If the path is a  geodesic, the word is called a geodesic word.   

\subsection{Strongly contracting elements}
\begin{defn}\label{SublinearContrDefn}
%Let $\kappa:\mathbb R_+\to\mathbb R_+$ be a sublinear function.
We say that  a subset $Y\subseteq X$ has \textit{strongly $C$-contracting} property if for any $x, y\in X$ with $d(x,y)\le d(x,Y)$,  $\proj_Y^\pi(x,y)\le C$.

A subset $Y$ is called \textit{Morse} if for any $\lambda\ge 1$, there exists $\sigma=\sigma(\lambda)$ so that any $\lambda$-quasi-geodesic with endpoints in $Y$ is contained in the $\sigma$-neighborhood of $Y$. 
%If $\kappa\equiv C$ is constant, then $Y$ is called \textit{strongly contracting}.  

A collection of strongly $C$-contracting subsets shall be referred to
as a $C$-\textit{contracting system}. 
\end{defn}
\begin{rem}\label{rem:basicfacts}
We record a few standard facts used later.
\begin{enumerate}
    \item A strongly contracting subset must be  Morse. 
    \item The strongly contracting property is preserved up to a finite Hausdorff distance.  
    \item Morse property is a quasi-isometric invariant by the very definition. 
\end{enumerate}

\end{rem}

\begin{lem}\label{BigFive}
Let $Y\subseteq X$ be a strongly $C$-contracting subset for $C>0$. Then 

\begin{enumerate}
%\item

%There exists $\sigma=\sigma(C)$ such that   $U$ is $\sigma$--Morse. 
%\item
%For any $r>0$, there exists $\hat C=\hat C(C, r)$ such that a subset $V\subseteq \U$ within $r$--Hausdorff distance to $U$ is $\hat C$--contracting. 

%\item
%If  a geodesic $\gamma$  intersects $N_r(U)$ only at the endpoint $\gamma^-$ for given $r\ge C$, then $$ \pi_U(\gamma^+)\subseteq B(\gamma^-, r+ C).$$
\item
For any  geodesic $\gamma$, we have $$\bigl |\proj_Y^\pi(\gamma^-,\gamma^+)- \mathrm{diam}(\pi_Y(\gamma))\bigr|\leq 4C.$$

\item
For any $y,z\in X$, $\proj_Y^\pi(y, z)\le d(y, z)+ 2C$.

\item 
For any $y, z\in X$, if $\proj^\pi_Y(y, z)>C$, then $\pi_Y(y)\subseteq N_{2C}([y, z])$.

\end{enumerate}
\end{lem}

\begin{defn}\label{BddProj}
Let $\f$  be a family of subsets in $X$. We say that $\f$ has    \textit{bounded intersection property} if there exists a bounded intersection function $\sigma:\mathbb R_+\to \mathbb R_+$ so that for any  $U\ne V\in \f$,
$$
\forall r>0,\quad \mathrm{diam}\left(N_r(U)\cap N_r(V)\right)\le \sigma(r).
$$
If there exists a finite number $\theta>0$ so that $\mathrm{diam}\left(\pi_U(V)\right)\le \theta$ for any $U\ne V\in \f$, then $\f$ has    \textit{$\theta$-bounded projection property}. 
\end{defn}

These two properties are equivalent for strongly contracting subsets. 
\begin{lem}\cite{YANG6}
Assume that $\f$ is a $C$-contracting system of subsets in $X$ for a constant $C>0$. Then  the bounded intersection property of $\f$ is equivalent to the bounded projection property.   
\end{lem}

\begin{defn}\label{defn:contrelem}
An infinite order isometry $h$ on $X$ is called \textit{strongly contracting} if it acts by translation on a strongly contracting quasi-geodesic.   Equivalently, the orbital map $\mathbb Z\to X$ defined by $n\to g^n o$ for some $o\in X$ is a strongly contracting quasi-geodesic (Remark \ref{rem:basicfacts}).   
\end{defn}

%We now consider an isometric action of $G$ on a metric space $X$.  %be a Morse  element, which acts as a strongly contracting isometry on $X$. 

An infinite order element $g$ is called \textit{Morse} if the cyclic subgroup generated by $g$ is a Morse subset in some Cayley graph of $G$. In a proper action on a metric space, a strongly contracting element is a Morse element, i.e. the cyclic subgroup generated is a Morse subset. More generally, a strongly contracting WPD element on a metric space is a Morse element by \cite{Sisto} (see Lemma \ref{lem:WPDisMorse}). 

%Let $G$ be a finitely generated group. A subset $H$ is called \textit{Morse} if any $\lambda$-quasigeodesic with endpoints in $H$ is contained in a $\sigma$-neighborhood of $H$ in a fixed Cayley graph of $G$. Equivalently, the shortest projection to $H$ is sublinearly contracting.  A Morse   element $h$ means an infinite order element so that the subgroup generated by $h$ is Morse. By definition, Morse elements are preserved by quasi-isometry.

It is known that any Morse element $h$ in $G$ is contained in the following group
$$
E(h):=\{g\in G: \exists n\in \mathbb Z_+, (gh^n=h^ng)\lor (gh^n=h^{-n}g) \}
$$
which is the maximal elementary group containing $h$. Equivalently, it can be characterized as the coarse set-stabilizer
$$
E(h)=\{g\in G: d_{Haus}(g\langle h\rangle, \langle h\rangle)<\infty\}
$$
where $d_{Haus}$ denotes the Hausdorff distance between subsets in word metric.

Two infinite order elements $h,k\in G$ are called \textit{independent} if no power of $h$ is conjugate to some power of $k$ in $G$. It amounts to saying that $gE(h)g^{-1}\ne E(k)$ for any $g\in G$. 
\iffalse
\begin{cor}\label{EPlush}
For a Morse element $h\in G$, let $E^+(h)$ be the subgroup of orientation preserving elements in $E(h)$. Then $E^+(h)=\{g\in G: \exists n > 0,\; gh^ng^{-1}=h^n\}$ is of index at most two and contains all Morse elements in $E(h)$.
\end{cor}\fi

\subsection{Admissible paths}

In this subsection, let $\f$ be a $C$-contracting system of subsets in $X$ for a constant $C>0$. Assume that $\f$ has $\theta$-bounded projection for some $\theta>0$.

\begin{defn}[Admissible path]\label{AdmDef}
Given $L,\tau\geq 0$, a path $\gamma$ is called $(L,\tau)$-\textit{admissible} in $\U$ if $\gamma$ is a concatenation of geodesics $p_0q_1p_1\cdots q_np_n$ $(n\in\mathbb{N})$ such that the two endpoints of  $p_i$ for each $0\le i\le n$ lie in some $Y_i\in \f$, and   the following   \textit{Long Local} and \textit{Bounded Projection} properties hold:
\begin{enumerate}
\item[(LL)] For each $1\le i< n$, $p_i$  has length greater than $L$, while $p_0$ and $p_n$ could be trivial;
\item[(BP)] For each $0\le i< n$, we have $Y_i\ne Y_{i+1}$ and $$\max\left\{\mathrm{diam}(\pi_{Y_i}(q_i)),\mathrm{diam}(\pi_{Y_i}(q_{i+1}))\right\}\leq\tau,$$ where by convention $q_0:=\gamma_-$ and $q_{n+1}:=\gamma_+$.
\end{enumerate} 
The collection $\{Y_i: 0\le i\le n\}$ is referred to as the contracting subsets associated with the admissible path.
\end{defn}

\begin{rem}\label{ConcatenationAdmPath}
    The transitional geodesic $q_i$ is allowed to be trivial, in which  case the condition (BP) reduces to checking $Y_i\ne Y_{i+1}$. Admissible paths can be concatenated as follows: let $p_0q_1p_1\cdots q_np_n$ and $p_0'q_1'p_1'\cdots q_n'p_n'$ be $(L,\tau)$-admissible. If $p_n=p_0'$ has length greater than $L$, then the concatenation $(p_0q_1p_1\cdots q_np_n)\cdot (q_1'p_1'\cdots q_n'p_n')$ naturally inherits an $(L,\tau)$-admissible structure.  
\end{rem}

Let $p$ be a path with a natural orientation. We say that  a sequence of points $\{x_i\in p: 0\le i\le n\}$   is  \textit{ordered} if $x_{i+1}\in [x_i, p^+]_p$ for each $0\le i<n$. We write $\{x_0<x_1<\cdots <x_n\}$.

\begin{defn}[Fellow travel]\label{Fellow}
Let   $\gamma = p_0 q_1 p_1 \cdots q_n p_n$ be an $(L, \tau)-$admissible
path. We say $\gamma$ has \textit{$r$-fellow travel} property for some $r>0$   if for any geodesic  
$\alpha$  with the same endpoints as $\gamma$,   there exists a sequence of ordered points $z_i, w_i$ ($0 \le i \le n$) on $\alpha$ such that  
$$d(z_i, (p_{i})_-) \le r,\quad d(w_i, (p_{i})_+) \le r.$$
In particular,  $\mathrm{diam}(N_r(Y_i)\cap \alpha)\ge L-2r$ for each $Y_i\in \f$ associated with $\gamma$. 
\end{defn}
The following result ensures that a locally long admissible path enjoys the fellow travel property.

\begin{prop}\label{admisProp}\cite[Proposition 3.1, Corollary 3.7]{YANG6}
For any $\tau>0$, there exist $L,  r, c>0$ depending only on $\theta, \tau$ and $C$ such that every $(L, \tau)$-admissible path  $\gamma$ has $r$-fellow travel property and  is a $c$-quasi-geodesic. Moreover, if   $z_i,w_i$ denote a shortest projection point of $\gamma_-$ and $\gamma_+$ to $Y_i$ respectively then 
$$
d(z_i, (p_{i})_-) \le r,\quad d(w_i, (p_{i})_+) \le r $$
for any $0\le i \le n.$ 
\end{prop}
 
%\begin{cor}\label{cor:close_to_proj}\cite[Corollary 3.7]{YANG6}
%\end{cor}

\subsection{Projection complex}
Let $\f = \{(Y, \rho_Y)\}$ be a collection of metric spaces. 
\begin{defn}\cite{BBFS}
We say that a family of functions
$\{\proj_Y:Y\in\f\}$, where
$$
\proj_Y:\f^2\setminus\{(Y,Y)\}\longrightarrow\mathbb{R}_{\ge 0},
$$
satisfies the \textit{strong projection axioms} with constant
$\theta>0$ if the following conditions hold:
\begin{enumerate}[label=\textnormal{(SP\arabic*)},leftmargin=*]
    \item\label{axiom:reflect}
    $\proj_Y(U,V)=\proj_Y(V,U)$.

    \item\label{axiom:tri_ineq}
    $\proj_Y(U,V)+\proj_Y(V,W)\ge\proj_Y(U,W)$.

    \item\label{axiom:strong_berstock}
    If $\proj_Y(U,V)>\theta$, then
    $\proj_V(U,W)=\proj_V(Y,W)$ for every
    $W\in\f\setminus\{V\}$.

    \item\label{axiom:bound_proj}
    $\proj_Y(U,U)<\theta$.

    \item\label{axiom:dist}
    The set
    $\{Y\in\f:\proj_Y(U,V)>\theta\}$ is finite for all
    $U,V\in\f$.
\end{enumerate}
The constant $\theta$ is called the \textit{projection constant}.
\end{defn}
\begin{rem}
The  axiom \ref{axiom:strong_berstock} with \ref{axiom:bound_proj} implies  the following axiom: 
\begin{enumerate}[label=\textnormal{(P3)},leftmargin=*]
    \item\label{axiom:berstock}
    For all pairwise distinct $U,Y,V\in\f$, if
    $\proj_Y(U,V)>\theta$, then
    $\proj_V(Y,U)\le\theta$.
\end{enumerate}
This is usually referred to as the Behrstock inequality \cite{Beh2}. The set of axioms with \ref{axiom:berstock} instead of \ref{axiom:strong_berstock} is introduced in  \cite{BBF}.
\end{rem}

%Let us assume that  a subset $\pi_V(U)$ is associated  for each pair of $U\ne V\in \f$. Define $$\proj_Y^{\pi}(U, V) = \diam{\pi_Y(U) \cup \pi_Y(V)}$$ %$$\proj_Y^S(U,V) =\diamS {\pi_Y(U) \cup \pi_Y(V)}$$
%It is clear that $\proj_Y^{\pi}$ satisfies (\ref{axiom:reflect}) and (\ref{axiom:tri_ineq}). 

An important refinement due to \cite{BBFS} shows that a family of projection functions $\proj_Y^{\pi}$ satisfying \ref{axiom:berstock} can be modified within uniformly bounded error so as to satisfy the strong Behrstock inequality \ref{axiom:strong_berstock}. The principal example arises from the shortest  projections associated with a system of strongly contracting subsets having uniformly bounded intersections; see Example~\ref{ex:projectionforcontracting}.

\begin{thm}\cite[Theorem 4.1]{BBFS}\label{thm:forcing_seq}
Assume that $\f = \{(Y, \rho_Y)\}$ is a family of metric spaces with $\{\proj_Y^\pi:
Y \in \f\}$ satisfying \ref{axiom:reflect}, \ref{axiom:tri_ineq}, \ref{axiom:berstock}, \ref{axiom:bound_proj}, \ref{axiom:dist}  with constant $\theta$. Then there exists suitable modified functions $\{\proj_Y: Y\in \f\}$ which satisfy the strong projection axioms for the constant $11\theta$ and
\begin{align}\label{eq:projdist}
\proj^{\pi}_Y - 2\theta\le  \proj_Y\le \proj^{\pi}_Y + 2\theta.    
\end{align}
    
Moreover, for each $U\ne  V\in \f$ there exists a subset of $V$ denoted as  $\widetilde \pi_V(U)$  such that the following holds:
    \begin{enumerate}
        \item $\widetilde \pi_V(U)) \subseteq N_\theta(\pi_V(U))$;
        \item $\proj_Y(U, V) = \mathrm{diam}\left(\widetilde \pi_Y(U)\cup \widetilde \pi_Y(V)\right)$ for any $Y\in \f$;
        \item If a group $G$ acts on $\f$ preserving the metrics and projections $\pi_V(U)$, then $G$ also preserves $\widetilde \pi_V(U)$.
    \end{enumerate}
\end{thm}

The following  notion is fundamental in the theory of projection complex.
\begin{defn}[Standard path]\label{def:stdpath}
Fix a constant $K>3\theta$. Given $U\ne V\in \f$, define $$\f_K(U, V)=\bigl\{Y\in \f \backslash \{U, V\}: \proj_Y(U, V) > K\bigr\}$$ Let us write $$\f_K[U,V]=\f_K(U,V)\cup\{U,V\}$$ on which we define a partial relation $<$ as follows:
    \begin{enumerate}
        \item If $\proj_{Y_1}(U, Y_2) > \theta$ for $Y_1, Y_2\in \f_{K}(U, V)$, we declare $Y_1 < Y_2$;
        \item $U$ and $V$ are the least and greatest elements: $U < Y$ and $Y < V$ for any $Y\in \f_{K}(U, V)$.
    \end{enumerate}
By definition, if $0 < K_1 \le K_2$, then $\f_{K_2}(U, V) \subseteq \f_{K_1}(U, V)$.    
\end{defn}

\noindent\textbf{Convention.}
Unless otherwise mentioned, assume $K> 3\theta$ in the sequel. 

\medskip

The partial relation turns out to be a total order on $\f_K(U, V)$.
%Then we have the following characterization of $\f_K(X, Z)$ for  by \cite{BBF} .
\begin{prop}\label{prop:std_path}
    Given $U\ne V\in \f$, the relation $<$ defines a total order on $\f_{K}[U, V]$ with  the least element $U$ and  the greatest element $V$. Moreover,  
    \begin{enumerate}
        \item \label{ord_from_end} Let $Y_1, Y_2 \in \f_K(U, V)$. Then $$Y_1 < Y_2 
        \quad \Leftrightarrow \quad \proj_{Y_2}(Y_1, V) >\theta
        \quad \Leftrightarrow \quad \proj_{Y_1}(Y_2, V) \le \theta
        \quad \Leftrightarrow \quad \proj_{Y_2}(U, Y_1)\le \theta .$$
        %\item \label{tot_ord} The relation $<$ is a total order on $\f_{K}[U, V]$;
        \item \label{induce_ord} Let $K_1 \le K_2$ and $Y_1, Y_2\in \f_{K_2}(U, V)$. Then $$Y_1 < Y_2 \;\;\text{in}\;\; \f_{K_2}(U, V) \quad \Leftrightarrow \quad Y_1 < Y_2 \;\; \text{in}\;\; \f_{K_1}(U, V)$$
        \item \label{mid_proj} If $Y_1, Y_2, Y_3\in \f_K[U, V]$ and $Y_1 < Y_2 < Y_3$, then $$\proj_{Y_2}(Y_1, Y_3) = \proj_{Y_2}(U, V)$$
        \item \label{crit_in_Y} Let $Y_1, Y_2\in \f_K[U, V]$. If $\proj_Y(Y_1, Y_2) > K$ for some $Y\in\f$, then $Y\in \f_K(U, V)$.
    \end{enumerate}
\end{prop}
\begin{rem}
The preceding statement is a repackaging of several results from \cite{BBFS}; we recall the details for the reader's convenience. The equivalence of the two descriptions of the order in (\ref{ord_from_end}) is proved in \cite[Lemma 2.2]{BBFS}. The assertion (\ref{induce_ord}) follows directly, since the relevant comparisons only involve the threshold $2\theta$, rather than the chosen value of $K$. The total order and the property (\ref{mid_proj}) are established for $\f_{2\theta}(U,V)$ in \cite[Proposition 2.3]{BBFS}, and the same argument applies to $\f_K(U,V)$ for every $K>3\theta$. Finally, (\ref{crit_in_Y}) is precisely \cite[Corollary 2.5]{BBFS}.

Equivalently, one may define the order on $\f_K(U,V)$ by declaring $Y_1<Y_2$ if
$\proj_{Y_1}(U,Y_2)>K,$
or, equivalently, if
$\proj_{Y_2}(Y_1,V)>K.$
Indeed, since $Y_1,Y_2\in \f_K(U,V)$, the middle projection property (\ref{mid_proj}) gives
\[
\proj_{Y_1}(U,Y_2)=\proj_{Y_1}(U,V)>K,
\qquad
\proj_{Y_2}(Y_1,V)=\proj_{Y_2}(U,V)>K
\]
whenever $Y_1<Y_2$ in the above order.   
\end{rem}
%By (\ref{induce_ord}), the total order $<$ is independent of $K$ when $K \ge 3\theta$.

The next two results clarify that the order on a standard path is inherited by its subpaths.
\begin{lem}\label{lem:std_path_prefix}
    Assume $\f_K(Y, Z) \neq \emptyset$ for  $Y, Z\in \f$. Let $W$ be the greatest element of $\f_K(Y, Z)$. Then $$\f_K(Y, Z) = \f_K(Y, W)\cup \{W\}.$$ Further, the order  on $\f_K(Y, W)$ agrees with the induced order on the subset $\f_K(Y, W)\subset\f_K(Y, Z)$. That is, for any two elements $U\ne V\in \f_K(Y, Z)$,  $U < V$ in $\f_K(Y, Z)$ if and only if either $U, V\in \f_K(Y, W)$ and $U < V$, or $U\in \f_K(Y, W)$ and $V = W$. 
    
    A similar  statement holds for the least element of $\f_K(Y, Z)$.
\end{lem}
 
\begin{proof}
    Let $A\in \f_K(Y, Z)$. If $A\neq W$, since $W$ is the greatest element of $\f_K(Y, Z)$, we have $A < W$. By Proposition \ref{prop:std_path}(\ref{mid_proj}), we have $\proj_A(Y, W) = \proj_A(Y, Z) > K$. Thus, $\f_K(Y, Z) \subseteq \f_K(Y, W)\cup \{W\}$.

    For the other direction, let $B\in \f_K(Y, W)$. By Proposition \ref{prop:std_path}(\ref{crit_in_Y}), since $W\in \f_K(Y, Z)$ and $\proj_B(Y, W) > K$, we have $B\in \f_K(Y, Z)$. Thus, $\f_K(Y, W) \subseteq \f_K(Y, Z)$. Since $W\in \f_K(Y, Z)$, $\f_K(Y, W)\cup \{W\} \subseteq \f_K(Y, Z)$ and the lemma holds.
\end{proof}
The following corollary follows by induction.
\begin{cor}\label{cor:sub_segment_std}
    Let $Y,Z\in \f$ and $U, V\in \f_K[Y,Z]$ with $U<V$. Then $$\f_K(U,V) = \{W: W\in \f_K(Y,Z), U<W<V\}.$$
\end{cor}

We define the projection complex following \cite{BBF}.
\begin{defn}\label{defn:projectioncplx}
    Fix $K> 3\theta$. The {\it projection complex} $\PC$ is defined to be the graph with the vertex set $\f$ and the edge set consisting of  edges between vertices $U$ and $V$ with $\f_K(U, V) = \emptyset$. We equip $\PC$ with the combinatorial metric denoted as $\rho$.
\end{defn}
It follows from  \cite[Proposition 3.7]{BBF} that $\PC$ is connected: $\f_K[U,V]=\f_K(U,V)\cup\{U,V\}$ gives a  path from $U$ to $V$, which we call  {\it standard path}.  Moreover, the standard paths have bottleneck property in the sense of Manning, so  the projection complex is a quasi-tree \cite[Theorem 3.5]{BBFS}.

%\begin{rem}\label{lem:std=dist_in_PC}
%    Note that $\rho$ agrees with $\LS$ on the small values. By definition, $\rho(U, V) = 1$ if and only if $\LS(U, V) = 1$. %Further, $\rho(U, V) = 2$ if and only if $\LS(U, V) = 2$. Indeed, $\rho(U, V) = 2$ if and only if $\f_K(U, V) \neq \emptyset$ and there exists $Y$ such that $\f_K(U, Y) = \f_K(V, Y) = \emptyset$, which is equivalent to $\LS(U, V) = 2$. %%CONTAINS GAP
%\end{rem}
%\begin{proof}
    %$\LS(X, Y) = 1$ if and only if $\f_K(X, Y) = \emptyset$, which is equivalent to $\rho(X, Y) = 1$. 
    %$\LS(X, Y) \le 2$ if and only if $|\f_K(X, Y)| \le 1$. On the other hand, $\rho(U, V) = 2$ if and only if $\f_K(U, V) \neq \emptyset$ and there exists $Y$ such that $\f_K(U, Y) = \f_K(V, Y) = \emptyset$, which is equivalent to $\LS(X, Y) = 2$.
%\end{proof}

\begin{lem}[Bottleneck property]\label{lem:bottleneck_in_PC_rho} 
    For  $K> 3\theta$, any path from $U$ to $V$ in $\PC$ intersects the $2$-neighborhood of $Y$ for every $Y\in \f_K(U, V)$. In particular, $\PC$ is a quasi-tree. 
\end{lem}
The standard path triangle enjoys the following almost-tripod property \cite[Lemma 3.6]{BBFS}.
\begin{lem}[Almost-tripod property]\label{lem:triangle_in_PC}
    Let $U, V, W\in \f$. Then $\f_K(U, V) \cup \f_K(V, W)$ contains all but at most 2 consecutive elements of $\f_K(U, W)$. In particular, there exists a decomposition $$\f_K(U, W) = \alpha\sqcup \beta\sqcup \gamma$$ where $\alpha\subseteq \f_K(U, V)$, $|\beta|\le 2$, $\gamma\subseteq \f_K(V, W)$, and $\alpha < \beta < \gamma$ in $\f_K(U, W)$.
\end{lem}

It will be convenient to  use the \textit{standard path length function} 
\begin{align}\label{eqn:standardpathlength}
\LS(U,V)=\begin{cases}
\, |\f_K(U,V)|+1, & U\ne V\\
\, 0, & U=V
\end{cases}
\end{align} 
which is   the length of the standard path from $U$ to $V$. By \cite[Corollary 3.7]{BBFS},  
\begin{align}\label{rem:std_len_quasi_isom}
  \forall Y\neq Z\in \f,\quad  \rho(Y, Z)\le \LS(Y, Z)\le 2\rho(Y, Z) -1
\end{align}
so $\LS$ is comparable with the graph metric $\rho$. 
Then we have the following bottleneck property in terms of $\LS$, which is a simple corollary of Lemma~\ref{lem:bottleneck_in_PC_rho} and Eq.~(\ref{rem:std_len_quasi_isom}).
\begin{lem}[Bottleneck property]\label{lem:bottleneck_in_PC}
    For  any $U, V, Y\in \PC$ with $Y\in \f_K(U, V)$, any path from $U$ to $V$ in $\PC$ contains a point $Z$ such that $\LS(Z, Y)\le 3$.
\end{lem}

The following weak triangle inequality for $\LS$ will be frequently invoked.
\begin{lem}\label{lem:std_tri_ineq}
    Let $U, V, W\in \f$. Then $\LS(U, V) \le \LS(U, W) +\LS(W, V) + 1$.
\end{lem}
\begin{proof}
    When $U=V$ or $V=W$ or $U=W$, the inequality holds trivially. Suppose that $U,V, W$ are distinct. We may decompose $$\f_K(U, V)=\alpha\sqcup \beta\sqcup \gamma$$ as in Lemma \ref{lem:triangle_in_PC}  and according to the definition of $\LS$ in \cref{eqn:standardpathlength}, $$
    \begin{aligned}
    \LS(U, V) &= |\f_K(U, V)|+1 \\
    &= |\alpha|+|\beta|+|\gamma| + 1\\
    &\le |\f_K(U,W)|+2+|\f_K(W,V)|+1\\
    &\le  \LS(U,W) + \LS(W,V) + 1.    
    \end{aligned}$$
    The lemma is proved.
\end{proof}

In the next two lemmas, we examine how the standard path is affected  when the endpoints are perturbed in a small neighborhood.

\begin{lem}\label{lem:common_prefix}
    Given $U, V\in \f$, write $\f_K(U, V) = \{Y_1< Y_2 < \cdots< Y_n\}$ for $n\ge 2$. Assume that $W\in \f$ satisfies $d := \LS(V, W)\le n-2$. Then $$\{Y_1< \cdots< Y_{n-d-1}\}\subseteq \f_K(U, W).$$
\end{lem}
\begin{proof}
    By Lemma \ref{lem:triangle_in_PC}, $\f_K(U, V)$ can be decomposed as $\alpha\sqcup \beta\sqcup \gamma$ where $\alpha\subseteq \f_K(U, W)$, $|\beta|\le 2$, $\gamma\subseteq \f_K(V,W)$, and $\alpha<\beta<\gamma$. Assume $W\ne V$; otherwise the proof is trivial. Then $|\f_K(V, W)|=\LS(V, W)-1 = d-1$, so $|\gamma| \le d-1$. Now,   $$|\alpha|= n - |\beta| - |\gamma| \ge n - 2 - (d-1) = n-d-1$$ Since $\alpha < \beta < \gamma$ is an ordered decomposition of $\f_K(U, V)$, we derive $\{Y_1< \cdots< Y_{n-d-1}\}\subseteq \alpha\subseteq \f_K(U, W)$.
\end{proof}
Applying Lemma~\ref{lem:common_prefix} twice yields the following lemma.
\begin{lem}\label{lem:common_mid}
    Let $U, V\in \f$ and $\f_K(U, V) = \{Y_1< Y_2 < \cdots< Y_n\}$. Then for any $d>0$ and $X,Y\in \f$, if $\LS(U,X)\le d$ and $\LS(V,Y)\le d$, then $$\{Y_{d+2}, \cdots, Y_{n-d-1}\}\subseteq \f_K(X, Y).$$
\end{lem}

The following is a simple extension of  the guard property in \cite[Lemma 3.3]{BBFS}. %Roughly speaking, when a point $Y\in \PC$ moves outside the ball of radius $d$ centered at $Z\in \PC$, the smallest $(d-3)$-elements of $\f_K(Z,Y)$   remain the same.

\begin{lem}\label{lem:more_guard_out_ball}
    Let $\{X_1,X_2, \cdots,X_m\}$ be a path in $\PC$. Given an integer $d\ge 3$,  let $Z\in \f$ such that $\LS(X_i, Z)\ge d$ for each $1\le i\le m$. Then for each $1 \le i \le m$, the first $(d-3)$ elements of $\f_{K}(Z, X_i)$ coincide. 
\end{lem}
\begin{proof}
    By induction,  it suffices to prove the lemma for $m=2$. Let $\f_K(Z, X_1) = \{Y_1<Y_2<\cdots<Y_n\}$ where $n\ge d-1$. By Lemma \ref{lem:common_prefix}, since $\LS(X_1, X_2) = 1$, we have $\{Y_1, \cdots, Y_{n-2}\}\subseteq\f_K(Z, X_2)$. Thus, $\{Y_1, \cdots, Y_{d-3}\}\subseteq \f_K(Z, X_1)\cap \f_K(Z, X_2)$. By the same reason, the smallest $(d-3)$ elements of $\f_K(Z, X_2)$ are contained in $\f_K(Z, X_1)$. Hence, by Proposition \ref{prop:std_path}, the smallest $(d-3)$ elements of $\f_K(Z, X_1)$ and $\f_K(Z, X_2)$ agree.
\end{proof}

\subsection{WPD elements and $\kappa$-divergence}
Assume that $G$  acts by isometry on a metric space $(X,d)$. Given  $r\ge 0$, define the \textit{$r$-coarse stabilizer} of a pair of $x,y\in X$: $$\mathrm{Stab}_r(x,y):=\{g\in G: d(x,gx)\le r, d(y,gy)\le r\}$$

\begin{defn}\label{defn:WPD}\cite{BF2}
We say that  $f\in G$  satisfies the \textit{weak proper discontinuity condition} (or that $f$ is a \textit{WPD element}) if for all $r>0$ and $x \in X$ there exists $L \in \mathbb N$ such that
$$|\mathrm{Stab}_r(x,f^Lx)|<\infty$$    
\end{defn}

\begin{rem}\label{rmk:WPD}
Assume, in addition, that $f\in G$ is a strongly contracting  isometry. 
\begin{enumerate}
    \item By \cite[Lemma 6.4]{DGO},  $L$ could be chosen so that $|\mathrm{Stab}_r(x,f^Mx)|<\infty$ holds for any $M>L$.
    \item 
    In practice,  it  suffices to verify the finite coarse stabilizer for a \emph{suffciently large} $r$, and those points $x$ in a \emph{large $R$-neighborhood} of the axis $\ax(f)$. The values of $r, R$ may depend on the contracting constant of $\ax(f)$. Indeed, \cite[Propoisition 5.31]{DGO}   says the same conclusion for acylindrical actions on hyperbolic spaces.  However, the hyperbolicity could be easilly replaced with an argument using the  strongly contracting property of $\ax(f)$.  
\end{enumerate}
\end{rem}

The following useful facts about WPD elements shall be used implicitly. The items (1)(2) are proved in \cite[Sec. 6.1]{DGO} and (2) is due to Sisto \cite[Theorem 1]{Sisto16}. 
\begin{lem}\label{lem:WPDisMorse}
Assume that $G$ acts by isometry  on a metric space $X$. Let $f\in G$ be a strongly contracting WPD   element. Then
\begin{enumerate}
    \item 
    $f$ is contained in the unique maximal elementary subgroup $E(f)$.
    \item 
    The system 
    $\{gE(f)o: g\in G\}$ has bounded projection property.
    \item 
    $f$ is a Morse element in $G$. 
\end{enumerate}   
\end{lem}
In the sequel, we shall call  $\ax(f):=E(f)o$ the \textit{quasi-axis} or \textit{WPD axis} of $f$.

A \textit{diverging} function $\kappa : \mathbb{R}_+ \to \mathbb{R}_+$ means that it is non-decreasing  and $\kappa(R)\to \infty$ as $R\to \infty$. 

\begin{defn}\label{defn:divergence}\cite{GS23}
Fix   a finite symmetric generating set $S$ of $G$, a diverging function  $\kappa$ and $N>0$. Let  $f\in G$ be a strongly contracting WPD element and denote $H:=E(f)$. We say that $f$ is $(\kappa, N)$-\textit{divergent} for the action $G\act X$ if  $p$ is a path from $x$ to $y$ in $\mathrm{Cay}(G,S)$  so that $p\cap N_R(H)=\emptyset$  and $\proj_{Ho}^\pi(xo,yo) \ge N$, then $p$ has length at least $\kappa(R)$. 
\end{defn}
\begin{rem}\label{rmk:divergence}
Up to modifying the divergence function $\kappa$, the divergence property of an element is independent of the choice of generating set $S$. In \cite[Definition 2.2]{GS23}, the subgroup $H$ is taken to be $\langle f\rangle$. By Lemma \ref{lem:WPDisMorse}, the subgroup $\langle f\rangle$ has finite index in $E(f)$. Therefore, under the WPD assumption, the two definitions are equivalent.
\end{rem}

In \cite[Lemma 2.5]{GS23}, it is proved that a loxodromic WPD isometry on a hyperbolic space is divergent. (A kind of converse is given in \cite[Lemma 3.8]{GS22}.) This extends  to a strongly contracting WPD isometry via a combination of known results in literature. %We include the proof for completeness. 
 
\begin{lem}\label{lem:WPDdivergent}
Assume that $G$ acts by isometry  on a metric space $X$. Let $f\in G$ be a strongly contracting WPD   element.  Fix a finite generating set $S$. Then there exist a diverging function $\kappa$  and a constant $N$ so that $f$ is $(\kappa,N)$-divergent for the action $G\act X$.
\end{lem}
 
\begin{proof}
By \cite[Theorem 5.10]{BBFS}, the system $\f=\{g\ax(f): g\in G\}$ satisfies projection axiom and  $G$ acts acylindrically  on the projection complex $\PC$. Let  $\mathcal C_K(\f)$  denote the quasi-tree  of spaces, which is obtained from $\PC$ by replacing each vertex $U\in \f$ with the corresponding axis $U$ in $X$. We call $U$  the vertex space in $\mathcal C_K(\f)$. The edges between $U$ and $V$ in $\PC$ are replaced with the union of edges with length comparable to $K$ between points in $\pi_U(V)$ and $\pi_V(U)$. By \cite[Theorem 6.9]{BBFS}, $G$ acts acylindrically on $\mathcal C_K(\f)$.  Thus,  by \cite[Lemma 2.5]{GS23}, $f$ is divergent with respect to the action on $\mathcal C_K(\f)$. By \cite[Corollary 4.10]{BBF}, the shortest projection of the vertex space $U$ to $V$ in $\mathcal C_K(\f)$ agrees with that of $U$ to $V$ in the original space $X$ up to a uniformly bounded error.  Thus, $f$ is also divergent for the action on $X$. The proof is complete.  
\end{proof}

Two  strongly contracting WPD elements $h,k\in G$ are called \textit{independent} if no power of $h$ is conjugate to some power of $k$ in $G$. Equivalently, the system $\{g\ax(h), g\ax(k): g\in G\}$ has bounded projection property. The next lemma provides a useful method to construct admissible paths.
\begin{lem}[Extension Lemma]\label{extend3}
Let $h_{1}, h_{2}, h_{3} \in G$ be pairwise independent, contracting WPD elements. Denote $\f=\{g\ax(h_i):i=1,2,3, g\in G\}$. Then there exist constants  $L, \tau>0$ depending only on $\f$ with the following property.

For each $1\le i\le 3$, choose an element $f_i\in \langle h_i\rangle$ such that $d(o,f_io)\ge L$, and let $F$ be the set of these elements. Then for any $g,h\in G$, there exists $f \in F$ such that the path  $$\gamma:=[o, go]\cdot(g[o, fo])\cdot(gf[o,ho])$$ is $(L, \tau)$-admissible with respect to the   system $\f$. 
%\item
%The point  $go$  is an $(r, f)$--barrier for any geodesic  $[\gamma^-,\gamma^+]$.	%In particular, the path $[o, go][go, gfo][gfo, gfho]$ is a $(1, 4\varepsilon_0)$--quasi-geodesic.
%\end{enumerate}
\end{lem}

\begin{rem} 
The (BP) and (LL) conditions in Definition \ref{AdmDef} are local conditions, so we may connect any number of elements  $g\in G$ using $F$ to obtain an admissible path. We refer the reader to \cite{YANG10} for a precise formulation.
\end{rem}

To conclude this discussion, we describe the principal example used
throughout the paper.

\begin{example}\label{ex:projectionforcontracting}
Let $X$ be a metric space, and let $\f$ be a $C$-contracting system
of subsets of $X$. For distinct $U,V\in\f$, let $\pi_V(U)$ denote
the shortest-point projection of $U$ to $V$, and define
$$\proj_Y^\pi(U,V)=\mathrm{diam}(\pi_Y(U)\cup\pi_Y(V)).$$
If $\f$ has the bounded projection property, then there exists
$\theta_0>0$, depending only on the contracting and bounded projection
constants, such that the family
$\{\proj_Y^\pi:Y\in\f\}$ satisfies
\ref{axiom:reflect}, \ref{axiom:tri_ineq},
\ref{axiom:berstock}, \ref{axiom:bound_proj}, and
\ref{axiom:dist} with constant $\theta_0$. By Theorem~\ref{thm:forcing_seq}, there are modified projection sets
$\widetilde\pi_Y(U)$ such that the functions
$$\proj_Y(U,V)=\mathrm{diam}(\widetilde\pi_Y(U)\cup\widetilde\pi_Y(V))$$
satisfy the strong projection axioms with constant $11\theta_0$ and
differ from the original projection distances by a uniformly bounded
amount.

From now on, we use the modified projection sets and write
$\pi_Y(U)$ in place of $\widetilde\pi_Y(U)$. We also replace
$11\theta_0$ by $\theta$. Thus, all subsequent projection distances
refer to the modified projections, while their comparison with the
original shortest-point projections is understood up to the uniform
error supplied by Theorem~\ref{thm:forcing_seq}.
\end{example}

For sufficiently large $K>3\theta$, let $\PC$ be the projection
complex associated with $\f$, equipped with its graph metric $\rho$.
Let
$$
Z:=\bigcup_{Y\in\f}Y
$$
with the metric induced from $X$. Define the set-valued map
$$
\Phi(x):=\{Y\in\f:x\in Y\}\subseteq\PC.
$$
This map is coarsely well defined. Indeed, if $Y,Y'\in\Phi(x)$, then
the bounded projection property implies
$\f_K(Y,Y')=\emptyset$ for sufficiently large $K$. Hence
$\rho(Y,Y')\le1$, and therefore any choice of a vertex in $\Phi(x)$ defines the same
coarse map up to distance one.

\begin{lem}\label{lem:coarseLip2Projcplx}
For every sufficiently large $K$, there exists $L=L(K)>1$ such that,
for all $x,y\in Z$ and all choices
$U\in\Phi(x)$ and $V\in\Phi(y)$,
$$
\rho(U,V)\le Ld(x,y)+L.
$$
In particular, any choice of representatives defines a coarsely
$L$-Lipschitz map $\Phi:Z\to\PC$.
\end{lem}

\begin{proof}
Let
$U=Y_0<Y_1<\cdots<Y_n=V$
be the standard path from $U$ to $V$, so that
$n=\LS(U,V)$. A \textit{standard lift} $\gamma$ of this path, with endpoints $x$ and
$y$, is obtained by joining the consecutive shortest
projections in $\pi_{Y_i}(Y_{i+1})$ and $\pi_{Y_{i+1}}(Y_{i})$ by
$d_X$-geodesics. There exist uniform constants $E,\tau>0$, depending only on the
contracting and bounded projection constants, such that $\gamma$ is
a $(K-E,\tau)$-admissible path with the associated contracting subsets $Y_i$ ($0<i<n)$. In particular, each of its $n-1$
interior contracting pieces has length at least $K-E$. See \cite[Lemma 4.5]{HLY} for full detail.

Choose $K$ sufficiently large that Proposition~\ref{admisProp}
applies, and let $c=c(K)>1$ be the resulting quasigeodesic constant.
Then $(n-1)(K-E)\le\len(\gamma)\le c\,d(x,y)+c.$
Since 
$\rho(U,V)\le n$,
$$
\rho(U,V)
\le \frac{c}{K-E}d(x,y)+\frac{c}{K-E}+1.
$$
The conclusion follows after enlarging the constant.
\end{proof}

\begin{lem}\label{lem:lift_wpd}
For
every sufficiently large $K>3\theta$, every element acting
loxodromically on $\PC$ is a strongly contracting WPD element for the
action on $X$.
\end{lem}

\begin{proof}
Let $g\in G$ act loxodromically on $\PC$. Since the action
$G\curvearrowright\PC$ is acylindrical by \cite[Theorem 5.10]{BBFS}, $g$ is WPD on $\PC$. We first transfer the WPD property to $X$. Choose $z\in Z$ and
$U\in\Phi(z)$. By Lemma~\ref{lem:coarseLip2Projcplx}, there exists
$L>1$ such that, for every $h\in G$,
$$
\rho(U,hU)\le Ld(z,hz)+L.
$$
The same estimate holds at $g^nz$, with $g^nU\in\Phi(g^nz)$. Hence,
for every $r>0$, the $r$-coarse stabilizer of
$(z,g^nz)$ in $Z$ is contained in a uniformly bounded coarse
stabilizer of $(U,g^nU)$ in $\PC$. Since $g$ is WPD on $\PC$, it is
therefore WPD on $Z$, and thus WPD on $X$ by Remark \ref{rmk:WPD}(2).

It remains to prove that $g$ is strongly contracting. 

First, every loxodromic element $g$ admits a unique axis $\gamma$ so that  any subpath of $\gamma$ is a standard path. Fixing a basepoint $\po\in \PC$, $n\mapsto g^{n}\po$ gives a quasi-geodesic denoted as $\alpha$. The sequence of standard paths $\alpha_n=\f_K[g^{-n}\po,g^n\po]$ remains within a $R$-neighborhood of $\alpha$ for some $R>0$ by Morse Lemma. By the almost-tripod property of $\alpha_n$,  a Cantor's diagonal argument extracts a bi-infinite standard path $\gamma$ from $\alpha_n$ as $n\to\infty$. Since $\gamma\subset N_R(\alpha)$, the almost-tripod property implies  $g\gamma=\gamma$, so the axis is unique and invariant under $g$. 

Next, as in the proof of Lemma \ref{lem:coarseLip2Projcplx}, the standard path $\gamma$ lifts to a $(K-E,\tau)$-admissible path $\tilde \gamma$ whose contracting subsets $Y_i$ are given by the corresponding vertices on $\gamma$. If $K$ is large enough, $\tilde \gamma$ is strongly contracting by \cite[Lemma 2.19]{YANG10}. Since $g$ acts cocompactly by translation on
$\widetilde\gamma$, it is strongly contracting on $X$. 
\end{proof}

\section{Anchored length of word geodesics in the projection complex}\label{sec:anchoredlength}
Let $(X,d)$ be a geodesic metric space and $G$ be a group with a finite symmetric generating set $S$. Assume that $G$ acts by isometry on $X$ with at least two independent  strongly contracting WPD elements. Fix a basepoint $o\in X$. Denote  $\Pi:\mathrm{Cay}(G,S)\to X$ the orbital map defined by $g\mapsto go$.

\subsection{Setup}\label{ssec:setup} We set up the necessary objects  used throughout  this and next Sections \ref{SecShortPC} and \ref{SecExpGenWPD}. 
\medskip 

\noindent\textbf{Finite set $F$ of WPD elements.} Fix \textit{any}  non-empty finite set  of independent  contracting contracting WPD elements $F\subseteq G$ for the action of $G$ on $X$. 
    \begin{itemize}
        \item Let $\beta>0$ be a constant such that $\Pi$ is a $\beta$-Lipschitz map. Furthermore, for each $f\in F$, the restriction  $E(f)\to \ax(f)=E(f)o$ is a $\beta$-quasi-isometry.
        \item 
        Let  $\kappa=\kappa(S,F)$ be a diverging function and $N=N(S)>0$ so that each $f\in F$ is $(\kappa, N)$-divergent  with respect to the map $\Pi: \mathrm{Cay}(G,S)\to X$.     
        
    \end{itemize} 
\medskip

\noindent\textbf{Projection complex.} 
\begin{itemize}
        \item 
        Let $C>0$ be the common contracting constant for the quasi-axes $\ax(f)$ with  $f\in F$. 
        \item 
        Let $\theta=\theta(\f)$ be the projection   constant so that
    $$\f = \{g\cdot \ax(f) : g\in G,\ f\in F\}$$   
    satisfies the projection axioms. See Lemma \ref{lem:WPDisMorse} and Example \ref{ex:projectionforcontracting}. 
\end{itemize}Let $\PC$ be the associated projection complex with graph metric $\rho$, where 
    \begin{align}\label{eq:defnK}
    K > \max\{4\theta + \beta + 5C, N + 4\theta\}.    
    \end{align} 
    (Note that $K$  depends on the choice of generating set $S$ via the dependence on $N$.)
    Fix $f_0\in F$ and choose the basepoint $\po = \ax(f_0)$ in $\PC$. In this section, the $r$-neighborhood of a subset $\alpha\subseteq \PC$ means the neighborhood with respect to $\LS$, i.e. the set consisting of $Y\in \PC$ such that $\LS(Y, A)\le r$ for some $A\in \PC$.
\medskip

We start with an elementary observation which will be used implicitly. %The path in $\mathrm{Cay}(G, S)$ labels a path in the projection complex.
\begin{lem}\label{lem:word_is_PC_path}
    Let $\gamma=\{g_0, g_1, \cdots, g_m\}$ be a path in $\mathrm{Cay}(G, S)$. Then  $$\gamma\po:=\{g_0\po, g_1\po, \cdots, g_m\po\}$$ is a path in $\PC$. 
\end{lem}
\begin{proof}
    It suffices to prove that for any $s\in S$, $\rho(\po, s\po)\le 1$.
    
    Let  $s\in S$ and $Y\in \f\setminus\{\po,s\po\}$. Since $d(o, so)\le \beta d_S(1, s) = \beta$, by Lemma~\ref{BigFive} we have $$\proj_Y^{\pi}(o, so)\le d(o, so) + 2C\le \beta + 2C.$$ Noting $o\in \po$ and $so\in s\po$, the axiom \ref{axiom:bound_proj} implies $$\proj_Y^{\pi}(\po, s\po)\le \mathrm{diam}(\pi_Y(\po)) + d(o, so) + \mathrm{diam}(\pi_Y(s\po))\le 2\theta + \beta + 2C.$$ Applying Eq. (\ref{eq:projdist}) in Theorem~\ref{thm:forcing_seq} yields $$\proj_Y(\po, s\po)\le \proj_Y^{\pi}(\po, s\po) + 2\theta \le 4\theta + \beta + 2C < K,$$ by the choice of $K$ in Eq. (\ref{eq:defnK}). Thus,  $\f_K(\po, s\po) = \emptyset$, so $\rho(\po, s\po)\le 1$.
\end{proof}

\iffalse
\begin{defn}
    Let $a, b\in G$ and $\gamma$ be a word geodesic from $a$ to $b$. For a finite subset $A=\{g_0=a, g_1, g_2, \cdots, g_m, g_{m+1} = b\}\subseteq \gamma$ with $m>0$ such that $g_0, g_1, \cdots, g_{m+1}$ are linearly ordered along $\gamma$ with $m\ge 0$. Define the \emph{total length} of $A$ 
    \[L(\gamma, A):=\sum_{i=0}^m \LS(g_i\cdot O, g_{i+1}\cdot O).\]
\end{defn}\fi

 \subsection{Anchored length and the main proposition}
 Let $\LS:\f\times\f\to\mathbb{R}_{\ge 0}$ be   the standard length function defined in Eq. (\ref{eqn:standardpathlength}). This is   the length of the standard path from $g \po$ to $h \po$. We extend the  length function to the following anchored version.
\begin{defn}[Anchored standard length]\label{def:anchored_set}
    Let $A=\{a_1< a_2< \cdots< a_m\}$  be a finite  \textit{ordered}  set of distinct elements in $G$ with $m\ge 0$. Given  two elements $g, h\in G$,  the \textit{standard $A$-length from $g$ to $h$} is defined as   
    \[\LS_A(g\po,h\po)\;:=\;\sum_{i=0}^m \LS(a_i \po, a_{i+1} \po) \]
    where $a_0=g, a_{m+1}=h$. The $A$ will be referred to as the  \textit{anchored set}. 

    In the sequel, we usually choose the  set $A$ on some word geodesic $\gamma$ from $g$ to $h$: $A=\{ a_1< a_2< \cdots< a_m\}$  on  $\gamma$. In this case,  $A$ is said \emph{geodesically anchored} from $g$ to $h$.
\end{defn}
Removing $g$ or $h$ from $A$ does not change
$\LS_A(g\po,h\po)$. Thus, we may always assume that
$A\cap\{g,h\}=\emptyset$.   
\begin{rem}
By the very definition, the standard $A$-length function is asymmetric in general: the order of $g$ and $h$ matters. We allow $A$ to be empty (i.e. $m=0$), in which case we omit $A$ and recover the usual one $\LS$ (which is then symmetric).   
\end{rem}

The main proposition of this section reads as follows.
\begin{prop}\label{prop:word_geod_on_proj_cplx}
    Suppose that each element $f\in F$ is $(\kappa, N)$-divergent. Then for any $\delta>0$ there exists a constant $E = E(\kappa, \beta, C, \delta)>0$ such that the following holds:
    
    Let $g,h\in G$ be two elements and let $A$ be an anchored set of elements on a geodesic from $g$ to $h$. Then \[\LS_A(g\po, h\po)\; \le\; \LS(g\po, h\po)+ \delta \, d_S(g,h)  + E\, |A|.\]
\end{prop}
\iffalse\begin{rem}
    If $G$ is a hyperbolic group acting geometrically on $X$, then the constant $E$ is bounded for each $\delta>0$. In general, $E\to \infty$ as $\delta\to 0$.

    In the proposition, $E$ is independent of $N$, and the projection complex constant $K$ depends on $N$. 
\end{rem}\fi

The remainder of this section is devoted to the proof of Proposition \ref{prop:word_geod_on_proj_cplx}.

\subsection{Preparatory results}
In the proof, we shall prove the inequality by induction   via  decomposing the geodesic into shorter segments. 
\begin{defn}[Short decomposition]
    Let $\gamma$ be a  path from $g$ to $h$ in $\mathrm{Cay}(G,S)$. A \emph{short decomposition} of $\gamma$ with length $m$ means a sequence of ordered points  on $\gamma$  \[
    \{g = c_0< c_1< \cdots< c_m< c_{m+1} = h\}
    \] which satisfy $\LS(c_i\po, c_{i+1}\po)\le 13$ for each $0\le i\le m$. 

    Let $A$ be an   anchored set of vertices on $\gamma$ from $g$ to $h$. We say that the short decomposition is $A$-proper if
$$
([c_j,c_{j+1}]_\gamma\cap A)\setminus\{c_j,c_{j+1}\}
\ne\emptyset
$$
for at least two values of $j$.  %$A$ is not entirely contained in the subpath $[c_i, c_{i+1}]_\gamma$ for each $0\le i\le m$. That is, $$|[c_i, c_{i+1}]_\gamma \cap A|< |A|.$$
\end{defn}

Short decomposition always exists thanks to the bottleneck property in Lemma \ref{lem:bottleneck_in_PC}.% of the projection complex.
\begin{lem}\label{lem:split_path_into_pieces}
    Let $\gamma$ be a path from $g$ to  $h$ in $\mathrm{Cay}(G,S)$. Then $\gamma$ admits a {short decomposition} with length  $m\le \lfloor\LS(g\po,h\po)/5\rfloor$ \[
    \{g = c_0< c_1< \cdots< c_m< c_{m+1} = h\}
    \] and there exists a sequence of ordered points on $\f_K[g\po, h\po]$ \[\{g\po = Z_0<Z_1<\cdots < Z_m<Z_{m+1} = h\po\}\] so that $\LS(Z_i, c_i\po) \le 3$ for each $1\le i\le m$.
\end{lem}

\begin{proof}
    Denote $n = \LS(g\po,h\po)$ and $\f_K(g\po, h\po) = \{Y_1<Y_2< \cdots <Y_{n-1}\}$. Write $Y_0 = g\po$ and $Y_n = h\po$ by convention. By Lemma~\ref{lem:word_is_PC_path}, $\gamma$ labels a path  in $\PC$ from $g\po$ to $h\po$, which passes through the $3$-neighborhood of every vertex on the standard path $\f_K(g\po, h\po)$ by the bottleneck property in Lemma~\ref{lem:bottleneck_in_PC}. 
    Set $m =\max\{0, \lfloor n/5\rfloor-1\}$. If $10\le n\le 14$; then $m=1$ and choosing $Z_1=Y_5\in \f_K(g\po, h\po)$ and  $c_1\in \gamma$ with $\LS(Z_1, c_1\po) \le 3$ satisfies the conclusion. If $n\le 9$ there is nothing to prove. 
    
    We now assume $m\ge 2$, and pick $Z_i = Y_{5i}$ for $0\le i\le m$, and set $Z_{m+1} = Y_n$. For each $1\le i\le m$, choose $c_i\in \gamma$ to be the first point 
    (in the natural order) with $\LS(c_i\po, Y_{5i})\le 3$. Set $c_0 = g$ and $c_{m+1} = h$. For $1\le i\le m$, $\LS(Y_{5(i-1)}, Y_{5i}) = 5$, and $5 \le \LS(Y_{5m}, h\po) \le 9$.
    By Lemma~\ref{lem:std_tri_ineq}, the length function $\LS$ satisfies  the weak triangle inequality, so for $1\le i\le m$,
    \[
    \begin{aligned}
        \LS(c_{i-1}\po, c_i\po) 
        &\le \LS(c_{i-1}\po, Y_{5(i-1)}) 
           + \LS(Y_{5(i-1)}, Y_{5i}) 
           + \LS(Y_{5i}, c_i\po) + 2\le  13\\
        \LS(c_m\po, h\po) 
        &\le \LS(c_m\po, Y_{5m}) + \LS(Y_{5m}, h\po) + 1  \le13 .
    \end{aligned}
    \]
    It remains to check that  $\{c_0, c_1, \cdots, c_{m+1}\}$ are ordered in $\gamma$. Indeed, by Corollary~\ref{cor:sub_segment_std}, $\f_K(g\po, Y_{5i}) = \{Y_1,\dots,Y_{5i-1}\}$. Since $\LS(c_i\po, Y_{5i})\le 3$, Lemma~\ref{lem:common_prefix} gives
    \[
    \{Y_1,\dots,Y_{5i-5}\} \subseteq \f_K(g\po, c_i\po).
    \]
    In particular $Y_{5(i-1)}\in \f_K(g\po, c_i\po)$. By definition, $c_{i-1}\in \gamma$ is the first point with $\LS(c_{i-1}\po, Y_{5i-5})\le 3$, and the bottleneck property for $\f_K(g\po, c_i\po)$ proves that  $c_{i-1}$ must appear before $c_i$. Hence, up to removing repetitions $\{c_0, c_1, \cdots, c_{m+1}\}$ forms a short decomposition. 
\end{proof}

The induction will not only be on the length of the geodesic, but also on the cardinality of the anchored set on it.
The next lemma  breaks  the anchored set into smaller ones. 
For $x\in \PC$ and $r\ge 0$, denote 
$$B^{\mathrm{std}}(x,r) = \big\{y\in \PC: \LS(x, y)\le r\big\}$$
the ball-like set  in $\PC$   with respect to the standard path length.

\begin{lem}\label{lem:leave_ball_twice}
Assume  $\LS(g\po, h\po) \le 13$ for some $g, h \in G$. Let $A$ be an anchored set of vertices on some   path $\gamma$  from $g$ to $h$ in $\mathrm{Cay}(G,S)$. Suppose there exist distinct $a, b \in A$ such that 
\[
\LS(a\po, g\po) \ge 200, \quad \LS(b\po, g\po) \ge 200, \quad \text{and} \quad \f_K[a\po, b\po] \cap B^{\mathrm{std}}(g\po,20)\neq \emptyset.
\]
Then $\gamma$ admits an $A$-proper {short decomposition} with length at most $12$. 
\end{lem}
\begin{proof}
Up to exchanging $a$ and $b$, we may assume that $g, a, b, h$ are linearly ordered along $\gamma$. By assumption that $\f_K[a\po, b\po] \cap B^{\mathrm{std}}(g\po,20)$, there exists $Y\in \f_K(a\po, b\po)$ such that $\LS(Y, g\po)\le 20$. So by Lemma~\ref{lem:bottleneck_in_PC}, there exists $p\in [a, b]_{\gamma}$ such that $\LS(p\po, Y)\le 3$. The weak triangle inequality (Lemma~\ref{lem:std_tri_ineq}) for $\LS$ implies
\[
\LS(p\po, g\po) \le \LS(p\po, Y) + \LS(Y, g\po) + 1 \le 24
\]
\[
\LS(p\po, h\po) \le \LS(p\po, g\po) + \LS(g\po, h\po) + 1 \le 38.
\]
By Lemma~\ref{lem:split_path_into_pieces}, there exists a {short decomposition} on $[g, p]_\gamma$ with length $l$:  
\[
l \le \left\lfloor \frac{\LS(g\po, p\po)}{5} \right\rfloor \le 4.
\]
Similarly, $[p, h]_\gamma$ admits a {short decomposition} with length $m$:  
\[
m \le \frac{\LS(p\po, h\po)}{5} \le 7.
\]
Since $g, p, h$ are linearly ordered on $\gamma$, concatenating the above decompositions yields the desired {short decomposition} on $\gamma$ of length $k := l + m + 1\le 12$:
\[
\{c_0 = g< c_1< \cdots< c_k< c_{k+1} = h\}
\]
where $c_{l+1} = p$.

It remains to show that the decomposition is $A$-proper. Indeed, by definition of short decomposition, $\LS(c_{i}, c_{i+1})\le 13$ for each $0\le i\le k$. By the weak triangle inequality (Lemma~\ref{lem:std_tri_ineq}),  
$$\LS(g\po, c_i\po)\le \sum_{j=0}^{i-1}\LS(c_{j}\po, c_{j+1}\po) + (i-1)\le 14i-1\le 168$$
for each $1\le i\le k$. By the assumption that $\LS(a\po, g\po), \LS(b\po, g\po)\ge 200$, we see $a, b\not\in\{c_0, \cdots, c_{k+1}\}$. Since  $p=c_{l+1}$ is contained in the decomposition and $a\in [g,p]_\gamma, b\in [p,h]_\gamma$, we conclude that $a$ and $b$ must be contained in two distinct $([c_j, c_{j+1}]_\gamma \cap A)\setminus \{c_j, c_{j+1}\}\ne \emptyset$. This justifies the $A$-properness of the decomposition.
\end{proof} 

To formulate the next lemma  conveniently, it would be helpful to introduce the term of guard. %A similar use of guard appears in \cite{BBF}, and ours is certainly indebted to theirs.

\begin{defn}[$K$-guard]\label{defn:Kguard}
Let  $A, B\subseteq G$ be two non-empty finite sets of elements. We say that $Y\in \PC$ is a \textit{$K$--guard} between $A$ and $B$ if for each pair $(a,b)\in (A\times A)\cup (B\times B)$,
    $$\begin{aligned}
    d_Y(a\po, b \po)\le \theta
    \end{aligned}$$  
    and for each pair $(a,b)\in A\times B$,
    $$d_Y(a \po, b\po) > K.  $$
It follows by definition that $A$ and $B$ must be disjoint.
\end{defn}

The following  lemma deals with the bottom case of the induction.  This is the place to require essentially that each element $f\in \f$ is $(\kappa, N)$-divergent and $\gamma$ is a   \emph{quasi-geodesic} rather than just an arbitrary path in Lemma \ref{lem:leave_ball_twice}.  

\begin{lem}\label{lem:cut_end_div}
For any $M,\lambda > 0$, there exists a constant $L = L(M, \lambda) > 0$ with the following property.

Let $\gamma = [g,h]$ be a $\lambda$-quasi-geodesic path in $\mathrm{Cay}(G,S)$, and let $A \subseteq \gamma$ be an ordered anchored set of vertices from $g$ to $h$. Assume that $Y \in \PC$ is a $K$-guard between $\{g,h\}$ and $A$.  
If $\len(\gamma) > L$, then there exists a subpath $\gamma_1 = [g_1, h_1]_\gamma$ with the  endpoints $g_1,h_1$ on $\gamma$ such that the following holds:
\begin{enumerate}
    \item $\gamma_1$ contains  the anchored set $A$;
    \item \(
    \len(\gamma_1) \le \len(\gamma) - M;
    \)
    \item $Y\notin \f_K[g\po, g_1\po]$ and $Y\notin \f_K[h_1\po, h\po]$.
\end{enumerate}

\end{lem}

\begin{proof}
For concreteness, let $H \subseteq G$ denote the left coset of $E(f)$ for some $f \in F$, which gives  $Y = Ho$. By assumption, each $f \in F$ is $(\kappa, N)$-divergent. Recall that the orbit map $\Phi:g \mapsto g \cdot o$ restricts to a $\beta$-quasi-isometric embedding on $E(f)$ for some $\beta>1$. 

We begin by fixing the constants. Given an integer $M>0$, let $R > 0$ be the least integer such that $$\kappa(R) > M$$ and then set
\begin{align}\label{eq:LofValue}
L = \lambda(2(R + M) + \beta\bigl(2\beta(R+M) + 4C + 3\theta + 1\bigr))+\lambda.    
\end{align}

Choose two elements $g_1, h_1$ on $ \gamma=[g,h]$ such that 
\[
\len([g, g_1]_{\gamma}) = M, \qquad \len([h_1, h]_{\gamma}) = M.
\]
The desired subpath $\gamma_1$ shall be obtained by cutting out either $[g, g_1]_{\gamma}$ or $[h_1, h]_{\gamma}$.  This is the content of the case (1), while the   case (2) shall be proven impossible.

\medskip
{\textbf{Case 1.}} Suppose either $[g, g_1]_\gamma$ or $[h_1, h]_\gamma$  is disjoint with the $R$-neighborhood of $H$ in word metric. Assume $[g, g_1]_\gamma \cap N_R(H) =\emptyset$ for concreteness; the other case is symmetric. We shall prove that $\gamma_1 = [g_1, h]_\gamma$ is the desired subpath.

For any given $p \in [g, g_1]_\gamma$, $\len([g, p]_{\gamma}) \le M < \kappa(R)$. Since $Y=Ho$ is $(\kappa, N)$-divergent (Definition \ref{defn:divergence}), this implies $\proj_Y^\pi(go, po)< N$. %\ywy{REMOVE 
%\[
%\proj_Y^S(g,p) < N
%\]
%whence by the $\beta$-Lipschitz of orbital map, $$\proj_Y^\pi(go, po)\, \le\, \beta\,\proj_Y^S(g,p)\le\, N\beta.$$} 
By Eq. (\ref{eq:projdist}) in Theorem~\ref{thm:forcing_seq},
\[
\proj_Y(g\po, p\po) \le \proj_Y^\pi(g\po,p\po)+2\theta\le \proj_Y^\pi(go, po) + 4\theta \le N + 4\theta < K
\]
where the last inequality follows by the choice of $K$ in (\ref{eq:defnK}).
This holds for every $p\in  [g, g_1]_\gamma$, so justifies the item (3): $Y\notin \f_K[g\po, g_1\po]$. 

Noting that $Y$ is a $K$-guard between $\{g,h\}$ and $A$, by Definition \ref{defn:Kguard} we obtain $\proj_Y(a\po, g\po) > K$ for every $a \in A$. This implies $p \notin A$ for any $p \in [g, g_1]_\gamma$. Hence,  $A \subseteq [g_1, h]_\gamma$. Taking the subpath $\gamma_1 = [g_1, h]_\gamma$ completes the proof in this case.

%{\textbf{Case 2.}} Suppose $d_S([h_1, h]_\gamma, H) > R$.

%By the same argument as in Case 1, we obtain $A \subseteq [g, h_1]_\gamma$. Taking $\gamma_1 = [g, h_1]_\gamma$ completes the proof.
\medskip
{\textbf{Case 2.}} Let us assume 
\[
d_S([g, g_1]_\gamma, H) \le R \quad \text{and} \quad d_S([h_1, h]_\gamma, H) \le R.
\]
The remaining proof is to derive a contradiction by estimating the length of $\gamma$. 
\medskip

First, the triangle inequality gives
\[
d_S(g, H) \le \mathrm{diam}_S([g, g_1]_\gamma) + d_S([g, g_1]_\gamma, H) \le M + R,
\]
\[
d_S(h, H) \le \mathrm{diam}_S([h_1, h]_\gamma) + d_S([h_1, h]_\gamma, H) \le M + R.
\]
Choose $g', h' \in H$ such that $d_S(g, g') \le R + M$ and $d_S(h, h') \le R + M$. The $\beta$-Lipschitz property of the orbital map $\Pi: g\in G \mapsto go\in X$ then gives
\[
d(go, g'o) \le \beta\, d_S(g, g') \le \beta(R + M).
\]
By Lemma~\ref{BigFive}(2), the shortest projection map to the strongly contracting subset $Y$ is coarsely Lipschitz, so
\[
\proj_Y^\pi(go, g'o) \le d(go, g'o) + 2C \le \beta(R+M) + 2C.
\]
We argue similarly as above and get
\[
\proj_Y^\pi(ho, h'o) \le \beta(R+M) + 2C.
\]

On the other hand, since $Y$ is a $K$-guard between $\{g,h\}$ and $A$, we have $\proj_Y(g\po, h\po)\le \theta$, so by Eq. (\ref{eq:projdist}),
\[
\proj_Y^\pi(go, ho) \le \proj_Y(g\po, h\po) + 2\theta \le 3\theta.
\]
Using the triangle inequality for $\proj_Y^\pi$,
\[
\begin{aligned}
\proj_Y^\pi(g'o, h'o)
&\le \proj_Y^\pi(g'o, go) + \proj_Y^\pi(go, ho) + \proj_Y^\pi(ho, h'o)\\
&\le 2\beta(R+M) + 4C + 3\theta.
\end{aligned}
\]
The restriction of the map $\Pi: H \to Y$ is a $\beta$-quasi-isometry, implying
\[
d_S(g', h') \le \beta\, d(g'o, h'o) + \beta \le \beta\bigl(2\beta(R+M) + 4C + 3\theta + 1\bigr).
\]

At last,  let us estimate   
\[
\begin{aligned}
d_S(g, h)
&\le d_S(g, g') + d_S(g', h') + d_S(h', h)\\
&\le 2(R + M) + \beta\bigl(2\beta(R+M) + 4C + 3\theta + 1\bigr)\\
Eq.\, (\ref{eq:LofValue})\; &= (L-\lambda)/\lambda.
\end{aligned}
\]
Since $\gamma$ is a $\lambda$-quasi-geodesic from $g$ to $h$, we see $\len(\gamma)\le L$.  
This contradicts the assumption $\len(\gamma) > L$, so the case (2) is  impossible. The lemma is proved. 
\end{proof}

\subsection{Proof of Proposition~\ref{prop:word_geod_on_proj_cplx}}
We first prove a special case, under the assumption $\LS(g\po, h\po) \le 13$, which   serves as an inductive step to the full case.

The following elementary lemma will be used in the proof.
\begin{lem}\label{lem:inductanchoredlength}
Assume that a sequence of points $$\{c_0=g < c_1 < \cdots < c_k < c_{k+1}=h\}$$ is ordered along a geodesic $\gamma=[g,h]$ with $k\ge 1$. 
    Let $A$ be an  anchored set on $\gamma$ from $g$ to $h$. Denote $A_i=([c_i,c_{i+1}]_\gamma\setminus\{c_i, c_{i+1}\})\cap A$ for $0\le i\le k$. Denote $m=|\{c_1, \cdots, c_k\}\setminus A|$. Then
    \[
    \LS_A(g\po, h\po) \le m + \sum_{i=0}^k \LS_{A_i}(c_i\po, c_{i+1}\po).
    \]   
\end{lem}
\begin{proof}
Write $A=\{a_1<\cdots< a_n\}$ and $a_0=g, a_{n+1}=h$. Then $\LS_A(g\po, h\po) = \sum_{i=0}^n \LS(a_i\po, a_{i+1}\po)$. If $(a_i,a_{i+1})_\gamma$ contains $c_j$, then $\LS(a_i\po, a_{i+1}\po)\le \LS(a_i\po, c_j\po)+\LS(c_j\po, a_{i+1}\po)+1$. By induction,  the conclusion follows by applying  the weak triangle inequality for $\LS$ at most $m$ times.      
\end{proof}

\begin{lem}\label{lem:barrier_free_on_proj_cplx}
    Suppose that each element $f\in F$ is $(\kappa, N)$-divergent. Then for any $\delta>0$ there exists a constant $E = E(\delta)>0$ such that the following holds:
    
    Let $g,h\in G$ be two elements and let $A$ be a non-empty geodesically anchored set from $g$ to $h$. Assume that $\LS(g\po, h\po) \le 13$. Then 
    \begin{align}\label{eq:keyinequalityinproof}
    \LS_A(g\po, h\po) \le \delta \, d_S(g,h)  + E\, |A| - 200.
    \end{align}

\end{lem}

\begin{proof}
    For given $\delta>0$, let $$L = L\left(\frac{62}{\delta},1\right)>0$$ be given by Lemma \ref{lem:cut_end_div}. Take $E = 4L + 1500$.

    Let $\gamma$ be the word geodesic from $g$ to $h$ that contains $A$.  Note that $d_S(g,h)=\len(\gamma)$ and $|A|\le \len(\gamma)$.
    By Lemma \ref{lem:word_is_PC_path}, $\gamma$ labels a path from $g\po$ to $h\po$, and the standard paths are $2$-quasi-geodesic. We thus note  $$\LS_A(g\po, h\po)  \le 2\len(\gamma)+2|A|.$$ Our goal Eq. (\ref{eq:keyinequalityinproof})  improves significantly this inequality.
    
    To that end, let us define  the \textit{complexity} of the pair $(\gamma,A)$ by $$\Omega(\gamma,A)\; :=\;d_S(g,h)+|A|$$ 
    We shall induct on $\Omega(\gamma,A)$ to prove the desired inequality (\ref{eq:keyinequalityinproof}).
 
\medskip    

\noindent{\textbf{Base step.}} Let $\gamma$ be a geodesic  with  an anchored set $A$ so that $\LS(g\po, h\po) \le 13$.  If $\len(\gamma) + |A| \le 2L$, then the choice of the constant $E$ shows $$
\begin{aligned}
\LS_A(g\po, h\po) & \le 2\len(\gamma)+2|A|\\
&\le 4L\le E - 200 \\
& \le \delta\; \len(\gamma) + E\;|A| - 200.    
\end{aligned}$$
Thus, Eq. (\ref{eq:keyinequalityinproof}) is verified under the complexity $\Omega(\gamma,A)  \le 2L$.

\medskip
\noindent{\textbf{Induction Step.}} Let $\Omega_0> 2L$ be an integer. Suppose the inequality (\ref{eq:keyinequalityinproof})  holds  for any pair $(\gamma,A)$ with $\Omega(\gamma,A)< \Omega_0$.  We need to prove (\ref{eq:keyinequalityinproof}) for any geodesic $\gamma$ with  an anchored set $A$ satisfying  $\Omega(\gamma,A)= \Omega_0$. Since  $\len(\gamma) + |A| \ge 2L$ and $|A|\le \len(\gamma)$, we obtain $\len(\gamma) \ge L$.
    
    \medskip

    We divide the induction into the following three cases.
    
    \medskip
    \noindent{\textbf{Case 1.}} Suppose there exists $a\in A$ such that $\LS(a\po, g\po)\le 200$.
     \medskip
    
    By the weak triangle inequality for $\LS$ (Lemma~\ref{lem:std_tri_ineq}), $$\LS(a\po, h\po)\le \LS(a\po, g\po) + \LS(g\po, h\po) + 1\le 214.$$ 
    
    By Lemma~\ref{lem:split_path_into_pieces}, $[g, a]_\gamma$ and $[a,h]_\gamma$ admit a {short decomposition} of length less than $\lfloor\LS(g\po, a\po)/5\rfloor\le 40$ and $\lfloor\LS(a\po, h\po)/5\rfloor\le 42$ respectively. Concentrating these yields a {short decomposition} of $\gamma$ with length $2\le k\le 83$, denoted as $$\{g = c_0< c_1<\cdots< c_k< c_{k+1} = h\}$$ 
    Note that $\LS(c_i\po,c_{i+1}\po)\le 13 $ for $0\le i\le k$.
    
    Denote $\gamma_i=[c_i, c_{i+1}]_\gamma$ and $A_i = (A\cap \gamma_i)\setminus \{c_i, c_{i+1}\}$ for $0\le i\le k$. By Lemma \ref{lem:inductanchoredlength}, 
    \begin{align}\label{eq:case1totalinequality}
    \LS_A(g\po, h\po)&\le 83 + \sum_{i=0}^k \LS_{A_i}(c_i\po, c_{i+1}\po).
    \end{align}
    By the construction, $a\not\in A_i$, so both $|A_i|$ and $\len(\gamma_i)$ decrease. Hence,  $\Omega(\gamma_i,A_i)<\Omega(\gamma,A)$ for each $0\le i\le k$. By the induction hypothesis, if $A_i\neq \emptyset$,
    \[
    \LS_{A_i}(c_i\po, c_{i+1}\po)\le \delta\; d_S(c_i, c_{i+1}) + E\; |A_i| - 200.
    \]
    If $A_i=\emptyset$, then $\LS_{A_i}(c_i\po, c_{i+1}\po)\le 13$. Thus, in both cases, 
    \begin{align}\label{eq:case1inequality}
        \LS_{A_i}(c_i\po, c_{i+1}\po)\le \delta\; d_S(c_i, c_{i+1}) + E\; |A_i| +13.
    \end{align}
    
    Note that $d_S(g,h)=\sum_{0\le i\le k} d_S(c_i, c_{i+1})$ and $\sum_{i=0}^k|A_i|\le |A|-1$. Summing up the inequalities (\ref{eq:case1inequality}) and plugging them into (\ref{eq:case1totalinequality}), with $E>1500$,
    \[\begin{aligned}
        \LS_{A}(g\po, h\po)&\le \delta d_S(g, h) + E(|A|-1) + 1175\\
        &\le \delta d_S(g, h) + E|A| - 200.
    \end{aligned}
    \]
    The inequality  (\ref{eq:keyinequalityinproof}) follows in this case.
     
     \medskip
     
    \noindent{\textbf{Case 2.}} Suppose that $\LS(a\po, g\po)\ge 200$ for each $a\in A$, but there exist two distinct  $a\ne b\in A$ so that $\mathbf \f_K(a\po,b\po)\cap B^{\mathrm{std}}(g\po, 20)\ne\emptyset$. 
     \medskip
     
    This is exactly the assumption of  Lemma \ref{lem:leave_ball_twice}. Thus, $\gamma$ admits a $A$-proper short decomposition with length $2\le k \le 12$ so that $a,b$ are not in the same path component. Namely, there exists a sequence of linearly ordered of elements $\{c_0 = g< c_1< \cdots< c_k< c_{k+1} = h\}$  on $\gamma$  so that for each $0\le i\le k$,  
    \begin{itemize}
        \item  $\LS(c_i\po, c_{i+1}\po)\le 13$,
        \item 
        $|A_i| < |A|$ with $\gamma_i := [c_{i}, c_{i+1}]_\gamma$ and $A_i = (A\cap\gamma_i)\setminus\{c_i, c_{i+1}\}$. 
        \item $A_i\neq \emptyset$ for least two $i$'s.
    \end{itemize}   
    Thus, $$\len(\gamma_i) + |A_i| < \Omega(\gamma, A)=\Omega_0.$$ Notice that here $A_i$ may be an empty set. 
 
    \begin{itemize}
        \item If $A_i = \emptyset$, by Definition \ref{def:anchored_set} of $\LS_A$, 
        \begin{align}\label{eq:AiEmptyset}
        \LS_{A_i}(c_i\po, c_{i+1}\po) = \LS(c_i\po, c_{i+1}\po)\le 13.    
        \end{align} 
        \item 
        If $A_i\neq \emptyset$, the induction hypothesis for $(\gamma_i,A_i)$  shows 
        \begin{align}\label{eq:AiNonemptyset}
        \LS_{A_i}(c_i\po,c_{i+1}\po)\le  \delta\; \len(\gamma_i) + E\; |A_i| - 200.\end{align} 
    \end{itemize}

    Let $I_1 = \{i: 0\le i\le k, A_i\neq \emptyset\}$. By Lemma \ref{lem:leave_ball_twice}, $|I_1|\ge 2$.  Let $I_2 = \{i: 0\le i\le k, A_i = \emptyset\}$. Then $|I_2|\le (k+1)-2\le 11$. By Lemma~\ref{lem:inductanchoredlength} and the weak triangle inequality for $\LS$, 
    \[\begin{aligned}
        \LS_A(g\po,h\po)&\le k+\sum_{i=0}^k  \LS_{A_i}(c_i\po, c_{i+1}\po) \\
     Eq. (\ref{eq:AiEmptyset}),\,Eq. (\ref{eq:AiNonemptyset}) \quad  &\le \sum_{i\in I_1} \left(\delta\; \len(\gamma_i) + E\;|A_i| - 200\right) + 13\;|I_2|+12\\
        &\le \delta\; \len(\gamma) + E\; |A| - 200\;|I_1| + 13\;|I_2|+12\\
        &\le \delta\; \len(\gamma) + E\; |A| - 200.
    \end{aligned}\]

    \medskip
    
    \noindent{\textbf{Case 3.}} Suppose that $\LS(a\po, g\po)\ge 200$ for any $a\in A$, and $\f_K[a\po, b\po] \cap B^\mathrm{std}(g\po,20) = \emptyset$  for any $a, b\in A$. 
    \medskip
     
    We first claim that
    \begin{claim}\label{claim:guardYexists}
    There exists a $K$-guard $Y\in \f$  between $\{g,h\}$ and $A$ so that $\LS(g\po,Y)=16$.    
    \end{claim}    
    \begin{proof}[Proof of the claim]
    
    Fix a reference $a\in A$. Let $b\in A$ be an arbitrary element. Then by assumption $\f_K[a\po, b\po]\cap B^{\mathrm{std}}(g\po, 20) = \emptyset$. By Lemma \ref{lem:more_guard_out_ball},  $\f_K(g\po, b\po)$ agrees with $\f_K(g\po,a\po)$ on the smallest 18 elements. Furthermore,  if we apply Lemma \ref{lem:common_prefix} to $\f_K(g\po, b\po)$ and $\f_K(h\po, b\po)$ with $\LS(g\po,h\po)\le 13$,  the 15th, 16th and 17th smallest elements denoted as $Y_0$, $Y$ and $Y_1$ in $\f_K(g\po, a\po)$ also belong to $\f_K(h\po, b\po)$. In particular, $\LS(g\po,Y)=16$. 
    
    In summary, we found three candidates $Y_0, Y, Y_1\in \f_K(g\po, b\po)\cap \f_K(h\po, b\po)$ for any $b\in A$. It remains to show that $d_Y(g\po, h\po)<\theta$ and $d_Y(b_1\po, b_2\po)<\theta$ for any $b_1, b_2\in A$. 

    Indeed, let us fix any $b\in A$. Apply Proposition~\ref{prop:std_path} to $Y_0<Y$ in $\f_K(g\po, b\po)$, we have $d_{Y_0}(g\po, Y)>K$. Again apply Proposition~\ref{prop:std_path} to $Y_0<Y$ in $\f_K(h\po, b\po)$, we have $d_{Y}(h\po, Y_0)\le \theta$. Apply \ref{axiom:strong_berstock} to $d_{Y_0}(g\po, Y)>K$, we obtain $$d_Y(g\po, h\po) = d_Y(Y_0, h\po)\le \theta.$$

    Fix any $b_1, b_2\in A$. Similarly, we get $d_{Y_1}(b_1\po, Y)>K$ and $d_{Y}(Y_1, b_2\po)\le \theta$. Apply \ref{axiom:strong_berstock} to $d_{Y_1}(b_1\po, Y)>K$ we obtain $$d_Y(b_1\po, b_2\po) = d_Y(Y_1, b_2\po)\le \theta.$$ 
    Hence, $Y$ is a $K$-guard between $\{g, h\}$ and $A$, and the claim follows.
    \end{proof}
\iffalse
    \begin{proof}[Proof of the claim]
    
    Fix a reference $a\in A$. By assumption, $a\not\in B^{\mathrm{std}}(g\po, 15)$, so the standard path $\f_K[a\po, g\po]$ has length at least 15. Let $b\in A$ be an arbitrary element, for which $\f_K[a\po, b\po]\cap B^{\mathrm{std}}(g\po, 15) = \emptyset$. By Lemma \ref{lem:more_guard_out_ball},  $\f_K(g\po, b\po)$ agrees with $\f_K(g\po,a\po)$ on the smallest 12 elements. Further,  if we apply Lemma \ref{lem:common_prefix} to $\f_K(g\po, b\po)$ and $\f_K(h\po, b\po)$ with $\LS(g\po,h\po)\le 10$,  the 12th  element denoted as  $Y$ in $\f_K(g\po, b\po)$ belongs to $\f_K(h\po, b\po)$. In particular, $\LS(g\po,Y)=12$. 
    
    In summary, we found a candidate  $Y\in \f_K(g\po, b\po)\cap \f_K(h\po, b\po)$ for any $b\in A$. This implies $d_Y(g\po, b\po)\ge K$ and  $d_Y(h\po, b\po)\ge K$ for any $b\in A$. It thus remains to show that $d_Y(g\po, h\po)<K$ and $d_Y(a\po, b\po)<K$ for any $a\ne b\in A$. Equivalently, $Y\notin \f_K(g\po, h\po)$ and $Y\notin \f_K(a\po, b\po)$.
    
    Since $\LS(g\po, h\po)\le 10$ and $\LS(g\po, Y)=12$, we have $Y\not\in \f_K(g\po, h\po)$. Also, for any distinct $a, b\in A$, since $\f_K[a\po, b\po]\cap B^{std}(g\po, 15) = \emptyset$ and $\LS(g\po, Y) = 12$, $Y\not\in \f_K(a\po, b\po)$. Hence, $Y$ is a $K$-guard between $\{g, h\}$ and $A$, and the claim follows.
    \end{proof}
\fi
    
    Now, we apply Lemma \ref{lem:cut_end_div} to the pair $(\gamma,A)$ with $\len(\gamma)\ge L$.  There exists a subpath $\gamma_1\subseteq \gamma$ with endpoints $g_1, h_1\in \gamma$ such that $$\len(\gamma_1) \le \len(\gamma) - 62\delta^{-1}$$ and $A$ is entirely contained in $\gamma_1$.

   \medskip 
   Let $a$ be the first point of $A$ (which must be on $\gamma_1$, since $A\subseteq \gamma_1$). By Lemma~\ref{lem:triangle_in_PC}, the tripod-like property of the triangle $(g\po, a\po, g_1\po)$ provides a decomposition $$\f_K(g\po, a\po) = \alpha_1\sqcup\alpha_2\sqcup\alpha_3$$ where $\alpha_1\subseteq \f_K(g\po,g_1\po)$, $|\alpha_2|\le 2$ and $\alpha_3\subseteq \f_K(g_1\po, a\po)$. 
   
   \medskip 
   We apply  Lemma~\ref{lem:cut_end_div}(3) to the $K$-guard $Y$ given by Claim \ref{claim:guardYexists}, so $Y\notin \f_K(g\po, g_1\po)$. However, by the definition of $K$-guard $Y\in \f_K(g\po, a\po)$, so either $Y\in \alpha_2$ or $Y\in \alpha_3$. In both cases, there exists $Z\in \f_{K}[g_1\po, a\po]$ such that $\LS(Z, Y)\le 2$. By the bottleneck property (Lemma~\ref{lem:bottleneck_in_PC}), there exists $g_2\in [g_1, a]_\gamma$ so that $\LS(g_2\po, Z)\le 3$. Hence by the weak triangle inequality (Lemma~\ref{lem:std_tri_ineq}), $$\LS(g_2\po, Y)\le \LS(g_2\po, Z) + \LS(Z, Y) + 1\le 6.$$ Note that $\LS(g\po, a\po)\ge 200$ and $\LS(g\po, Y)=16$ by the claim above. The weak triangle inequality (Lemma~\ref{lem:std_tri_ineq}) gives
    $$\LS(g\po, g_2\po)\le \LS(g\po, Y) + \LS(Y, g_2\po) + 1\le 23.$$ So $a\neq g_2$.  %and Lemma~\ref{lem:std_tri_ineq},  such that $$\LS(g_2\po, Y)\le \LS(Y, \f_K[g_1\po, a\po]) + 2 + 1\le 5.$$
    
    \medskip 
   Similarly, if $b$ is the last point in $A$, then there exists $h_2\in [b, h_1]_\gamma$ such that $\LS(h_2\po, Y)\le 6$ and $h_2\neq b$. In addition, by the weak triangle inequality (Lemma~\ref{lem:std_tri_ineq}), \[
    \begin{aligned}
    %\LS(g\po, g_2\po)&\le \LS(g\po, Y) + \LS(Y, g_2\po) + 1\le 12+4+1=17,  \\
    \LS(h_2\po, h\po)\le \LS(g\po, h\po) + \LS(g\po, Y) + \LS(Y, h_2\po) + 2\le 37.\\
    \end{aligned}
    \]
     Let  $\gamma_2=[g_2,h_2]_\gamma$. Then $A\subseteq \gamma_2\setminus\{g_2, h_2\}$, and $\Omega(\gamma_2, A) < \Omega(\gamma, A)$. Notice that
    \[\LS(g_2\po, h_2\po) \le \LS(g_2\po, Y) + \LS(Y, h_2\po) + 1\le 13.\]
    Hence, by the induction hypothesis for $(\gamma_2,A)$, $$\LS_A(g_2\po,h_2\po)\le \delta \;\len(\gamma_2) + E |A| - 200.$$ 
    Recalling $\len(\gamma_2)\le \len(\gamma_1)\le \len(\gamma)- 62\delta^{-1}$, we have
    \[\begin{aligned}
        \LS_A(g\po,h\po)&\le \LS_A(g_2\po,h_2\po) + \LS(g\po, g_2\po) + \LS(h_2\po, h\po)+2\\
        &\le \delta(\len(\gamma) - 62\delta^{-1}) + E\; |A| - 200 + 62\\
        &\le \delta\; \len(\gamma) + E\; |A| - 200.
    \end{aligned}\]
    Thus, the inequality (\ref{eq:keyinequalityinproof}) holds in the case (3). 
    
   \medskip 
   \noindent Therefore, the inductive step proves (\ref{eq:keyinequalityinproof}) in each case and the proof is completed.
\end{proof}
We are now ready to prove Proposition \ref{prop:word_geod_on_proj_cplx}.

\begin{proof}[Proof of Proposition \ref{prop:word_geod_on_proj_cplx}]
    
     By Lemma \ref{lem:split_path_into_pieces}, $\gamma=[g,h]$ admits a short decomposition $$C:=\{c_0=g<c_1<\cdots<c_k<c_{k+1}=h\}\subseteq \gamma$$ with $k\le \LS(g\po, h\po)/5$ and a sequence of points $$\{Y_0 = g\po < Y_1<\cdots < Y_k<Y_{k+1} = h\po\}\subseteq\f
     _K[g\po, h\po]$$ such that $\LS(c_i\po, Y_i)\le 3$ for each $0\le i\le k+1$. Denote $\gamma_i=[c_i,c_{i+1}]_\gamma$ and $A_i=(A\cap \gamma_i)\setminus \{c_i,c_{i+1}\}$.
    If $A_i\ne \emptyset$, by Lemma \ref{lem:barrier_free_on_proj_cplx}, 
    \[
    \LS_{A_i}(c_i\po, c_{i+1}\po) \le  \delta \, d_S(c_i,c_{i+1})  + E\, |A_i| - 200.
    \] 
     
    Let $I$ be the set of indices $0\le i\le k$ so that  $A_i\neq \emptyset$. Then $B:=A\setminus \cup_{i\in I} A_i$ is contained in $C$. Let $\widehat C$ be  the union of $\{g,h\}$, $\{c_i, c_{i+1}: i\in I\}$ and $B$. It is the subset of the above short decomposition $C$ which excludes $c_i$'s  that  are not contained in the anchor set $A$, unless $(c_i,c_{i+1})_\gamma\cap A\ne\emptyset$.  
    Hence,  the number \(m\) of interior points of \(\widehat C\)  is at most  $2|I|+|B|$. Explicitly, write $$\widehat C=\{g=c_{i_0}<c_{i_1}<c_{i_2}<\cdots<c_{i_m}<c_{m+1}=h\}$$ and we estimate the partial sum of $\LS_A(g\po, h\po)$ over $(a_i, a_{i+1})\in A$ appearing in $C$: 
     $$
     \begin{aligned}
     &\sum_{j=0}^{m}\LS(c_{i_j}\po, c_{i_{j+1}}\po) \le \sum_{j=0}^{m} (\LS(Y_{i_j}, Y_{i_{j+1}})+8) \le  8m+\LS(g\po,h\po)
     \end{aligned} 
     $$
     where the term $8$ follows by applying twice the weak triangle inequality of $\LS$.
    Note that $\widehat C$ may not be a short decomposition of $\gamma$ anymore, to which we may apply Lemma \ref{lem:inductanchoredlength}:
    \[
    \begin{aligned}
    \LS_A(g\po, h\po) & \le m +(8m+\LS(g\po,h\po))+ \sum_{i\in I}\LS_{A_{i}}(c_{i}\po, c_{i+1}\po) \\
    & \le 9m + \LS(g\po,h\po) + \sum_{i\in I}\left(\delta \; d_S(c_i,c_{i+1})  + E\; |A_i| - 200 \right) \\
    & \le \LS(g\po,h\po) +  \delta \; d_S(g,h)  + E\; |A|-E|B|- 200|I|+9m\\
    & \le \LS(g\po, h\po)+ \delta \; d_S(g,h)  + E\; |A|
    \end{aligned}
    \]
    where  $\sum_{i\in I}|A_i|=|A|-|B|$ and $m\le 2|I|+|B|$. The proof is complete.
    \end{proof} 

\section{Growth tightness of short displacements in projection complex}\label{SecShortPC}

We retain the setup of Section~\ref{sec:anchoredlength}. Thus, $(G,S)$ acts on $X$, and $F\subseteq G$ is a set consisting of three pairwise independent strongly contracting WPD elements. Let $\theta$ be the projection constant associated with $\f$, and let $\PC$ be the projection complex corresponding to a sufficiently large constant $K\gg 3\theta$, as specified in \eqref{eq:defnK}. We now introduce the additional constants needed in this section.

\medskip

\noindent{\textbf{Finite set $F$ of WPD elements}}.  
\begin{itemize}
    \item Let $\tau>0$ be the constant so that   Extension  Lemma \ref{extend3} holds for $F$;
    \item Let $L, r>0$ be given by Proposition \ref{admisProp} such that an $(L,\tau)$-admissible path relative to $\f$ has the $r$-fellow travel property;
    \item By replacing each $f\in F$ by a sufficiently large power, we assume 
    \begin{enumerate}
        \item $d(o, fo)>\max\{L, K + 2r + 4\theta\}$;
        \item Since each $f\in F$ is a WPD element, there exists $M>0$ such that $$\forall f\in F,\quad |\mathrm{Stab}_{3r}(o, fo)| < M$$ 
    \end{enumerate}
\end{itemize}
We remark that Lemma \ref{extend3} requests $F$ to contain three independent  elements.

\begin{thm}\label{thm:grow_tight_PC}
    There exists a constant $\varepsilon = \varepsilon(G, X, S, K, \kappa, N)>0$ such that the following set $$A := \bigl\{g\in G: \, \LS(\po, g\po)\le \varepsilon \, d_S(1, g)\bigr\}$$ is growth tight in $G$ with respect to the word metric $d_S$.
\end{thm}
Since the isometric action of $G$ on any metric space has Lipschitz orbital map, the non-triviality of the above statement lies in the existence of a small $\varepsilon$ satisfying the growth tightness. Note that \ref{MainThmShortElems} is reduced to Theorem \ref{thm:grow_tight_PC}.
\begin{proof}[Proof of \ref{MainThmShortElems} assuming Theorem \ref{thm:grow_tight_PC}]
Consider  the $G$-equivariant collapsing map $\Phi: \cup_{Y\in \f} Y \to \PC$, which is coarsely $L$-Lipschitz by Lemma \ref{lem:coarseLip2Projcplx} for some $L>0$. Then for any $g\in G$, $$ \LS(\po, g\po)\le Ld(o,go)+L$$
Thus, if $\LS(\po, g\po)> \varepsilon \, d_S(1, g)$ for some $\varepsilon>0$ then 
$$\begin{aligned}
d(o, go)&> \frac{\varepsilon \, d_S(1, g)-L}{L} \ge \frac{\varepsilon d_S(1,g)}{2L}    
\end{aligned}$$
when $d_S(1,g)>2L/\varepsilon$. Therefore, setting $\varepsilon_0:=\varepsilon/2L$, $\bigl\{g\in G: \, d(o, go)> \varepsilon_0 \, d_S(1, g)\bigr\}$
is exponentially generic by  Theorem \ref{thm:grow_tight_PC}.  The proof is complete.
\end{proof}

The remainder of this section is devoted to the proof of Theorem \ref{thm:grow_tight_PC}.    
Concretely, let us set  $l: = \max_{f\in F}\{d_S(1, f)\}$, and choose
    \begin{align}\label{eqn:defepsilon}
    \varepsilon = \frac{1}{4|S|^{l} M}    
    \end{align}

\subsection{Defining the insertion map}\label{subsec:insertion_map} The main tool of the proof is the map of inserting contracting elements in a word introduced in \cite{DY24}. We need several auxiliary definitions.
\begin{defn}\label{def:insert_set}
Let $0<m<n$. An ordered subset
$I=\{i_1<\cdots<i_m\}\subseteq\{1,\ldots,n-1\}$
is called an \emph{anchored set}. We denote
\[
P_m:=\{I\subseteq\{1,\ldots,n-1\}:|I|=m\}.
\]
\end{defn}
\begin{defn}\label{def:anchored_decomp}
    Let $g = s_1 \cdots s_n\in G$ be a word representation with $s_i\in S$ ($1\le i\le n$). Let $I = \{i_1<i_2<\cdots< i_m\}\in P_m$ and $i_0=0, i_{m+1} = n$. The \emph{anchored decomposition} of $g$ relative to  $I$ is defined as the product decomposition $g = h_0 h_1 \cdots h_m$, where $h_j = s_{i_j + 1} \cdots s_{i_{j+1}}.$
    %The set $I$ will be refered to as anchored set.
\end{defn}

Recall $B_n:= \{g\in G: d_S(1,g) \le n\}$ and $S_n:= \{g\in G: d_S(1,g) = n\}$. 
\begin{defn}[Insertion map]\label{def:insert_map}
    Fix $n>m>0$. Given the data
    \begin{enumerate}
        \item For any $g\in S_n$, fix a word geodesic representation $g = s_1 s_2 \cdots s_n$. 
    
        \item 
        For any $I\in P_m$, let $g = h_0h_1 \cdots h_m$ be the anchored decomposition of $g$ with anchored set $I$.
    \end{enumerate}
    we define the map $\Phi: S_n\times P_m\to G$ by
    $$\Phi(g, I)\; :=\; h_0 \prod_{i=1}^m f_i h_i = h_0 f_1 h_1 f_2 h_2 \cdots f_m h_m$$ 
    where for each $1\le i\le m$, $f_i\in F$ are chosen by Lemma~\ref{extend3} (depending on $h_{i-1}$ and $h_i$) subject to   the condition:
    \begin{equation}
    \tag{$\dagger$}\label{eq:tripleadmissible}
    \begin{array}{lr}
    \text{the triple } (h_{i-1},\, f_i, h_i) \text{ labels an } (L,\tau) \text{–admissible path.} & 
    \end{array}
    \end{equation}  
\end{defn}

\begin{lem}\label{lem:(gI)labeledpath}
Given $g\in S_n$ and $I\in P_m$, the path  in $X$ denoted as   $\gamma(g,I)$  labeled by the following tuple
\[
\bigl(
h_0, f_1, h_1, f_2, h_2,
\ldots,
f_m, h_m
\bigr)
\]
is an $(L, \tau)$-admissible path. In particular, 
writing $$x_i = h_0 f_1 h_1 \cdots f_{i-1} h_{i-1}$$ for each $1\le i\le m$, we have $$ \pi_{x_i\ax(f_i)}(o)\subseteq N_{r}(x_io)\quad \text{and} \quad \pi_{x_i\ax(f_i)}(\Phi(g,I))\subseteq N_{r}( x_if_io).$$  
\end{lem}
\begin{proof}
By Remark \ref{ConcatenationAdmPath}, the concatenation of appropriate $(L, \tau)$-admissible paths gives an $(L, \tau)$-admissible path, so the admissibility of  $\gamma(g,I)$ follows by the condition (\ref{eq:tripleadmissible}). The ``in particular" statement then follows by Proposition~\ref{admisProp}.  
\end{proof}

%Throughout the section, we assume $g\in A_n$ and $I\in P_m$ are chosen arbitrarily.

The next lemma states that the strongly contracting subsets in the above $(L, \tau)$-admissible path appears in the standard path between $\po$ and $\Phi(g, I)\po$.

\begin{defn}\label{def:char_axis}
    Let $g\in S_n$ and $I\in P_m$ as above. For $1\le i\le m$, define the \emph{$i$-th characteristic axis} of $\Phi(g, I)$ by $$Y_i(g, I):= h_0 \left(\prod_{j=1}^{i-1} f_jh_j\right) \ax(f_i)$$ 
    The set of characteristic axes forms the ordered set $$\begin{aligned}
        \mathcal{CA}(g, I):&=\Bigl\{Y_i(g, I): 1\le i\le m\Bigr\}\\
        &= \Bigl\{h_0 \ax(f_1), h_0f_1h_1\ax(f_2), \cdots, h_0 f_1 h_1 f_2 \cdots h_{m-1}\ax(f_m)\Bigr\}.
    \end{aligned}$$
\end{defn}

\begin{lem}\label{lem:CA_in_std_path}
    The characteristic axes are pairwise distinct, and
$$
\mathcal{CA}(g,I)\subseteq\f_K[\po,\Phi(g,I)\po].
$$
Moreover, their order agrees with the order on this standard path.
\end{lem}
\begin{proof}
    Given any $1\le i\le m$, denote $Y := Y_i(g, I)$. The goal is to prove that $$\proj_Y(\po, \Phi(g, I)\po)>K.$$ By Theorem~\ref{thm:forcing_seq}, $\proj_Y(\po, \Phi(g, I)\po) \ge \proj_Y^\pi(o, \Phi(g, I)o) - 2\theta$. It then suffices to prove   $$\proj_Y^\pi(o, \Phi(g, I)o)> K + 2\theta.$$

    In fact,  $\gamma(g, I)$ has $r$-fellow-travel property and $Y$ is a strongly contracting subset in the admissible path structure of $\gamma(g, I)$. Hence, $$\proj_Y^\pi(o, \Phi(g, I) o)\ge d(o, f_i o) - 2r \ge K + 2\theta$$ where we used $d(o, fo) > K + 2r + 2\theta
    $ for each $f\in F$. Similarly, we can check  $\proj_Y(\po, Y')>K$ for any $Y':=Y_j(g, I)$ with $j>i$. Thus, the order on $\mathcal{CA}(g, I)$ agrees with the order on $\f_K(\po, \Phi(g, I)\po)$.
\end{proof}

The main result of Section \ref{sec:anchoredlength}, Proposition~\ref{prop:word_geod_on_proj_cplx},   gives an upper bound of $\LS(\po, \Phi(g, I) \po)$.
\begin{lem}\label{lem:std_length_of_image}
    For any $\delta>0$, there exists a constant $E = E(G, X, S, \delta)$ such that for any $n>m>0$, $g\in S_n$ and $I\in P_m$, we have $$\LS(\po, \Phi(g, I) \po)\le \LS(\po, g\po) + \delta n + Em.$$
\end{lem}
\begin{proof}
    As in Definition~\ref{def:insert_map}, let $g = h_0\cdots h_m$ be the anchored decomposition with anchored set $I$. The  index set $I$ gives an anchored set on a geodesic from $1$ to $g$. Namely, setting $g_i = h_0 \cdots h_i$ for $0\le i\le m$, the set $$Q = \{g_0< g_1< \cdots< g_{m-1}\}$$ is a geodesically anchored set from $1$ to $g$ (see Definition~\ref{def:anchored_set}). 

    Set $D:=\max\{\LS(\po, f_i\po):f_i\in F\}$. By the weak triangle inequality (Lemma~\ref{lem:std_tri_ineq}), $$\begin{aligned}
        \LS(\po, \Phi(g, I)\po) &\le 2m + \LS(\po, h_0 \po) + \sum_{i=1}^m \bigl(\LS(\po, f_i \po) + \LS(\po, h_i \po)\bigr)\\
        &=2m + \LS_Q(\po, g\po) + \sum_{i=1}^m\LS(\po, f_i\po)\\
        &\le 2m+mD + \LS_Q(\po, g\po)
    \end{aligned}$$
    where the equality follows by the definition of the anchored length $\LS_Q$.
    
    By Proposition~\ref{prop:word_geod_on_proj_cplx}, there exists $E_1 = E_1(G, X,S, \delta)$ such that $$\LS_Q(\po, g\po)\le \LS(\po, g\po) + \delta n + E_1 m$$ 
    which yields $$\LS(\po, \Phi(g, I)\po) \le \LS(\po, g\po) + \delta n + (E_1 + D+3)m.$$
    Setting $E = E_1 + D+3$ completes the proof.
\end{proof}

\subsection{Estimating the kernel of the map}
Recall from Theorem \ref{thm:grow_tight_PC}, $$A = \bigl\{g\in G: \, \LS(\po, g\po)\le \varepsilon d_S(1, g)\bigr\}$$
Set $A_n=A\cap S_n$. We give a bound on the kernel of the restriction map $\Phi:A_n\times P_m\to G$. %The key step is to bound the size of the preimage of $\Phi$.
\begin{lem}\label{lem:size_of_preimage}
For any $\varepsilon>0$, there exists a constant $E = E(G,X,S,\varepsilon)>0$ such that for any $h\in G$ and any $n>m\ge 1$,
\[
\Bigl|\bigl\{(g,I):\; g\in A_n,\; I\in P_m,\; \Phi(g,I)=h\bigr\}\Bigr|
\le M^m \; \binom{2\varepsilon n + Em}{m}.
\]
\end{lem}

\begin{proof}
Fix $h\in G$. Let $g\in A_n$ and $I\in P_m$ satisfy $\Phi(g,I)=h$. By Lemma~\ref{lem:std_length_of_image}, given $\varepsilon>0$, there exists a constant $E>0$ (independent of $h$, $g$, and $I$) such that
\[
\LS(\po,h\po)\le \LS(\po,g\po)+\varepsilon n+Em \le 2\varepsilon n+Em.
\]

Our goal is to bound the cardinality of the preimage $\Phi^{-1}(h)$, that is the number of pairs $(g,I)$ that give rise to the same element $h=\Phi(g,I)$. We proceed in two steps. First, we determine the possible locations of the characteristic axes appearing in the admissible path $\gamma(g,I)$. We then recover the anchored decomposition of $g$ via the fellow travel property of $\gamma(g,I)$.

\medskip
\noindent\textbf{Step 1:}
Recall that the index set $I$ records the location where the map $\Phi(g,I)$ inserts the $f_i$'s into the   word $g=s_1s_2\cdots s_n$. By Lemma~\ref{lem:CA_in_std_path}, we have
\[
\mathcal{CA}(g,I)\subseteq \f_K(\po,h\po).
\]
Thus every characteristic axis arising from $(g,I)$ lies in the collection of axes between $\po$ and $h\po$. Therefore, the number of  choices for $\mathcal{CA}(g,I)$ is at most
\[
\binom{\LS(\po,h\po)}{m}\;
\le
\;\binom{2\varepsilon n+Em}{m},
\]
which is   the number of possible ways to choose $m$ indices from the set $\f_K(\po,h\po)$.

\medskip
\noindent\textbf{Step 2:}
Once $\mathcal{CA}(g,I)=\{Y_1<\dots<Y_m\}$ is fixed, the main difficulty is to determine the elements $h_i$ in the decomposition
\[
g=h_0\cdots h_m.
\]
The anchored index set
\[
I=\left\{\sum_{j=0}^{i-1} d_S(1,h_j):\, 1\le i\le m\right\}
\]
will be then uniquely determined. As we said, this relies on the $r$-fellow travel property.

The elements $f_i$ are associated to the axes via
\[
Y_i
=
h_0\left(\prod_{j=1}^{i-1} f_j h_j\right)\ax(f_i).
\]
Because the elements of $F$ are pairwise independent, the translates
of $\ax(f)$ and $\ax(f')$ cannot coincide for distinct $f,f'\in F$.
Thus, once $Y_i$ is fixed, the element $f_i$ is determined.

For a fixed $1\le i\le m$, it suffices to determine the element
\[
x_i=h_0f_1h_1\cdots f_{i-1}h_{i-1},
\]
so that $Y_i=x_i\ax(f_i).$

Indeed, once the elements $x_i$ ($1\le i\le m$) are known, the factors $h_i$ with $0\le i\le m$ can be recovered inductively from the relation
\[
h_0f_1h_1\cdots f_m h_m = h.
\]

Fix one possible prefix $x_i$, and let $x_i'$ be another possibility
corresponding to the same ordered characteristic axes. By the
$r$-fellow-travel property,
\[
d(x_io,x_i'o)\le 2r,
\qquad
d(x_if_io,x_i'f_io)\le 2r.
\]
Setting $a=x_i^{-1}x_i'$, we obtain
\[
d(o,ao)\le 2r,
\qquad
d(f_io,af_io)\le 2r.
\]
Hence $a\in\operatorname{Stab}_{2r}(o,f_io)$, so there are at most $M$
possible choices for $x_i$. Hence, the number of possible anchored decompositions $g=h_0\cdots h_m$ is bounded above by $M^m$. Combining this with the bound from Step~1, we conclude that
\[
\Bigl|\bigl\{(g,I):\; g\in A_n,\; I\in P_m,\; \Phi(g,I)=h\bigr\}\Bigr|
\le M^m \cdot \binom{2\varepsilon n+Em}{m}.
\]
The proof is complete.
\end{proof}

The output of the above analysis gives an inequality to bound $|A_n|$. 
\begin{lem}\label{lem:fund_comb_ineq}
    For any $\varepsilon>0$, there exists $E = E(G,X,S, \varepsilon)>0$ such that for any $1\le m\le n/10$, we have $$|A_n|\cdot \binom{n-1}{m} \; \le\;  |B_n| \cdot |S|^{ml} \cdot M^m \cdot \binom{2\varepsilon n + Em}{m}$$
where  $l = \max_{f\in F}\{d_S(1, f)\}$.    
\end{lem}
\begin{proof}
    The left hand side of the inequality is the size of the domain $A_n\times P_m$ of $\Phi$.

    Now we consider the right hand side of the inequality. Let $g\in A_n$ and $I\in P_m$ with anchored decomposition $g = h_0 \cdots h_m$. By triangle inequality, $$d_S(1, \Phi(g, I))\le \sum_{i=0}^m d_S(1, h_i) + \sum_{i=1}^m d_S(1, f_i)\le n + ml.$$
    Hence, $\mathrm{Im}(\Phi) \subseteq B_{n+ml}$, which then implies $$|\mathrm{Im}(\Phi)| \le |B_{n+ml}|\le |B_n| \cdot |S|^{ml}.$$
    By Lemma~\ref{lem:size_of_preimage}, for each $h\in \mathrm{Im}(\Phi)$, $$\left|\Phi^{-1}(h)\right|\le M^m \; \binom{2\varepsilon n+Em}{m}.$$
    Combining these two inequalities completes the proof.
\end{proof}
 
\begin{proof}[Proof of Theorem~\ref{thm:grow_tight_PC}]
    By Lemma~\ref{lem:fund_comb_ineq}, it suffices to find $m = m(n)$ such that the right-hand side of $$
    \frac{|A_n|}{|B_n|}\; \le \; \frac{|S|^{ml}\cdot M^m \cdot \binom{2\varepsilon n+Em}{m}}{\binom{n-1}{m}}$$ decays exponentially in $n$. Recall from \eqref{eqn:defepsilon} that $\varepsilon = \frac{1}{4|S|^{l} M}$.  

Let
\[
\delta=\min\left\{\frac1{100},\frac{\varepsilon}{E}\right\},
\qquad
m= \delta n.
\]
If
$2\varepsilon n+Em<m$, the conclusion is immediate. Otherwise, since
$n-i-1\ge n-m$ for $0\le i<m$, we obtain
    \[\begin{aligned}
        \frac{|A_n|}{|B_n|}\;&\le \;\left(M\,|S|^l\right)^m \; \prod_{i=0}^{m-1}\left(\frac{2\varepsilon n + Em - i}{n-i-1}\right)\\
        \;&\le\; \left(M\,|S|^l\right)^m \; \left(\frac{2\varepsilon n + Em}{n-m}\right)^m\\
        \;&\le\; \left(M\,|S|^l\cdot \frac{2\varepsilon + E\delta}{1-\delta}\right)^m\\
        \;&\le\; \left(M\,|S|^l\cdot \frac{3\varepsilon}{1-\delta}\right)^m\\
        \;&\le\; \left(\frac{3}{4(1-\delta)}\right)^m \;=\left(\frac{4}{5}\right)^{\delta n}.
    \end{aligned}\] 
    Thus the ratio decays exponentially. Theorem~\ref{thm:grow_tight_PC} is proved.
\end{proof}
\section{Genericity of WPD elements}\label{SecExpGenWPD}

In this section, we prove ~\ref{MainThmWPDGeneric}. We retain the notation and standing assumptions of Section~\ref{SecShortPC}; in particular, $F\subseteq G$ consists of three pairwise independent strongly contracting WPD elements, and $\PC$ denotes the associated projection complex for $K$ chosen as in \eqref{eq:defnK}.

\iffalse We assume always the following notations throughout this section.
\begin{enumerate}
    \item Let $F$ be a finite set of independent WPD elements on $X$, which satisfy the $(\kappa, N)$-diverging property;
    
    \item Suppose $\beta>0$ is a constant such that the orbit map $\iota: G\to Go$ given by $g\mapsto go$ is a $\beta$-Lipschitz map. Furthermore, for each $f\in F$, the restriction of $\iota$ on $E(f)\to \ax(f)$ is a $\beta$-quasi-isometry. Suppose that the orbit map $\iota_{K}: Go\to \PC$ given by $go\mapsto g\po$ is also a $\beta$-Lipschitz map.

    \item Let $g\in G$ be an arbitrary element. Write $$N_R(g):=\{h\in G: d_S(g, h)\le R\}.$$
\end{enumerate}\fi

\subsection{Guard decomposition of long displacements}

\begin{defn}\label{def:block_decomposition}
    Let $\gamma$ be an oriented path from $g$ to $h$ in $\mathrm{Cay}(G,S)$. Given $K>0$, a \emph{$K$-guard decomposition} of $\gamma$ consists of an ordered set of points  on $\gamma$ $$\{g_0 = g< g_1< \cdots< g_m=h\}$$  for some $m\ge 0$, and an ordered set of axes called \emph{guard axes}  $$\{Y_1< \cdots< Y_m\}\subseteq \f_K(g\po, h\po)$$ such that the following conditions are satisfied for each $1\le i\le m$: %$Y_i$ is a $K$-guard between $\{\po, g_{i-1}\po\}$ and $\{g_i\po, g\po\}$ (c.f. Definition \ref{defn:Kguard}):  
    $$\proj_{Y_i}(g\po, g_{i-1}\po)\le \theta, \quad \proj_{Y_i}(g_i\po, h\po) \le \theta,\quad \text{and} \quad Y_i\in \f_K(g_{i-1}\po, g_i\po).$$  We say the subpath $\gamma_i := [g_{i-1}, g_i]_{\gamma}$ is \emph{guarded} by $Y_i$, and $(\gamma_i, Y_i)$ a \emph{guarded block}. The \emph{length} of the decomposition is $m$. For $m=0$, we allow the empty guard decomposition, consisting of no guard axes or guarded blocks.
\end{defn}

The following lemma provides a guard decomposition for any word path.
\begin{lem}\label{lem:block_decomp}
    Let $\gamma$ be a   path from $g$ to $h$ in $\mathrm{Cay}(G,S)$ and $\mathcal{Y}\subseteq \f_K(g\po, h\po)$ be a subset of the standard path. Then there exists a $K$-guard decomposition of $\gamma$ with length $$\left\lfloor \frac{|\mathcal{Y}|+1}{12}\right\rfloor$$ whose guard axes are contained in $\mathcal{Y}$.
\end{lem}
\begin{proof}
    Denote $n=|\mathcal{Y}|+1$.  Set  $Y_0 = g\po$ and $Y_n = h\po$. Let us list $\mathcal{Y}\subseteq \f_K(g\po, h\po)$ in the induced order $$\mathcal{Y} = \{Y_1<Y_2< \cdots< Y_{n-1}\}$$ 
    
    Set $m = \lfloor n/12\rfloor$. By the bottleneck property of standard paths in Lemma~\ref{lem:bottleneck_in_PC}, for each $1\le i\le m$ there exists $g_i\in \gamma$  such that $\LS(g_i\po, Y_{12i})\le 3$. By Lemma~\ref{lem:common_prefix}, $$\{Y_1< Y_2 < \cdots < Y_{12i-5}\}\subseteq \f_K(g\po, g_i\po).$$ Hence, we may choose $g_i$ inductively from $i=m$ to $1$ such that $g_{i-1}\in [g, g_i]_{\gamma}$. Thus $g_1, \cdots, g_{m-1}$ are ordered on $\gamma$.
    
    By Corollary~\ref{cor:sub_segment_std}, $\{Y_{12i-11}< Y_{12i-10}< \cdots< Y_{12i-1}\}\subseteq \f_K(Y_{12(i-1)}, Y_{12i})$. Since $\LS(Y_{12(i-1)}, g_{i-1}\po)\le 3$ and $\LS(Y_{12i}, g_i\po)\le 3$, Lemma~\ref{lem:common_mid} implies 
    $$\{Y_{12i-7}< Y_{12i-6}< Y_{12i-5}\}\subseteq \f_K(g_{i-1}\po, g_i\po).$$ 
    
    We shall verify that $\gamma_i=[g_{i-1}, g_i]_{\gamma}$ is guarded by the axis $Y_{12i-6}$. Write $U=Y_{12i-7}$ and $V=Y_{12i-6}$ for short. Then it suffices to show that $\proj_{V}(g\po, g_{i-1}\po)<\theta$ and $\proj_{V}(g_i\po, h\po)<\theta$. Indeed, since $U\in \f_K(g_{i-1}\po, V)$, we have $\proj_{U}(g_{i-1}\po, V)>K>\theta$. The strong Behrstock inequality~\ref{axiom:strong_berstock} with $U\in \f_K(g\po, V)$ then implies
    $$\proj_{V}(g_{i-1}\po,g\po) = \proj_{V}(U, g\po)<\theta.$$
    Similarly, since $Y_{12i-5}\in \f_K(V, g_i\po)$, one proves $\proj_{V}(g_i\po, h\po)<\theta$. Thus, $\{g_i:1\le i\le m\}$ together with $\{Y_{12i-6}:1\le i\le m\}$  forms a $K$-guard decomposition.
\end{proof}
\begin{rem}\label{rem:block_every_10}
    In the proof of Lemma~\ref{lem:block_decomp}, a guard axis is picked from every 12 consecutive axes in the ordered set $\mathcal{Y}$.
\end{rem}

Let  $(\gamma_i, Y_i)$ be  a guarded block of a   path from $g$ to $h$ as in Definition \ref{def:block_decomposition}. Let $H_i$ be the left coset representing $Y_i\in \f$, i.e. $\Pi(H_i)=H_io=Y_i$. 

\begin{defn}\label{def:good_block}
Given $R>0$, we say that $(\gamma_i, Y_i)$ is \emph{$(K, R)$-good} if there exists $x\in \gamma_i$ such that $d_S(t, x)\le R$ for any group element $t\in \pi_{H_i}^X(g)$. We shall refer to $x$ as a \emph{neighbor point} of the block $(\gamma_i, Y_i)$.
Otherwise,  $(\gamma_i, Y_i)$ is called \emph{$(K, R)$-bad}. 
\end{defn}
  
We observe that the bad blocks in any guard decomposition is rare in a qualitative way.
\begin{lem}\label{lem:good_block}
Set $K>2\theta+ N$. For any $\varepsilon>0$, there exists $R=R(\varepsilon)>0$ with the following property. For any   path $\gamma$ from $g$ to $h$ in $\mathrm{Cay}(G,S)$ and any $K$-guard decomposition of $\gamma$,   the number of $(K,R)$-bad blocks is at most $\varepsilon\,\len(\gamma)$.  
\end{lem}
\begin{proof}
Let $(\gamma_i=[g_{i-1},g_i]_\gamma, Y_i)$ be a  $(K, R)$-bad block  of $\gamma$. By definition, $\proj_{Y_i}(g_{i-1} \po, g_i \po)>K$. We denote by $H_i\subseteq G$ the coset representing $Y_i$. We shall prove that $\len(\gamma_i)\ge \kappa(r)$ where   $$r:=\frac{R-\beta(3\theta +N+1)}{2\beta^2+1}$$ 
 Depending on whether $N_r(H_i)\cap \gamma_i=\emptyset$, we deal with the following two cases.
  
\medskip

\noindent \textbf{Case 1} Assume $N_r(H_i)\cap \gamma_i=\emptyset$. 
%The orbital map $\Pi:H_i\to Y_i$  is a $\beta$-quasi-isometry, so by definition of $X$-projection, $\proj^X_{H_i}(g_{i-1},g_i)\ge \proj^\pi_{Y_i}(g_{i-1} o, g_i o)/\beta-\beta$. 
   By Eq. (\ref{eq:projdist}) in Theorem~\ref{thm:forcing_seq}, $$\proj^\pi_{Y_i}(g_{i-1} o, g_i o)\ge \proj_{Y_i}(g_{i-1} \po, g_i \po)-2\theta \ge K-2\theta> N.$$  Since the coset $H_i$ is $(\kappa, N)$-divergent, we obtain $\len(\gamma_i)>\kappa(r)$. 
 
\medskip

\noindent \textbf{Case 2}. Assume $N_r(H_i)\cap \gamma_i\ne \emptyset$. Let $x$ be the entry point of $\gamma_i$ in $N_r(H_i)$. If $\proj_{Y_i}^\pi(g_{i-1}o,xo)\ge N$, then by the  $(\kappa, N)$-divergence of $H_i$, we already have $\len(\gamma_i)\ge \len([g_{i-1},x]_{\gamma_i})>\kappa(r)$ .  Otherwise, let us assume $\proj_{Y_i}^\pi(g_{i-1}o,xo)< N$ to derive a contradiction.  
\medskip

Indeed, let $y\in \pi_{H_i}^X(x)$ be an element representing the shortest projection of $xo$ to $Y_i$. By definition of the guard decomposition, $\proj_{Y_i}(g\po, g_{i-1}\po)<\theta$, so by Eq. (\ref{eq:projdist}), $\proj^\pi_{Y_i}(go,  g_{i-1}o)\le \proj_{Y_i}(g\po, g_{i-1}\po) + 2\theta\le 3\theta$.  Thus, $$\proj_{Y_i}^\pi(go,xo)\le \proj^\pi_{Y_i}(go,  g_{i-1}o)+\proj^\pi_{Y_i}(g_{i-1}o,xo)\le 3\theta +N.$$ Let us choose any $t\in \pi_{H_i}^X(g)$. It follows that $d(to,yo)\le \proj_{Y_i}^\pi(go,xo)\le 3\theta +N$. 

We are going to prove $d_S(t,x)\le R$, which contradicts $(\gamma_i,Y_i)$ being a $(K,R)$-bad block. To this end, let $z\in H_i$   such that $d_S(x,z)=d_S(x,H_i)\le r$. As $yo\in \pi_{Y_i}(xo)$, we have $$d(xo,yo)\le d(xo,zo)\le \beta d_S(x,z)\le  r\beta$$ Thus, $d(yo,zo)\le d(yo,xo)+d(xo,zo)\le 2r\beta$. Since  $\Pi: H_i\to Y_i$ is a $\beta$-quasi-isometry, $$d(to,zo)\le d(to,yo)+d(yo,zo)\le 2r\beta+3\theta +N$$  implies $$d_S(t,z)\le \beta d(to,zo)+\beta\le 2r\beta^2+\beta(3\theta +N+1) = R-r$$ where the equality follows by the definition of $r$. Thus, $d_S(t,x)\le d_S(t,z) +d_S(z,x)\le R$. 
\medskip

Summarizing the two cases, each bad block  $\gamma_i$ has length at least   $\kappa(r)$. If $m$  is the  number of $(K, R)$-bad blocks $(\gamma_i, Y_i)$, then $$m\cdot \kappa(r)\le \len(\gamma).$$ By choosing $R$ large enough so that $R>\beta(3\theta +N+1)$ and 
$$\kappa(r)\cdot \varepsilon = \kappa\left(\frac{R-\beta(3\theta +N+1)}{2\beta^2+1}\right)\cdot \varepsilon >1 ,$$ we obtain $m\le \varepsilon\,\len(\gamma)$, and
the proof is complete.
\end{proof}
\begin{rem}
    In the proof of Lemma~\ref{lem:good_block}, we observe that each $(K, R)$-bad block has length at least $$\kappa\left(\frac{R-\beta(3\theta +N+1)}{2\beta^2+1}\right).$$ This observation will be used in Lemma~\ref{lem:good_block_in_each_copy}.
\end{rem}

To continue the proof, we use the following variant of the Gromov product in the projection complex $\PC$. For three vertices $x,y,z\in \PC$, define
\[
\langle x,y\rangle_z^{\mathrm{std}}
:=\frac{1}{2}\bigl(\LS(x,z)+\LS(y,z)-\LS(x,y)\bigr).
\]

\begin{lem}\label{lem:factsonstdpath}
Let $x,y,z$ be three vertices in $\PC$. Then:
\begin{enumerate}
    \item Let $\alpha$ be the maximal common subpath of the standard paths $\f_K[z,x]$ and $\f_K[z,y]$. Then
    \[
    \bigl|\len(\alpha)-\langle x,y\rangle_z^{\mathrm{std}}\bigr|\le 3.
    \]
    \item  If
$\f_K[u,v]=\f_K[u,y]\cdot\f_K[y,v]$
for some $u\in\f_K(x,y)$ and $v\in\f_K(y,z)$,
then
$\f_K[x,z]=\f_K[x,y]\cdot\f_K[y,z].$
\end{enumerate}
\end{lem}

\begin{proof}
By Lemma~\ref{lem:triangle_in_PC},the standard paths joining the three vertices
$x,y,z$ form a tripod, except that the center may be replaced by a triangle whose
side lengths are at most $3$. After removing $\alpha$, let $a$ and
$b$ be the lengths of the remaining central portions of
$\f_K[z,x]$ and $\f_K[z,y]$, respectively, and let $c$ be the length of
the opposite central side. Then
$\langle x,y\rangle_z^{\mathrm{std}}
=\len(\alpha)+(a+b-c)/2$.
The weak triangle inequality gives $c\le a+b+1$, while
$a,b,c\le3$. Hence $-1/2\le(a+b-c)/2\le3$, proving the first assertion.

For the second assertion, consider the vertex sequence of
$\f_K[x,y]\cdot\f_K[y,z]$. Every three consecutive vertices
$A,B,C$ contained in either of the two constituent standard paths satisfy
$\proj_B(A,C)>K$ by
Proposition~\ref{prop:std_path}(\ref{mid_proj}). The same inequality holds
for the triple centered at $y$, since
$\f_K[u,y]\cdot\f_K[y,v]$ is a standard path. Therefore \cite[Lemma~2.4]{BBFS}, applied to the interval underlying the
concatenated path, shows that all its vertices occur, in the given order,
on $\f_K[x,z]$. Consecutive vertices of the concatenation are adjacent in
$\PC$, so no additional vertex of $\f_K[x,z]$ can occur between them.
Hence
$\f_K[x,z]=\f_K[x,y]\cdot\f_K[y,z],$
as required.
\end{proof}
 
\begin{lem}\label{lem:stable_axis}
Let $g\in G$ be an element such that
\[
\langle g^{-1}\po,g\po\rangle_\po^{\mathrm{std}}
\le \LS(\po,g\po)/2-7.
\]
Let $u$ denote the greatest element of
$\f_K[g^{-1}\po,g\po]\cap \f_K[g^{-1}\po,\po],$
and let $v$ denote the smallest element of
$\f_K[g^{-1}\po,g\po]\cap \f_K[\po,g\po],$
both ordered as vertices of $\f_K[g^{-1}\po,g\po]$. Then
\[
\gamma:=\bigcup_{m\in\mathbb Z}
g^m\bigl(\f_K[v,gu]\cdot\f_K[gu,gv]\bigr)
\]
is a bi-infinite standard path. In particular, $g$ is loxodromic on $\PC$.
\end{lem}

\begin{proof}
Set \(L=\LS(\po,g\po)\). Let \(\alpha\) be the common initial subpath of
\(\f_K[\po,g^{-1}\po]\) and \(\f_K[\po,g\po]\), and let \(w\) be its terminal
vertex. By Lemma~\ref{lem:factsonstdpath}(1) and the assumption,
\[
\len(\alpha)
\le \langle g^{-1}\po,g\po\rangle_\po^{\mathrm{std}}+3
\le L/2-4.
\]

By Lemma~\ref{lem:triangle_in_PC}, the standard-path triangle with vertices
\(g^{-1}\po,\po,g\po\) is tripod-like, with a center triangle of side length at
most \(3\). The vertices of this center triangle on the three sides are precisely
\(u,v,w\), where \(u\) lies on \(\f_K[g^{-1}\po,\po]\), \(v\) lies on
\(\f_K[\po,g\po]\), and \(w\) lies on both
\(\f_K[\po,g^{-1}\po]\) and \(\f_K[\po,g\po]\).

Thus, measured along the standard path \(\f_K[\po,g\po]\), the vertex \(v\)
lies within distance \(3\) of \(w\). Hence
\[
\LS(\po,v)\le \len(\alpha)+3\le L/2-1.
\]
On the other hand, measured along \(\f_K[g^{-1}\po,\po]\), the vertex \(u\)
lies within distance \(3\) of \(w\), and therefore
\[
\LS(g^{-1}\po,u)\ge L-\len(\alpha)-3\ge L/2+1.
\]
Applying \(g\), this says that \(gu\) lies on \(\f_K[\po,g\po]\) after the
midpoint, while \(v\) lies before the midpoint. Hence
$v<gu$
in the order on \(\f_K[\po,g\po]\).

Applying \(g^{-1}\) to this inequality gives
$g^{-1}v<u$
on \(\f_K[g^{-1}\po,\po]\). Since \(u\) and \(v\) occur in this order on
\(\f_K[g^{-1}\po,g\po]\) (possibly $u=v$), we obtain
$g^{-1}v<u\le v<gu$
inside \(\f_K[g^{-1}\po,g\po]\). Applying \(g\), we get
$v<gu\le gv<g^2u$
inside \(\f_K[\po,g^2\po]\). Therefore
$\f_K[v,gu]\cdot\f_K[gu,gv]\cdot\f_K[gv,g^2u]$
is a standard path. By equivariance, the same holds for every translate by a
power of \(g\), and hence every finite subpath of \(\gamma\) is a standard path.
Thus by Lemma~\ref{lem:factsonstdpath}(2), \(\gamma\) is a bi-infinite standard path.

Finally, standard paths are uniformly quasi-geodesic by
Eq.~\eqref{rem:std_len_quasi_isom}. Since \(g\gamma=\gamma\) and \(g\) translates
\(\gamma\) nontrivally, \(g\) is loxodromic on \(\PC\).
\end{proof}

This is the key lemma of this section, which will imply exponential genericity.
\begin{lem}\label{lem:split_into_conjugacy}
    Fix $\varepsilon>0$. % be such that $\LS(\po, g\po)\ge \varepsilon_0\, d_S(1, g)$.%let $\mathcal{Y}\subseteq \f_{3K}(\po, g\po)$ be a subset of vertices on the standard path from $\po$ to $g\po$. % such that $$|\mathcal{Y}|\ge \varepsilon\, d_S(1, g).$$ 
    There exists $M = M(\varepsilon, K)$ such that for any $g\in G$, at least one of the following holds:
    \begin{enumerate}
        \item  $\langle g^{-1}\po, g\po\rangle_{\po}^{\mathrm{std}}<\varepsilon\,d_S(1, g)$ in $\PC$;
        %\item  There exists a subset $\mathcal{Y}_0\subseteq\mathcal{Y}$ such that $|\mathcal{Y}_0|\ge |\mathcal{Y}|-\varepsilon\,d_S(1, g)$ and for each $Y\in \mathcal{Y}_0$, either $Y = gY$ or $Y\neq gY$ and $\proj_Y(g^{-1}Y, gY)>K$;
        %\item\label{split_into_conjugacy:parallel} There exists $Y\in \mathcal{Y}$ such that $Y=gY$;
        %\item\label{split_into_conjugacy:alive} There exists $Y\in \mathcal{Y}$ such that $Y\neq gY$ and $\proj_Y(g^{-1}Y, gY)>K;$
        \item\label{split_into_conjugacy:conjugate} There exist $u,v\in G$ such that $g=uvu^{-1}$ and 
        \begin{align}\label{eq:sumthdist}
        d_S(1, u) + d_S(1, v)\le \left(1-\frac{\varepsilon}{1000}\right)d_S(1, g) + M.    
        \end{align}
    \end{enumerate}
\end{lem}
\begin{proof}
    We assume the assertion (1) is false, and shall verify the assertion (2).
    \medskip

    \noindent
    \textbf{Step 1: Setup.} %We will find a collection of axes $\mathcal{Y}_1$ of cardinality at least $$|\mathcal{Y}_1|\ge \varepsilon\,d_S(1, g)/2$$ such that for any $Y\in \mathcal{Y}_1$, $Y\in \f_{K}(\po, g\po)$ and $gY\in \f_K(\po, g\po)$.
    %
    %\medskip\noindent
    By assumption, $\langle g^{-1}\po, g\po\rangle_{\po}^{\mathrm{std}}\ge\varepsilon\,d_S(1, g)$. Let $\mathcal{Y}_1$ be the common subpath of standard paths $\f_K[g^{-1}\po,\po]$ and $\f_K[\po,g\po]$. By Lemma~\ref{lem:factsonstdpath}, $$|\mathcal{Y}_1|\ge \langle g^{-1}\po, g\po\rangle_{\po}^{\mathrm{std}}-3\ge \varepsilon\,d_S(1, g) - 3.$$
    Then for each $Y\in \mathcal{Y}_1$, $Y\in \f_K(\po, g\po)$ and $gY\in g\cdot \f_K(g^{-1}\po, \po) = \f_K(\po, g\po)$.

    \medskip
    \noindent \textbf{Step 2: Doubly-good blocks.} 
        Let $\gamma$ and $\bar \gamma$ be an oriented word geodesic from $1$ to $g$ and its reverse, respectively.  Choose $R = R(\varepsilon/1000)>0$ as in Lemma~\ref{lem:good_block}. We will find a set of axes in $\mathcal Y_1$ denoted by $\mathcal Z$ so that the following holds:
        \begin{enumerate}[label=(\roman*)]
            \item $|\mathcal{Z}|\ge (\varepsilon/500)\cdot d_S(1, g) -2$;
            \item Each $Y\in \mathcal{Z}$ is a $(K, R)$-good block for $\gamma$ and $gY$ is a $(K, R)$-good block for $\bar \gamma$;
            \item There exists an injective neighbor map $f: \mathcal{Z}\to \gamma$ such that for each $Y\in \mathcal{Z}$ with representing coset $H$ and each $x\in \pi_H^X(1)$, $d_S(x, f(Y))\le R$. Furthermore, $d_S(f(Y_1), f(Y_2))\ge 1$ for each $Y_1, Y_2\in \mathcal{Z}$ with $Y_1\ne Y_2$.
        \end{enumerate} 
    \medskip
 
    \noindent First, by  Lemma~\ref{lem:block_decomp}, $\gamma$ admits a $K$-guard decomposition  with guard axes contained in $\mathcal{Y}_1$, and of length at least $$\left\lfloor\frac{|\mathcal{Y}_1|+1}{12}\right\rfloor\ge \frac{\varepsilon}{12}\cdot d_S(1, g)-2.$$ 
    Denote the set of guard axes of the $(K, R)$-good blocks by $\mathcal{Y}_2\subseteq \mathcal Y_1$. Then by Lemma~\ref{lem:good_block}, $$|\mathcal{Y}_2|\ge \frac{\varepsilon}{12}\cdot d_S(1, g) - 2 -\frac{\varepsilon}{1000}\cdot \len(\gamma)\ge \frac{\varepsilon}{13}\cdot d_S(1, g)-2.$$
    
   \medskip 
   \noindent  Next, we apply a similar argument to $\bar \gamma$ and $g\mathcal{Y}_2$. Lemma~\ref{lem:block_decomp} provides a $K$-guard decomposition of $\bar \gamma$ with guard axes contained in $g\mathcal{Y}_2$ and of length at least $$\left\lfloor\frac{|\mathcal{Y}_2|+1}{12}\right\rfloor\ge \frac{\varepsilon}{200}\cdot d_S(1, g) -2.$$ Denote the set of guard axes of the $(K, R)$-good blocks by $\mathcal{Z}_1\subseteq g\mathcal{Y}_2$. By Lemma~\ref{lem:good_block}, $$|\mathcal{Z}_1|\ge \frac{\varepsilon}{200}\cdot d_S(1, g) -2 - \frac{\varepsilon}{1000}\cdot \len(\gamma) = \frac{\varepsilon}{250}\cdot d_S(1, g) -2.$$
    Finally, let $\mathcal{Z}_2\subseteq \mathcal{Y}_2$ such that $\mathcal{Z}_1 = g\mathcal{Z}_2$, and set $\mathcal{Z}$ to be the subset of odd-positioned elements in $\mathcal{Z}_2$. Then $$|\mathcal{Z}|\ge |\mathcal{Z}_1|/2-1\ge \frac{\varepsilon}{500}\cdot d_S(1, g) -2.$$
    
   \medskip 
   \noindent   We verify that $\mathcal Z$ is the desired set of axes.
    
   By construction, for each $Y\in \mathcal Z$, $Y$ is a $(K, R)$-good axis for $\gamma$, while $gY$ is a $(K, R)$-good axis for the inverse $\bar\gamma$. The property (\emph{ii}) is fulfilled. 
    
    Furthermore, the point $x=f(Y)$ lies on the subpath of $\gamma$ guarded by $Y$, which are disjoint for distinct $Y\in \mathcal Z$. Since $\gamma$ is a word geodesic, each subpath has length at least $1$. This justifies the injectivity of the map $f$ in (\emph{iii}). 
    \medskip
    
    \noindent\textbf{Step 3: Short conjugation.} It remains to find $u,v\in G$  with  $g=uvu^{-1}$ satisfying  Eq. (\ref{eq:sumthdist}).
    \medskip
    
    By property (\textit{i}) in Step 2, $|\mathcal{Z}|\ge \varepsilon/500\cdot d_S(1, g)-2$. By property (\textit{iii}) in Step 2, the neighbor map $f: \mathcal{Z}\to \gamma$ is injective and the images are separated by at least distance 1. Thus, we can find $Y\in \mathcal{Z}$ and $x = f(Y)$ such that $x$ does not lie in the very front and back of $\gamma$, i.e. $$\frac{\varepsilon}{1000}\cdot d_S(1, g) - 2 \le d_S(1, x)\le \left(1-\frac{\varepsilon}{1000}\right) \cdot d_S(1, g)+2.$$ 

    Let $H$ be the left coset representing $Y$, i.e. $Y = \Pi(H) = Ho$. The property (\emph{ii}) of $\mathcal{Z}$ in Step 2 shows that for any given $t\in \pi_{H}^X(1)$, $d_S(t, \gamma)\le R$ and $d_S(gt, \gamma)\le R$. So we choose $y\in \gamma$ witnessing $d_S(y, gt)<R$. Then by triangle inequality, $$
    \begin{aligned}
    d_S(t, gt)&\le d_S(t, x) + d_S(x, y) + d_S(y, gt)\le d_S(x, y) + 2R,\\
    d_S(1, t)&\le d_S(1, x) + d_S(x, t) \le d_S(1, x) + R,\\
    d_S(1, gt)&\le d_S(1, y) + d_S(y, gt)\le d_S(1, y) + R,\\
    \end{aligned}$$
    so we have
    \begin{align}\label{eq:distxandy}
        \notag &\quad|d_S(1, y) + d_S(1, x) -d_S(1, g)|= |d_S(g, y) - d_S(1, x)|\\
        &\le |d_S(g,y)-d_S(g,gt)|+|d_S(g, gt) - d_S(1, t)|+ |d_S(1,t)-d_S(1,x)|\\
        \notag &\le d_S(y,gt) + d_S(x,t) \le 2R.
    \end{align}

    If $d_S(1, x)\le d_S(1, y)$, write $u:=t, v := t^{-1}gt$ so that $g = u\cdot v \cdot u^{-1}$. Then $$\begin{aligned}
        d_S(1, u) + d_S(t,gt)&\le d_S(1,x) + d_S(x,y) + 3R=d_S(1, y) + 3R\\
        &\le \left(1-\frac{\varepsilon}{1000}\right) \cdot d_S(1, g)+(5R+2)
    \end{aligned}$$
    verifying the inequality (\ref{eq:sumthdist}). The case that  $d_S(1, x)\ge d_S(1, y)$ could also be possible. We then  write $u:=gt, v: = t^{-1}gt$ so that $g = u\cdot v \cdot u^{-1}$, and the above inequality (\ref{eq:sumthdist}) holds by the same estimates. 
    
    Taking $M = 5R+2$ finishes the proof.
\end{proof}

%The next elementary lemma singles out a computational component of the proof. 
\begin{lem}\label{lem:Bgrowthtight}
Fix $M>0$ and $\delta>0$. Let $B(\delta, M)$ be the subset of elements $g\in G$ so that  $g = uvu^{-1}$ for some $u, v\in G$ satisfying $$d_S(1, u) + d_S(1, v)\le (1- \delta) d_S(1, g) + M.$$  
Then   $B(\delta, M)$ is growth tight.    
\end{lem}
\begin{proof}

Indeed, for each $\omega>\omega(G)$, there exists $C = C(\omega)$ such that $$\forall n\ge 0, \quad |S_n|\le |B_n|\le C(\omega)\cdot e^{\omega n}.$$  By the defining property of $g\in B$, for each $n\ge 0$ and $\omega>\omega(G)$,
$$\begin{aligned}
        |B\cap S_n|&\le \sum_{t\in B_n}|B((1-\delta)n + M - d_S(1, t))|\\
        &\le \sum_{k=1}^{(1-\delta)n} |S(k)|\cdot |B((1-\delta)n + M -k)|\\
        &\le \sum_{k=1}^{(1-\delta)n} C(\omega)e^{\omega k} \cdot C(\omega)e^{\omega ((1-\delta) n + M-k)}\\
        &\le nC(\omega)^2\cdot e^{\omega((1-\delta)n + M)},
     \end{aligned}$$
     where we write $S(k), B(k)$ for $S_k, B_k$ for clarity.

     Let us take $\omega>0$ such that $\omega(G)<\omega<\omega(G)/(1-\delta)$, so $$|B\cap S_n| \le nC(\omega)^2e^{\omega M} \cdot e^{\omega_1 n},$$ where $\omega_1 = \omega\cdot (1-\delta)<\omega(G)$. Hence, $B$ is a growth tight set.
\end{proof}
%All the ingredients are ready, and  let us spell out the proof of the genericity of WPD elements.
%\begin{proof}[Proof of \ref{MainThmWPDGeneric}]
\subsection{Proof of \ref{MainThmWPDGeneric} and \ref{MainThmRecurrence}}
Assume that $K \gg 3\theta$ is sufficiently large so that Lemma \ref{lem:lift_wpd} holds.  By Proposition~\ref{thm:grow_tight_PC} there exists $\varepsilon>0$ such that the following set is growth tight in $\mathrm{Cay}(G, S)$: $$A := \bigl\{g\in G: \, \LS(\po, g\po)\le \varepsilon \, d_S(1, g)\bigr\}$$ 
Let $M>0$ be given by Lemma~\ref{lem:split_into_conjugacy} for $\frac{\varepsilon}{4}$. Then the set $B$ of elements $g\in G$ with $g=uvu^{-1}$ and
$$\begin{aligned} 
d_S(1, u) + d_S(1, v)\le \left(1-\frac{\varepsilon}{4000}\right)d_S(1, g) + M    
\end{aligned}$$
is growth tight by Lemma \ref{lem:Bgrowthtight}. To finish the proof of \ref{MainThmWPDGeneric}, we shall prove that each element $g\in G\setminus (A\cup B)$ satisfying $d_S(1, g)\ge 28\varepsilon^{-1}$ is a strongly contracting WPD element. 
\medskip

In fact, since $g\in G\setminus A$, we have $\LS(\po, g\po)\ge \varepsilon\, d_S(1, g)$. Since $g\in G\setminus B$,  Lemma~\ref{lem:split_into_conjugacy} implies $\langle g^{-1}\po, g\po\rangle^{\text{std}}_\po\le \varepsilon\,d_S(1, g) /4$. Since $d_S(1, g) \ge 28\varepsilon^{-1}$, we have $$\langle g^{-1}\po, g\po\rangle^{\text{std}}_\po\le \varepsilon\,d_S(1, g) /4\le \varepsilon\,d_S(1, g)/2 - 7\le \LS(\po, g\po)/2 - 7.$$ By Lemma~\ref{lem:stable_axis}, $g$ is a loxodromic element in $\PC$. Hence, $g$ is a strongly contracting WPD element by Lemma \ref{lem:lift_wpd}.

\medskip

     To summarize, except finitely many elements $g$ (i.e. $d_S(1, g)\le 28\varepsilon^{-1}$),   every element $g\in G\setminus(A\cup B)$ are strongly contracting WPD elements. Since  $A\cup B$ is  growth tight, \ref{MainThmWPDGeneric} is proved.

 \medskip
 \ref{MainThmRecurrence} is proved along the way.
    For any $g\in G\setminus A$, Lemma \ref{lem:block_decomp} provides a $K$-guard decomposition of $\gamma=[1,g]_S$ with guard axes contained in $\f_K(\po,g\po)$ with length $m\ge \LS(\po,g\po)/12>2\varepsilon_0\, d_S(1, g)$, where $\varepsilon_0=\epsilon/24$.  Let $R=R(\varepsilon_0)$ given by Lemma \ref{lem:good_block} so that the number of $(K,R)$-bad blocks is at most $\varepsilon_0\,d_S(1,g)$. Hence the number of  $(K,R)$-good blocks for $\gamma$ is at least $\varepsilon_0\,d_S(1,g)$. Unveiling  the definition of good blocks proves \ref{MainThmRecurrence}.

\section{Stable length of generic elements}\label{SecWPDGen}

Throughout this section, we retain the notation and standing assumptions of Section~\ref{SecExpGenWPD}. 
\begin{defn}
    Define $\tau_X: G\to \mathbb{R}_{\ge 0}$ to be the function of the \emph{stable translation length}, that is, for each $g\in G$, $$\tau_X(g):=\lim_{n\to\infty} \frac{d(o,g^no)}{n}$$ where the limit exists due to Fekete's lemma.
    By triangle inequality, it is independent of the choice of basepoint $o\in X$.
    Similarly, define \emph{stable word length} function $\tau_S: G\to \mathbb{R}_{\ge 0}$, for each $g\in G$, $$\tau_S(g):=\lim_{n\to\infty} \frac{d_S(1,g^n)}{n}.$$

\end{defn}

We restate Theorem~\ref{MainThmStableLength} as follows.

\begin{thm}\label{thm:generic_stable_lengths}
Fix a basepoint $o\in X$ and a finite symmetric generating set $S$. For every $\varepsilon>0$, the set of elements
$g\in G$ satisfying
\[
\tau_X(g)\ge (1-\varepsilon)d(o,go)
\qquad\text{and}\qquad
\tau_S(g)\ge (1-\varepsilon)d_S(1,g)
\]
is exponentially generic.
\end{thm}

We prove the two inequalities separately. Theorem~\ref{thm:stable_trans_in_X}
establishes the assertion for $\tau_X$, while
Theorem~\ref{thm:stable_in_cay} establishes the assertion for $\tau_S$.
The theorem then follows by intersecting the two exponentially generic
sets, since a finite intersection of exponentially generic sets is again
exponentially generic.

\subsection{Preparatory lemmas}
\iffalse
The following lemma is a simple corollary of Theorem~\ref{thm:grow_tight_PC}, indicating that the orbital map is bi-Lipschitz on generic elements.
\begin{lem}\label{lem:bilipshitz_for_generic_elements}
    There exists a constant $\beta_1>0$ such that exponentially generic elements $g\in G$ satisfy $$d(o, go)\ge \beta_1\, d_S(1, g).$$
\end{lem}
\begin{proof}
    By Theorem~\ref{thm:grow_tight_PC}, for exponentially generic elements $g\in G$, $\LS(\po, g\po)\ge \varepsilon\,d_S(1, g)$ for some $\varepsilon>0$. Note that for any $u, v\in G$, $d(uo, vo)\ge (K-2\theta)\cdot \LS(u\po, v\po)$. The lemma follows by taking $\beta_1 = \varepsilon\cdot (K-2\theta)$.
\end{proof}\fi

The following lemma will be used to estimate stable lengths. It says that,
for a generic element, every prescribed initial segment of a word geodesic
determines linearly many common vertices of the associated standard paths.

\begin{lem}\label{lem:generic_prefix}
For every $0<\varepsilon<1$, there exist
$\delta=\delta(\varepsilon)>0$ and an exponentially generic subset
$G_\varepsilon\subseteq G$ with the following property.

For every $g\in G_\varepsilon$, every word geodesic $\gamma=[1,g]_S$, and
every vertex $h\in\gamma$ satisfying
$d_S(1,h)=\lfloor\varepsilon d_S(1,g)\rfloor,$
we have
\[
\bigl|\f_K(\po,g\po)\cap\f_K(\po,h\po)\bigr|
\ge \delta\,d_S(1,g).
\]
\end{lem}

\begin{proof}
Let $\varepsilon_0>0$ be the constant given by
Theorem~\ref{thm:grow_tight_PC}, and set
$\delta:=\frac{\varepsilon\varepsilon_0}{3}.$
Consider the set
\[
A:=\left\{g\in G\;\middle|\;
\begin{aligned}
&\text{there exist a word geodesic $\gamma=[1,g]_S$ and $h\in\gamma$}\\
&\text{such that }
d_S(1,h)=\lfloor\varepsilon d_S(1,g)\rfloor
\text{ and }
\LS(\po,h\po)<2\delta d_S(1,g)
\end{aligned}
\right\}.
\]

We first prove that $A$ is growth tight. By
Theorem~\ref{thm:grow_tight_PC}, the set
\[
Z:=\left\{h\in G:
\LS(\po,h\po)\le\varepsilon_0d_S(1,h)\right\}
\]
is growth tight. Choose constants
$\omega(Z,S)<\omega_1<\omega(G,S),  C_1>1$
such that
\[
|Z\cap B_m|\le C_1e^{\omega_1m}
\]
for every $m\ge0$. Choose $\omega_2>\omega(G,S)$ sufficiently close to
$\omega(G,S)$ that
\[
\omega_*:=\varepsilon\omega_1+(1-\varepsilon)\omega_2
<\omega(G,S).
\]
After increasing a constant $C_2>1$ if necessary, we may assume that
$|B_m|\le C_2e^{\omega_2m}$
for every $m\ge0$.

\medskip
Indeed, let $g\in A\cap S_n$, and let $h$ be a witness for $g\in A$. Set
$r:=d_S(1,h)=\lfloor\varepsilon n\rfloor.$
If $n\ge3\varepsilon^{-1}$, then
\[
\LS(\po,h\po)
<2\delta n
=\frac{2\varepsilon\varepsilon_0}{3}n
\le\varepsilon_0(\varepsilon n-1)
\le\varepsilon_0r.
\]
Thus $h\in Z$. Since $h$ lies on a word geodesic from $1$ to $g$, we may
write $g=hk$ with $d_S(1,k)=n-r.$
Consequently,
\[
\begin{aligned}
|A\cap S_n|
&\le |Z\cap B_r|\,|B_{n-r}|\\
&\le C_1C_2e^{\omega_1r+\omega_2(n-r)}\\
&\le C_1C_2e^{\omega_2-\omega_1}e^{\omega_*n}.
\end{aligned}
\]
Since $\omega_*<\omega(G,S)$, this proves that $A$ is growth tight.

We now consider the set
\[
D:=\left\{g\in G\setminus A\;\middle|\;
\begin{aligned}
&\text{there exist a word geodesic $\gamma=[1,g]_S$ and $h\in\gamma$}\\
&\text{such that }
d_S(1,h)=\lfloor\varepsilon d_S(1,g)\rfloor
\text{ and}\\[-2mm]
& \bigl|\f_K(\po,g\po)\cap\f_K(\po,h\po)\bigr|
<\delta d_S(1,g)
\end{aligned}
\right\}.
\]
We claim that $D$ is finite.

\medskip
Indeed, let $g\in D$, let $h$ be a witness, and write $n=d_S(1,g)$. 
Since
$g\notin A$, we have $\LS(\po,h\po)\ge2\delta n$. Applying Proposition~\ref{prop:word_geod_on_proj_cplx} to the anchored
set $\{h\}$ gives
\[
\LS(\po,h\po)+\LS(h\po,g\po)
\le\LS(\po,g\po)+\delta n+E,
\]
where $E=E(\delta)$.

\medskip
\noindent Let $\alpha:=\f_K(h\po,\po)\cap\f_K(h\po,g\po).$
By Lemma~\ref{lem:triangle_in_PC}, the standard-path triangle with
vertices $\po,h\po,g\po$ is tripod-like, with an exceptional part of
length at most $3$. Hence
\[
\LS(\po,h\po)+\LS(h\po,g\po)
\ge\LS(\po,g\po)+2\len(\alpha)-3.
\]
Combining the preceding two inequalities gives
$2\len(\alpha)-3\le\delta n+E.$
The tripod-like property also gives
\[
\begin{aligned}
\len(\alpha)
&\ge
\LS(\po,h\po)
-\bigl|\f_K(\po,g\po)\cap\f_K(\po,h\po)\bigr|-2\\
&\ge 2\delta n-\delta n-2\ge \delta n-2.
\end{aligned}
\]
where $\LS(\po,h\po)\ge2\delta n$.
It follows that
\[
2(\delta n-2)-3\le\delta n+E,
\]
and therefore $n\le\delta^{-1}(E+7)$.
Thus $D$ is finite.

The complement of the desired set is contained in $A\cup D$. Since
$A$ is growth tight and $D$ is finite, $A\cup D$ is growth tight.
Therefore \(G_\varepsilon:=G\setminus(A\cup D)\) is exponentially
generic, completing the proof.
\end{proof}

\begin{lem}\label{lem:good_block_in_each_copy}
Given any $\varepsilon>0$ there exist $L=L(\varepsilon)$ and $R=R(\varepsilon)>0$ with the following property. 
Let $g$ be an element in $G$ so that $d_S(1, g)\ge L$. Assume that   $\mathcal Y$ is a non-empty subset of $\f_K(\po,g\po)$ so that
\begin{enumerate}
    \item 
     $|\mathcal Y|\ge \varepsilon   \cdot d_S(1,g)$;
    \item 
    for any $n\ge 1$, $\cup_{m=0}^n g^m\mathcal Y$ is contained in $\f_K(\po,g^{n+1}\po)$;
    \item 
    $g^i\mathcal Y$ appears before $g^j\mathcal Y$ in $\f_K(\po, g^k\po)$ for any $i<j<k\in \mathbb N$.
\end{enumerate} Then  for  any word geodesic $[1,g^{n+1}]_S$ with $n\ge 1$ and each $0\le m\le n$, there exists $Y_m\in \mathcal Y$ so that $g^mY_m$ is a $(K,R)$-good guard axis for $[1,g^{n+1}]_S$.  Furthermore, if $H_m$ is the coset representing $g^mY_m$, then $\proj_{H_m}^X(1, g^m)\le R$.
\end{lem}
\begin{proof}
    Let us fix $\gamma:=[1,g^{n+1}]_S$. We first prepare the data to build a $K$-guard decomposition of $\gamma$.
    
    Let $Y_0$ denote the minimal axis in $\mathcal Y\subseteq \f_K(\po,g\po)$  with respect to the total order. Applying Lemma~\ref{lem:block_decomp} to $\mathcal Y\setminus \{Y_0\}$, $[1, g]_S$ admits a $K$-guard decomposition with guard axes contained in $\mathcal Y\setminus\{Y_0\}$, which has length at least $\lfloor{|\mathcal Y|}/{12}\rfloor$. Let $R_0 = R(\varepsilon/24)$ be given by Lemma~\ref{lem:good_block} which yields a set, denoted as $\mathcal Z_0$, of guard axes of the $(K, R_0)$-good blocks of $[1, g]_S$ with cardinality   $$|\mathcal Z_0| \ge \left\lfloor\frac{|\mathcal Y|}{12}\right\rfloor-\frac{\varepsilon}{24}\, d_S(1, g)\ge  \frac{\varepsilon}{24}\, d_S(1, g) -1$$
    where $|\mathcal Y|\ge \varepsilon   \cdot d_S(1,g)$ by assumption (1).
    
    \medskip
    \noindent \textbf{The $K$-guard decomposition of $\gamma$.}
    By assumption (2),  $\mathcal Z:=\cup_{m=0}^n g^m\mathcal Z_0$ is a subset of $\f_K(\po,g^{n+1}\po)$, so by Lemma \ref{lem:block_decomp}, $\gamma=[1,g^{n+1}]_S$ admits a $K$-guard decomposition with guard axes in $\mathcal Z$ of length at least $\lfloor |\mathcal Z|/12 \rfloor$. In fact, by Remark~\ref{rem:block_every_10} the set of guard axes which we denote by $\hat{\mathcal Z}\subseteq \mathcal Z$ has the following explicit description : there exist subsets $\mathcal{Y}_0, \cdots, \mathcal{Y}_{n}\subseteq \mathcal{Z}_0$ such that:
    \begin{itemize}
        \item $|\mathcal{Y}_i|\ge \lfloor|\mathcal{Z}_0|/12\rfloor$ for each $0\le i\le n$;
        \item  $\hat{\mathcal Z}=\mathcal{Y}_0\sqcup g\mathcal{Y}_1 \sqcup\cdots \sqcup g^{n}\mathcal{Y}_{n}.$
    \end{itemize}
 
    \noindent 
    We first prove the ``furthermore" statement.
    Namely, given $0\le m\le n$  and $Y\in g^m\mathcal Y_m$, $$\proj_H^X(1, g^m)\le  3\beta\theta+\beta$$
    where $H$ is the coset representing $Y$. Indeed, since $g^{-m}\mathcal Y\subseteq \f_K(\po, g\po)$ and $Y_0$ is the minimal in $\mathcal Y$, we have $\proj_{g^mY_0}(g^m\po, Y) = \proj_{Y_0}(\po, g^{-m}Y)>K$. Since $g^m\mathcal Y\subseteq \f_K(\po, g^{n+1}\po)$, Proposition~\ref{prop:std_path} gives $\proj_{Y}(\po, g^mY_0)<\theta$. Therefore by \ref{axiom:strong_berstock} and Eq.~(\ref{eq:projdist}), $$\proj^\pi_Y(o, g^mo)\le 2\theta + \proj_Y(\po, g^m\po)=2\theta + \proj_Y(\po, g^mY_0)\le 3\theta,$$ whence $\proj_H^X(1, g^m)\le \beta \proj^\pi_Y(o, g^mo)+\beta\le  3\beta\theta+\beta$ by the $\beta$-quasi-isometry map $\Pi_H: H\to Ho$.
    %assume that $Y$ is a $(K, R)$-good block for $\gamma=[1, g^{n+1}]_S$. Let $H$ be the coset representing $Y$. 
        
   \medskip
    
   \noindent The remainder of the proof is to find $Y_m\in \mathcal Y_m$ for each  $0\le m\le n$ so that $g^mY_m\in \hat{\mathcal Z}$ is a $(K,R)$-good guard axis for $\gamma$. 
   The constant  $R$ shall be determined below; see Eq.~(\ref{eq:value_of_R}). 
   
   We argue by contradiction. Fix $0\le m\le n$. Suppose that for each $Y\in \mathcal{Y}_m$, the block guarded by $g^mY$ is a $(K, R)$-bad block. If we set 
    \begin{align}\label{eq:value_of_r}
    r = \frac{R-\beta(3\theta +N+1)}{2\beta^2+1},    
    \end{align} the subpath of a $(K, R)$-bad block has length at least $\kappa\left(r\right)$ by Lemma~\ref{lem:good_block}. Here, $\kappa$ is the diverging function of WPD elements $f\in F$. 
    
    The proof idea is that, if $R$ is sufficiently large, then $\kappa(r)\to \infty$ and the total length of  $(K, R)$-bad blocks guarded by all $Y\in \mathcal{Y}_m$ will exceed the length of $\gamma$; a contradiction. The next paragraph explains which $(K, R)$-bad blocks are chosen.

    \medskip
    \noindent We first consider at the axes $\mathcal{Y}_0\sqcup g\mathcal{Y}_1 \sqcup\cdots \sqcup g^{n-1}\mathcal{Y}_{m-1}$ over the indices $0\le i\le m-1$ before $g^m\mathcal{Y}_m$. We examine the following two cases.
    
    \medskip
    \noindent\textbf{Case 1}. There exists a $(K,R)$-good block of $\gamma=[1, g^{n+1}]_S$ before those guarded by $g^m\mathcal{Y}_m$. Let $Y\in g^i\mathcal{Y}_i$ ($0\le i\le m-1$) be the $(K,R)$-good guard axis that is maximal in the order, and $H$ its representing coset. Choose $x\in\pi_H^X(1)$, $y\in\pi_{H}^X(g^i)$, and let $a\in\gamma$ be the neighbor point of this good block. Then $d_S(a,x)\le R$ and $d_S(x, y)\le \proj_H^X(1, g^i)\le 3\beta\theta + \beta$ by \textbf{(1)} above. 
    
    Since $g^{-i}Y\in\mathcal Z_0$, it is  a $(K, R_0)$-good guard axis for $[1, g]_S$. Moreover, $g^{-i}y\in\pi_{g^{-i}H}^X(1)$.
Hence there exists $t\in[1,g]_S$ such that $d_S(t,g^{-i}y)\le R_0$.
Consequently,
\[
d_S(g^i,y)=d_S(1,g^{-i}y)
\le d_S(1,t)+R_0
\le d_S(1,g)+R_0.
\] 
Therefore by triangle inequality, setting $M = 3\beta\theta+\beta+R+R_0$,
    \begin{equation}\label{equ:good_copy_1}
        d_S(g^i,a)\le d_S(g^i, y) + d_S(y, x) + d_S(x, a)\le d_S(1,g)+M.
    \end{equation}
    
    \noindent\textbf{Case 2}. If no such good block exists, we set $i=-1$ and $a=1$.

    \medskip
    \noindent Similarly, we consider $g^{m+1}\mathcal{Y}_{m+1}\sqcup g^{m+2}\mathcal{Y}_{m+2} \sqcup\cdots \sqcup g^{n}\mathcal{Y}_{n}$ over the indices $n\ge i\ge m+1$ after $g^m\mathcal{Y}_m$. If there is a  $(K,R)$-good block, we find $j\ge m+1$ and $b\in\gamma$ with $d_S(g^j,b)\le d_S(1,g)+M$; otherwise we take $j=n+1$, $b=g^{n+1}$.

    \medskip
    \noindent In either case, by Eq.~(\ref{equ:good_copy_1}) the triangle inequality gives
    \begin{align}
    \notag d_S(a,b)&\le d_S(a, g^i) + d_S(g^i, g^j) + d_S(g^j, b)\\ 
    &\le (j-i+2)d_S(1,g)+2M.\label{equ:good_copy_3}
    \end{align}
    
    \medskip
    
    Let us now estimate the total lengths of those bad blocks we found. By the above defining property of $i,j$ and $a,b$, every axis between $g^{i+1}\mathcal{Y}_{i+1}$ and $g^{j-1}\mathcal{Y}_{j-1}$ guards a $(K,R)$-bad block of $\gamma$, so their block subpaths contribute
    \begin{align}
    \notag d_S(a,b)&\ge \kappa(r)\sum_{l=i+1}^{j-1}|\mathcal{Y}_l|\\
    \label{equ:good_copy_2} &\ge (j-i-1)\,\kappa(r)\left(\frac{|\mathcal{Z}_0|+1}{12}-1\right)\\
    \notag &\ge (j-i-1)\,\kappa(r)\left(\frac{\varepsilon}{300}d_S(1,g)-1\right).
    \end{align}
    
    \noindent Choose $R$ sufficiently large that
    \begin{align}\label{eq:value_of_R}
    R>\max\{R_0,3\beta\theta+\beta,\beta(3\theta+N+1)\}\end{align}
    and  $Q:=\kappa(r)>1500\varepsilon^{-1}$ where $r$ given in (\ref{eq:value_of_r}) depends on $R$. 

Set $s:=j-i-1\ge1$. Combining
\eqref{equ:good_copy_2} and \eqref{equ:good_copy_3} gives
\[
sQ\left(\frac{\varepsilon}{300}d_S(1,g)-1\right)
\le 4s\,d_S(1,g)+2M.
\]
Since $Q\varepsilon/300>5$, it follows that
\[
s\bigl(d_S(1,g)-Q\bigr)\le2M.
\]
Thus $d_S(1,g)\le Q+2M$. Taking
\[
L:=\lceil Q+2M\rceil+1
\]
gives the desired contradiction. The proof is complete.
\end{proof}

We next verify that the assumptions of Lemma \ref{lem:good_block_in_each_copy} hold generically.

\begin{lem}\label{lem:generic_many_axes_stable}
For any $0<\varepsilon<1$ there exist $R=R(\varepsilon)>0$ and $\delta = \delta(\varepsilon)>0$ and an  exponentially generic set $G_\varepsilon\subseteq G$  with the following properties.

\begin{enumerate}
    \item For each $g \in G_\varepsilon$, let $u$ denote the greatest element in $\f_K[g^{-1}\po,g\po]\cap \f_K[g^{-1}\po,\po]$ and $v$ the smallest element in $\f_K[g^{-1}\po,g\po]\cap \f_K[\po,g\po]$, in the order on $\f_K[g^{-1}\po,g\po]$. Then there exists a subset $\mathcal Y$ contained in $\f_K(v, gu)$ of cardinality at least $\delta\,d_S(1, g)$;
    \item Furthermore, pick $h\in [1, g]_S$ so that $d_S(1, h) = \lfloor\varepsilon\,d_S(1, g)\rfloor$. Then $\mathcal Y$ consists of $(K, R)$-good guard axes for a fixed word geodesic $[1, h]_S$.
    \item For every $n\ge1$,
\[
\bigcup_{m=0}^n g^m\mathcal Y
\subseteq\f_K(\po,g^{n+1}\po),
\]
and these subsets occur in the order
$\mathcal Y<g\mathcal Y<\cdots<g^n\mathcal Y.$
\end{enumerate}
\end{lem}
\begin{proof}
    Let $\delta_0 = \delta(\varepsilon)$ be given in Lemma~\ref{lem:generic_prefix}. Then there exists an exponentially generic set $G_\varepsilon\subseteq G$ of elements $g\in G$ so that $$|\f_K(\po, h\po)\cap \f_K(\po, g\po)|\ge \delta_0\,d_S(1, g)$$ where  $h\in [1, g]_S$ so that $d_S(1, h) = \lfloor\varepsilon\,d_S(1, g)\rfloor$. 
    Excluding the elements in Eq.~(\ref{eq:sumthdist}) of  Lemma~\ref{lem:split_into_conjugacy} from $G_\varepsilon$ with $\varepsilon_0 = \delta_0/100$ removes a growth tight subset of elements by Lemma~\ref{lem:Bgrowthtight}, so we may assume without loss of generality that for each $g\in G_\varepsilon$,  $$\langle g^{-1}\po, g\po\rangle_\po^{\text{std}}\le \delta_0\,d_S(1, g)/100.$$  Since  $\LS(\po,g\po)
\ge\delta_0d_S(1,g)$, we may assume further that for each $g\in G_\varepsilon$,
\[
\langle g^{-1}\po,g\po\rangle_\po^{\mathrm{std}}
\le\frac12\LS(\po,g\po)-7,
\] 
after removing finitely many additional elements.
Thus Lemma~\ref{lem:stable_axis} applies.

    Consider the $K$-guard decomposition of $[1, h]_S$ with guard axes contained in $\alpha:=\f_K(\po, h\po)\cap \f_K(\po, g\po)$ provided by Lemma~\ref{lem:block_decomp}. By Lemma~\ref{lem:good_block}, there exists $R = R(\varepsilon)$ and a set $\mathcal{Z}\subseteq\alpha$ of $(K, R)$-good guard axes for $[1, h]_S$ so that  $|\mathcal{Z}|\ge \delta_0\,d_S(1, g)/20-1$. By Lemma~\ref{lem:factsonstdpath}, $$|\f_K(g^{-1}\po, \po)\cap \f_K(\po, g\po)|\le \langle g^{-1}\po, g\po\rangle_\po^{\text{std}} + 3\le \frac{\delta_0}{50}\cdot d_S(1, g)-3.$$ Note that $\f_K(\po, g\po) = g\cdot \f_K(g^{-1}\po, \po)$ and $\f_K(g\po, g^2\po) = g\cdot \f_K(\po, g\po)$. By Lemma~\ref{lem:triangle_in_PC}, the tripod-like property implies $$\begin{aligned}
        |\mathcal Z\cap \f_K(v,gu)|&\ge |\mathcal{Z}|-|\f_K(g^{-1}\po, \po)\cap \f_K(\po, g\po)| - |\f_K(\po, g\po)\cap \f_K(g\po, g^2\po)| -4\\
        &\ge \left(\frac{\delta_0}{20}\cdot d_S(1, g)-1\right) - 2\left(\frac{\delta_0}{50}\cdot d_S(1, g)-3\right) -4\\
        &\ge \frac{\delta_0}{100}\cdot d_S(1, g).
    \end{aligned}$$
    Setting  $\delta = \delta_0/100$ and $\mathcal Y = \mathcal Z\cap \f_K(v, gu)$ proves the first two assertions. By the construction in the proof of Lemma~\ref{lem:stable_axis},
together with Corollary~\ref{cor:sub_segment_std}, the translates
$g^m\mathcal Y$ are contained in $\f_K(\po,g^{n+1}\po)$ and occur
in increasing order. This proves the third assertion.  
\end{proof}

\begin{lem}\label{lem:entry_point_ordered_on_geod}
Assume that $K>4\theta+5C$. Let $[o,go]$ be a geodesic in $X$ and
$\{Y_1<\cdots<Y_n\}\subseteq\f_K(\po,g\po)$. Pick
$x_i\in\pi_{Y_i}(o)$ for each $1\le i\le n$. Then there exists an
ordered sequence $z_1,\ldots,z_n$ on $[o,go]$ such that
$d(z_i,x_i)\le2C$ for every $1\le i\le n$.
\end{lem}

\begin{proof}
Since $Y_i\in\f_K(\po,g\po)$, Eq.~\eqref{eq:projdist} and bounded
projection give
$\proj_{Y_i}^{\pi}(o,go)>K-4\theta>C$. Thus, by
Lemma~\ref{BigFive}, we may choose $z_i$ to be the first point on
$[o,go]$ such that $d(z_i,x_i)\le2C$.

It remains to prove that these points occur in the required order.
For each $1\le i\le n-1$, Proposition~\ref{prop:std_path} gives
$\proj_{Y_i}(\po,Y_{i+1})>K$. Hence Eq.~\eqref{eq:projdist} and
bounded projection imply
$\proj_{Y_i}^{\pi}(o,x_{i+1})>K-4\theta$. Moreover, by
Lemma~\ref{BigFive},
$\proj_{Y_i}^{\pi}(z_{i+1},x_{i+1})\le4C$. Therefore,
\[
\proj_{Y_i}^{\pi}(o,z_{i+1})
>K-4\theta-4C>C.
\]
Lemma~\ref{BigFive} now implies that $[o,z_{i+1}]$ contains a point
at distance at most $2C$ from $x_i$. Since $z_i$ is the first such
point on $[o,go]$, we have $z_i\in[o,z_{i+1}]$. Thus
$z_1,\ldots,z_n$ occur in order.
\end{proof}

\subsection{Estimating stable lengths}
We first  estimate the stable length of generic elements in $X$. %The proof uses only the fact that $\mathcal{Y}$ is non-empty in Lemma~\ref{lem:generic_many_axes_stable}.
\begin{thm}\label{thm:stable_trans_in_X}
For every $0<\varepsilon<1$, there exists an exponentially generic set
of elements $g\in G$ satisfying
\[
\tau_X(g)\ge(1-\varepsilon)d(o,go).
\]
\end{thm}

\begin{proof}
By~\ref{MainThmShortElems}, there exists $\beta_1>0$ such that
exponentially generic elements $g\in G$ satisfy
$d(o,go)\ge\beta_1d_S(1,g)$. Decreasing $\beta_1$ if necessary, assume
that $\beta_1\le\beta$, where $\beta$ is the Lipschitz constant of the
orbital map.

Pick $h\in[1,g]_S$ such that
$d_S(1,h)=\lfloor(\varepsilon\beta_1/4\beta)d_S(1,g)\rfloor$.
Combining Lemmas~\ref{lem:generic_many_axes_stable} and
\ref{lem:good_block_in_each_copy}, we obtain a constant $R>0$ such
that, for exponentially generic $g$ and every $n\ge1$, there exist
$Y_0,\ldots,Y_n\in\f_K(\po,g\po)\cap\f_K(\po,h\po)$ for which
\[
Y_0<gY_1<\cdots<g^nY_n
\quad\text{in }\;\f_K(\po,g^{n+1}\po).
\]
Moreover, if $H_i$ represents $A_i:=g^iY_i$, then
$\proj_{H_i}^X(1,g^i)\le R$.

Choose $x_i\in\pi_{Y_i}(o)$. Since
$Y_i\in\f_K(\po,h\po)$, Lemma~\ref{BigFive} gives
\[
\begin{aligned}
d(o,x_i)
&\le d(o,ho)+2C\\
&\le\beta d_S(1,h)+2C\\
&\le\frac{\varepsilon\beta_1}{4}d_S(1,g)+2C\\
&\le\frac{\varepsilon}{4}d(o,go)+2C.
\end{aligned}
\]
Consequently,
$d(x_i,gx_{i+1})\ge d(o, go) - d(o,x_i) - d(go, gx_{i+1})\ge(1-\varepsilon/2)d(o,go)-4C$.

Set $q_i:=g^ix_i\in\pi_{A_i}(g^io)$ and choose
$p_i\in\pi_{A_i}(o)$. Since $\proj_{H_i}^X(1,g^i)\le R$, the
$\beta$-Lipschitz property of the orbital map gives
$d(p_i,q_i)\le\beta R$.

Applying Lemma~\ref{lem:entry_point_ordered_on_geod} to
$A_0<\cdots<A_n$, we obtain ordered points
$z_0,\ldots,z_n\in[o,g^{n+1}o]$ such that
$d(z_i,p_i)\le2C$. Hence $d(z_i,q_i)\le2C+\beta R$. It follows that
\[
\begin{aligned}
d(o,g^{n+1}o)
&\ge\sum_{i=1}^n d(z_{i-1},z_i)\ge\sum_{i=1}^n \bigl(d(q_{i-1},q_i)-d(q_{i-1},z_{i-1})-d(q_{i},z_i)\bigr)\\
&\ge n\bigl((1-\varepsilon/2)d(o,go)-8C-2\beta R\bigr).
\end{aligned}
\]
Taking $n\to\infty$ gives
$\tau_X(g)\ge(1-\varepsilon/2)d(o,go)-(8C+2\beta R)$. Finally, after removing the finitely many elements satisfying
\[
d_S(1,g)\le
\frac{2(8C+2\beta R)}{\varepsilon\beta_1},
\]
we have $8C+2\beta R\le\varepsilon d(o,go)/2$, and therefore
$\tau_X(g)\ge(1-\varepsilon)d(o,go)$.
\end{proof}

We next estimate the stable word length.
\begin{thm}\label{thm:stable_in_cay}
For every $0<\varepsilon<1$, there exists an exponentially generic set
of elements $g\in G$ satisfying
\[
\tau_S(g)\ge(1-\varepsilon)d_S(1,g).
\]
\end{thm}

\begin{proof}
Apply Lemma~\ref{lem:generic_many_axes_stable} with parameter
$\varepsilon/4$, and denote its good-block constant by $R_0$ and the
resulting density by $\delta>0$. Apply
Lemma~\ref{lem:good_block_in_each_copy} with density $\delta$, and
denote its constant by $R_1$. Set $R:=R_0+R_1$.

Let $n\ge1$ and let $\gamma=[1,g^{n+1}]_S$ be a word geodesic. For
exponentially generic $g$ and each $0\le k\le n$, the preceding lemmas
give an axis
$Y_k\in\f_K(g^k\po,g^{k+1}\po)$ such that $Y_k$ is a
$(K,R_1)$-good guard axis for $\gamma$ and
\[
Y_0<Y_1<\cdots<Y_n
\quad\text{in }\f_K(\po,g^{n+1}\po).
\]
Let $H_k$ be the coset representing $Y_k$. Pick
$t_k\in\pi_{H_k}^X(1)$ and a neighbor point $x_k\in\gamma$ for
$Y_k$. Then $d_S(x_k,t_k)\le R_1\le R$, and the points
$x_0,\ldots,x_n$ occur in this order on $\gamma$.

Write $Y_k=g^k\widetilde Y_k$ and
$H_k=g^k\widetilde H_k$, where $\widetilde Y_k\in\mathcal Y$.
Since $\widetilde Y_k$ is a $(K,R_0)$-good guard axis for
$[1,h]_S$, there exists $w_k\in[1,h]_S$ such that
$d_S(w_k,s)\le R_0$ for every
$s\in\pi_{\widetilde H_k}^X(1)$. Hence, by equivariance,
\[
\begin{aligned}
d_S\bigl(g^k,\pi_{H_k}^X(g^k)\bigr)
&=d_S\bigl(1,\pi_{\widetilde H_k}^X(1)\bigr)\\
&\le d_S(1,h)+R_0\\
&\le\frac{\varepsilon}{4}d_S(1,g)+R_0.
\end{aligned}
\]
while Lemma~\ref{lem:good_block_in_each_copy} gives
$\proj_{H_k}^X(1,g^k)\le R_1$. Therefore,
$d_S(g^k,t_k)\le\varepsilon d_S(1,g)/4+R$.

Assume in addition that $d_S(1,g)\ge8R\varepsilon^{-1}$. Summing the
lengths of the subpaths $[x_{k-1},x_k]_\gamma$, we obtain by triangle inequality
\[
\begin{aligned}
d_S(1,g^{n+1})
&\ge\sum_{k=1}^n d_S(x_{k-1},x_k)\\
&\ge \sum_{k=1}^{n} \left(d_S(g^{k-1}, g^k)-d_S(g^{k-1}, t_{k-1}) - d_S(g^k, t_k) - d_S(t_{k-1}, x_{k-1}) - d_S(t_k, x_k)\right)\\
&\ge n\bigl((1-\varepsilon/2)d_S(1,g)-4R\bigr) \ge n(1-\varepsilon)d_S(1,g).
\end{aligned}
\]
Taking $n\to\infty$ yields
$\tau_S(g)\ge(1-\varepsilon)d_S(1,g)$.
\end{proof}

\section{Application I: Growth tightness of acylindrically hyperbolic groups}\label{SecGrowthTight}

In this section, we prove a uniform growth gap for the Schreier graphs of
confined subgroups of acylindrically hyperbolic groups. The proof combines
two ingredients developed earlier. Theorem~\ref{thm:grow_tight_PC} provides
linearly many well-separated positions along a standard path for generic
elements, while the confined extension lemma provides nontrivial insertions
at these positions that preserve the corresponding coset. Admissibility makes
the insertion patterns distinguishable, and the resulting exponential
multiplicity yields the desired growth gap.

\begin{defn}
		A nontrivial subgroup $H\subseteq G$ is called \textit{confined} if there exists a finite set $P\subseteq G$ such that for every element $g\in G$, $g^{-1} H g\cap (P \setminus \{1\})\neq \emptyset$. The set $P$ is called a \textit{confining subset} for $H$.
	\end{defn}
    A nontrivial normal subgroup $H$ is confined, since we may choose $P$ to contain a nontrivial element of $H$. Note that, if $F$ is a nontrivial finite normal subgroup, then for  any subgroup $H<G$, the product $HF$ is confined in $G$. Thus there is no hope to study  commensurability-invariant properties (e.g. growth rate) of confined subgroups without further restrictions. To avoid pathological examples, we will focus on the class of confined subgroups admitting a non-degenerate confining subset.

   \subsection{Main results and consequences}
   It is known that if $G$ is a non-elementary acylindrically hyperbolic group then $G$ contains   a unique, maximal, finite normal subgroup denoted $E(G)$ (see \cite[Theorem 6.14(a)]{DGO}).  A subset $P \subseteq G$ is called \textit{non-degenerate} if it  is disjoint from $E(G)$.
   
    \begin{thm}\label{thm:growthtightnessAH}
    Assume that $G$ is a non-elementary acylindrically hyperbolic group. Then for any finite symmetric generating set $S$ and any finite non-degenerate  subset $P$, there exists a constant $\omega_0=\omega_0(P,S)< \omega(G,S)$ so that for any confined subgroup $H$ with confining set $P$,
    \[
    \omega(G/H,\bar S) \le \omega_0.
    \]
    \end{thm}

The important point is that the upper bound $\omega_0(P,S)$ is uniform over
all confined subgroups with the same confining set $P$.
     Since $E(G)$ is finite, any infinite normal subgroup must admit a non-degenerate confining subset. We thus obtain \ref{MainThmGrowthTight} as an immediate corollary. 
      \begin{cor}
    Assume that $G$ is a non-elementary acylindrically hyperbolic group. Then for any finite generating set $S$ and any infinite normal subgroup $H$,
    \[
    \omega(G/H,\bar S) <  \omega(G,S).
    \]
    \end{cor}
    To derive the strict lower bound on cogrowth, let us cite the following.
    \begin{thm}\cite[Theorem D]{DY26A}\label{thm:growthcogrowth}
    Assume that $G$ is a non-elementary acylindrically hyperbolic group. Then for any finite generating set $S$ and for any confined subgroup $H$ with a finite non-degenerate confining set,
    \[
    \omega(H, S)+\frac{\omega(G/H,\bar S)}{2}\ge \omega(G, S).
    \]
    \end{thm}

     \begin{cor}
    Assume that $G$ is a non-elementary acylindrically hyperbolic group. Then for any finite generating set $S$ and any infinite normal subgroup $H$,
    \[
    \omega(H,S) >  \omega(G,S)/2.
    \]
    \end{cor}
    
    The remainder of this section is to prove Theorem~\ref{thm:growthtightnessAH}.

    \subsection{Projection complex setup} 
     
    Let $G$ act by isometry on a geodesic metric space $(X,d)$ with at least two strongly contracting WPD elements.  As in Section \ref{sec:anchoredlength}, choose the axis system $\f$ of a non-empty  finite set $F$  of strongly contracting WPD elements and let $(\PC,\rho)$ be the projection complex constructed from   $\f$. Then for $K\gg 0$, $G$ admits a non-elementary acylindrical action on  $\PC$ by \cite[Theorem 5.10]{BBFS}.   Let $\po$ denote the  basepoint in $\PC$, which is the axis $\ax(f_0)$ for a fixed choice $f_0\in F$. 
    
    \medskip
    \noindent\textbf{Convention.} 
    In the remainder of this section,  we work with the action on the projection complex  $\PC$ (instead of the space $X$!). Hence, if  $f$  is a loxodromic element on $\PC$,  $\ax(f)=E(f)\po$ and $\proj_{\ax(f)}^\pi$ will refer to the axis   and the $\rho$-diameter of the shortest projection  to $\ax(f)$ respectively. %These are \emph{not} the vertices and projection maps that appear in the definition of $\PC$.
    
    We adapt the insertion argument of \cite{DY24} to the action on the
projection complex. The new input is Theorem~\ref{thm:grow_tight_PC}, which
supplies linearly many insertion positions along standard paths without any
statistically convex-cocompact hypothesis. The confined extension lemma then
allows us to perform insertions at these positions while remaining in the
same coset of the confined subgroup.

    \medskip
    
    We recall the following result proved in \cite[Lemma 5.7]{CGYZ} for the action of $G$ on $\PC$.  See \cite[Lemma 7.7]{DY26A} also. 
	\begin{lem}\label{ExtensionLemConfined}
		Let $P$ be a finite non-degenerate subset in $G$. Then there exist $ \tau_0> 0$ and a finite subset $Q_0\subseteq G$ of independent loxodromic elements on   $\PC$ with the following property:
  
       Let  $H\subseteq G$ be a confined subgroup with a confining subset $P$.  For any $g, h\in G$ and $n\ge 1$, there exist $q\in Q_0^n$  and $p\in P$ such that 
       \begin{enumerate}
           \item 
           $g q p q^{-1} g^{-1}\in H$,
           \item $p\ax(q)\ne \ax(q)$,
           \item each of $[\po,g\po]$, $[\po,p\po]$ and $[\po,h\po]$ has $\tau_0$-bounded projection to $\ax(q)$,        
       \end{enumerate}
       where $Q_0^n=\{q^n:q\in Q_0\}$. 
	\end{lem} 
      
    \noindent\textbf{Admissible constants.} 
    Let $Q_0\subseteq G$ be a finite set  of independent loxodromic elements on  $\PC$, and $\tau_0>0$ given by Lemma \ref{ExtensionLemConfined}. Denote  
    $$\mathcal Q=\{g\ax(q): q\in Q_0, g\in G\}$$ the system of loxodromic axes associated to $Q_0$. 

    \medskip
    \noindent  By Morse Lemma, we may  specify the further constants   $\tau_1,L_0$ so that the following holds. 
\begin{enumerate} 

    \nameditem{$\tau_1$}{cst:tau}
    If a geodesic segment has $\tau_0$-projection to $Y\in \mathcal Q$, then any geodesic with endpoints in a $3$-neighborhood of it has $\tau_1$-projection to $Y$. 
    \nameditem{$L_0$}{cst:L} If a geodesic segment of length $\ge L_0$ between $Y_1,Y_2\in \mathcal Q$ has $\tau_1$-projection to $Y_1$ and $Y_2$, then $Y_1\ne Y_2$.
\end{enumerate} 
Let $r=r(\tau_1),L_1=L_1(\tau_1)$ be given by Proposition \ref{admisProp} for $(L_1,\tau_1)$-admissible path. Assume that $$L>\max\{L_0,L_1,4r\}$$
Denote $Q=Q_0^n$ for some large $n\ge 1$ so that $\LS(\po,q\po)>L$ for $q\in Q$. This will not affect the above constants by Lemma \ref{ExtensionLemConfined}.  %Hence,  the tuple $(g, q, p, q^{-1}, h)$ labels an $(L, \tau_0)$-admissible path relative to $\mathcal Q$ in $\PC$.   

We first isolate the consequence of Theorem~\ref{thm:grow_tight_PC} needed
for the insertion argument: generic elements admit linearly many uniformly
separated positions along their standard paths.
    %For completeness,  we recall the necessary matetails from \cite{DY24} and emphasize the main differences.
        
    \subsection{Reduction to linearly recurrent elements} 

     We adapt \cite[Definition 2.21]{DY24} to the projection complex on $\PC$.
     \begin{defn}\label{defn:linearlyrecurrent}

     Let $\varepsilon\in (0,1]$ and $L>0$. An element $g\in G$ is said to be \emph{$(\varepsilon, L)$-linearly recurrent} if there exist  a linearly ordered set of distinct points on the standard path  $\gamma=\f_K[\po,g\po]$ in $\PC$ with cardinality $m+1$, where $m\ge \lfloor \varepsilon\, d_S(1,g)\rfloor$: $$\{\po=x_0<x_1<\cdots <x_{m}=g\po\}$$ and a linearly ordered set of elements on a  geodesic $[1,g]_S$ in $\mathrm{Cay}(G,S)$: $$\{1=g_0<g_1<\cdots<g_{m}=g\}$$ such that  
     \begin{equation}\label{eq:linearlyrecurrent}
     \begin{aligned}
     \LS(g_i \po, x_i)&\le 3, \quad \forall\, 0\le i\le m\\
     \LS(g_i \po, g_j \po)&\ge L,  \quad \forall\,  0\le i \ne j\le m
     \end{aligned}
     \end{equation}  
     Setting $s_i = g_{i-1}^{-1} g_{i}$ for $1\le i\le m$ yields a product decomposition $g = s_1 \cdots s_m$.

    \end{defn}
    \begin{rem}
    This differs from  \cite[Definition 2.21]{DY24} in choosing $M=3$ in (\ref{eq:linearlyrecurrent}) and in that $\gamma$ is a standard path. Note that $s_i$ are not   generators in general. We say that $g = s_1 \cdots s_m$ is a \textit{$3$-almost geodesic decomposition} if $g_i\po$ lies in the $3$-neighborhood of $\f_K[\po,g\po]$ (\cite[Definition 2.20]{DY24}).     
    \end{rem}

    Let $\mathcal{LR}(\varepsilon, L)$ denote the set of $(\varepsilon, L)$-linearly recurrent elements in $G$. An immediate consequence of Theorem~\ref{thm:grow_tight_PC} is as follows.
    %Note that if $\varepsilon_1 \ge \varepsilon_2$, $L_1 \ge L_2$, then $\mathcal{LR}(\varepsilon_1, L_1) \subseteq \mathcal{LR}(\varepsilon_2, L_2)$.

    \begin{lem}\label{lem:LRset_growth_tight}
    For any  $L\ge 8$, there exists  $\varepsilon=\varepsilon(L)>0$  so that $\mathcal{LR}(\varepsilon, L)$ is exponentially generic.    
    \end{lem}
    \begin{proof}
    By Theorem~\ref{thm:grow_tight_PC}, there exists $\varepsilon_0>0$ so that the set $A=\{g\in G: \LS(\po,g\po)\ge \varepsilon_0 d_S(1,g)\}$ is exponentially generic. Below is an adaptation of the argument in Lemma \ref{lem:split_path_into_pieces}.
    
    Let $L\ge 8$ be an integer and $m=\lfloor\LS(\po,g\po)/(2L+8)\rfloor$. Given $g\in A$, we subdivide the standard path $\f_K[\po,g\po]$ into segments of length $\ge 2L+8$; that is,   division points $y_0:=\po, y_{m}:=g\po$ and  $y_i$ satisfy  $\LS(y_i,y_{i+1})\ge 2L+8$ for $0\le i< m$. If $m= 1$, there is nothing to do. If $m\ge 2$, choose the middle point $x_i$ of the standard path from $y_{i-1}$ to $y_{i}$ for $1\le i < m$.  By the bottleneck property of Lemma \ref{lem:bottleneck_in_PC}, there exists $g_1\in [1,g]_S$ so that $\LS(g_1\po, x_1)\le 3$. By Lemma \ref{lem:common_prefix}, since $\LS(g_1\po,x_1)\le 3$ and $\LS(x_1,x_2)=2L+8\ge 6$,  $x_2$ also lies on the standard path from $g_1\po$ to $g\po$. Then Lemma \ref{lem:bottleneck_in_PC}  provides $g_2$ on $[g_1,g]_\gamma$ so that $\LS(g_2\po, x_2)\le 3$. Inductively we choose $g_{i+1}$  on $[g_i,g]_\gamma$ for $i+1< m$ so that $\LS(g_{i+1}\po, x_{i+1})\le 3$, and set $g_m=g,x_m=g\po$. Applying twice the weak triangle inequality for $\LS$ gives $\LS(g_i\po,g_j\po)\ge \LS(x_i,x_j)-\LS(x_i,g_i\po)-\LS(x_j,g_j\po)-2 \ge L$.
    
    To conclude, set $\varepsilon=\varepsilon_0/3L$ and then $m\ge \varepsilon d_S(1,g)$, so $\{x_i: 0\le i\le m\}$ and $\{g_0:0\le i\le m\}$ verify the linearly recurrent conditions of $g\in A$.    
    \end{proof}

 Let $H$ be a confined subgroup of $G$ with $P$ as a confining subset. Let $[G/H]\subseteq G$ denote any section of the natural projection $$G\to G/H=\{Hg: g\in G\}$$ 
 which picks exactly one element $g$ from each $Hg$ so that $d_S(1,g)=d_S(1,Hg)$. Thus,  $$\omega(G/H,\bar S)= \omega([G/H],S).$$ Here $\omega(G/H,\bar S)$ is the growth rate of the quotient metric on $G/H$ (or the combinatorial metric of the Schreier graph $G/H$).   
The following theorem is the technical core of this section. It shows that
the linearly recurrent representatives in any fixed Schreier section
have a uniform growth gap, and it immediately implies
Theorem~\ref{thm:growthtightnessAH}.
    \begin{thm}\label{AGroTig}
    Fix $L>\max\{L_0,L_1,4r\}$. For any $0<\varepsilon\le 1$,  $$A:=[G/H]\cap \mathcal{LR}(\varepsilon, L)$$ is growth tight in $G$. Moreover, the gap $(\omega_G-\omega_A)>0$ depends only on $\varepsilon$ and $P$ (but \emph{not} on $H$).
    \end{thm}

    \begin{proof}[Proof of Theorem~\ref{thm:growthtightnessAH}]
     For $L$ given by Theorem \ref{AGroTig} and $\varepsilon=\varepsilon(L)$ by Lemma \ref{lem:LRset_growth_tight}, $[G/H]$ is the union of two growth tight subsets. It follows that $$\omega(G/H,\bar S)=\omega([G/H],S)<\omega(G,S)$$ with a gap depending on $P$ but not on $H$. The theorem is proved.  
    \end{proof}

  \subsection{Proof of Theorem \ref{AGroTig}}\label{ConsMap}
  We first define an insertion map similar to that in \textsection \ref{def:insert_map}.  
  \medskip
  
  \noindent\textbf{Defining the insertion map}
 Let $g=s_1s_2\cdots s_m\in A$ be a linearly recurrent decomposition where $m \ge \lfloor\varepsilon d_S(1,g)\rfloor$, as in Definition \ref{defn:linearlyrecurrent}.  Recall $g_0=1$, and $g_i=s_1\cdots s_i$ for $1\le i\le m$, and for $0\le i< j\le m$,  
\[
g_i^{-1}g_j = s_{i+1}s_{i+2}\cdots s_j.
\]
By Lemma~\ref{ExtensionLemConfined}, for each $i$ with $0\le i\le m-1$ there exist $q_i\in Q$ and $p_i\in P$ such that $g_{i}q_ip_iq_i^{-1}g_{i}^{-1}\in H$ and 
\begin{equation}
\tag{$\ddagger$}\label{eq:gkconjugates}
\begin{array}{lr}
p_i\ax(q_i)\ne \ax(q_i) , \text{ and }  [\po,g_i\po],[\po,p_i\po],[\po,g_i^{-1}g_m\po] \text{ have } \tau_0 \text{–bounded projection to}\,\ax(q_i).  & 
\end{array}
\end{equation}
  
Note that $m$ may depend on the specific element $g$.  Let $\mathcal P_m:=\mathcal P(\{0,1,\cdots,m-1\})$ be the power set. 
%Recall first that   $\str$ denotes the space of  strings with length $m$ over $\{0,1\}$. Equivalently, $\str$  could be understood as the power set  over $\{0,1,2,\cdots, m-1\}$.  We refer to \textsection \ref{subsec:notation} for the detail.
For each recurrence position $g_i$, Lemma~\ref{ExtensionLemConfined}
provides $q_i\in Q$ and $p_i\in P$ such that
\[
g_iq_ip_iq_i^{-1}g_i^{-1}\in H.
\]
Thus, inserting $q_ip_iq_i^{-1}$ at the position $g_i$ changes the resulting
element by left multiplication by an element of $H$ and therefore preserves
its right coset. For a subset $I$ of the recurrence positions, the map
$\Phi_g(I)$ performs precisely the insertions indexed by $I$. Precisely,

\begin{defn}\label{def:extMap}
Let $I\in\mathcal P_m$ be non-empty, enumerated in increasing order as
\[
I=\{i_1<i_2<\dots<i_\alpha\}, \qquad \alpha=|I|.
\]
%and set $i_{\alpha+1}=m+1$. 
Define the map $\Phi_g: \mathcal P_m \to G$ by
\[
\Phi_g(I)
= g_{i_1}(q_{i_1}p_{i_1}q_{i_1}^{-1})\, g_{i_1}^{-1}g_{i_2}(q_{i_2}p_{i_2}q_{i_2}^{-1})
\cdots
(q_{i_\alpha}p_{i_\alpha}q_{i_\alpha}^{-1})\, g_{i_\alpha}^{-1}g
\]  
and $\Phi_g(\varnothing) = g$.
\end{defn}

\begin{lem}\label{lem:criterion12}
For any $g\in A$, the map $\Phi_g: \mathcal  P_m\to G$ satisfies the following properties:
\begin{enumerate}[label=(\roman*)]
    \item $\operatorname{Im}(\Phi_g) \subseteq Hg$;
    \item For every non-empty $I\in \mathcal  P_m$, the element $\Phi_g(I)$ labels an $(L_1,\tau_1)$-admissible path  $\gamma_g(I)$  in $\PC$.
\end{enumerate} 
\end{lem}

\begin{proof}
We briefly recall the proof of~(\textit{i}) from \cite[Lemma~4.4]{DY24}, as it provides the key idea behind Lemma~\ref{lem:criterion3}. Suppose that
\[
J=\{k\}\cup I=\{k<i_1<i_2<\cdots<i_\alpha\},
\]
so that \(I\) and \(J\) first differ at the index \(k\). If
\[
\Phi_g(I)\in Hg \quad\Longrightarrow\quad \Phi_g(J)\in Hg,
\]
then the general case follows by modifying the index set one insertion at a time. Indeed,
\begin{align}\label{eq:conjugatemove}
\Phi_g(J)=g_k(q_kp_kq_k^{-1})g_k^{-1}\Phi_g(I),
\end{align}
and the conjugating element \(g_k(q_kp_kq_k^{-1})g_k^{-1}\) belongs to \(H\), so right cosets are preserved.

Property~(\textit{ii}) is established in \cite[Lemma~4.8]{DY24}. We only explain the additional point relevant to the present setting and sketch the argument.

Unlike the proof of Lemma~\ref{lem:(gI)labeledpath}, the construction of \(\Phi_g\) here is more restrictive: the elements \(q_k\) and \(p_k\) must be chosen \emph{a priori}, independently of the insertion index \(k\in I\); see the \(\tau_0\)-bounded projection condition in~\eqref{eq:gkconjugates}. Consequently, the essential task is to verify that the geodesic $[g_{i_1}\po,g_{i_2}\po]$
has uniformly bounded projection to the adjacent (appropriately translated) axes \(\ax(q_{i_1})\) and \(\ax(q_{i_2})\).

To see this, observe that the endpoints of \([g_{i_1}\po,g_{i_2}\po]\) are \(3\)-close to the corresponding endpoints of the geodesics \([g_{i_1}\po,g\po]\) and \([\po,g_{i_2}\po]\); see~\eqref{eq:linearlyrecurrent}. By the defining property of the constant in~\ref{cst:tau}, together with~\eqref{eq:gkconjugates}, these two geodesics have \(\tau_1\)-bounded projection to the corresponding translates of \(\ax(q_{i_1})\) and \(\ax(q_{i_2})\), respectively. The same therefore holds for \([g_{i_1}\po,g_{i_2}\po]\). Moreover, these two translated axes are distinct by the choice of the constant in~\ref{cst:L}. Hence all conditions for an \((L_1,\tau_1)\)-admissible path are satisfied by \(\Phi_g(I)\).

We refer the interested reader to \cite[Lemma~4.8]{DY24} for the complete proof.
\end{proof}

\begin{lem}\label{lem:criterion3}
For any $g\in A$, the map $\Phi_g: \mathcal P_m\to G$  is injective. 
\end{lem}

\begin{proof} 
This is an abridged version of \cite[Lemma~4.10]{DY24}, which we include for completeness.

Assume, toward a contradiction, that \(I\neq I'\) but
\[
\Phi_g(I)=\Phi_g(I')=:h.
\]
Let \(k\) be the smallest integer lying in exactly one of \(I\) and \(I'\); that is, \(k\) is the first position where the insertion of \(q_kp_kq_k^{-1}\) differs. Without loss of generality, assume that \(k\in I\) and \(k\notin I'\).

Let \(i_1<i_2<\cdots<i_l\) be the common indices of \(I\) and \(I'\) that are smaller than \(k\). Define
\[
\tilde g_k
=
g_{i_1}(q_{i_1}p_{i_1}q_{i_1}^{-1})g_{i_1}^{-1}
g_{i_2}(q_{i_2}p_{i_2}q_{i_2}^{-1})
\cdots
(q_{i_l}p_{i_l}q_{i_l}^{-1})g_{i_l}^{-1}g_k,
\]
and set \(\tilde g_k=g_k\) if there are no such common indices. Let
\[
t=\tilde g_kq_kp_kq_k^{-1}\tilde g_k^{-1},
\]
and define $Y=\tilde g_k\ax(q_k),\;
Y'=tY.$
We shall derive contradictory estimates for
\(\proj_Y^\pi(\po,h\po)\).

\medskip

\noindent\textbf{(1)}
By Lemma~\ref{lem:criterion12}, \(\gamma_g(I)\) is an \((L_1,\tau_1)\)-admissible path. The corresponding consecutive contracting subsets are precisely \(Y\) and \(Y'=\tilde g_kq_kp_k\ax(q_k)\). By Proposition~\ref{admisProp}, the projection of \(\po\) to \(Y\) is \(r\)-close to \(\tilde g_k\po\), while the projection of \(\Phi_g(I)\po\) to \(Y\) is \(r\)-close to \(\tilde g_kq_k\po\). Hence
\[
\proj_Y^\pi(\po,h\po)
=
\proj_Y^\pi\!\bigl(\po,\Phi_g(I)\po\bigr)
>
L-2r.
\]

\medskip

\noindent\textbf{(2)}
Now insert the index \(k\) into \(I'\) by setting
$J:=I'\cup\{k\},$
and let \(h':=\Phi_g(J)\). By~\eqref{eq:conjugatemove},
\[
h'
=
\tilde g_kq_kp_kq_k^{-1}\tilde g_k^{-1}\Phi_g(I')
=
t\,\Phi_g(I').
\]
Applying Proposition~\ref{admisProp} to the admissible path \(\gamma_g(J)\), the projection of \(h'\po\) to \(Y'=tY\) is \(r\)-close to
\(\tilde g_kq_kp_kq_k^{-1}\po\). Translating back by \(t^{-1}\), the projection of \(\Phi_g(I')\po\) to \(Y\) is therefore \(r\)-close to \(\tilde g_k\po\). Consequently,
\[
\proj_Y^\pi(\po,h\po)
=
\proj_Y^\pi\!\bigl(\po,\Phi_g(I')\po\bigr)
\le 2r.
\]

If \(L>4r\), these two estimates contradict the assumption
\(h=\Phi_g(I)=\Phi_g(I')\). Therefore, \(\Phi_g\) is injective.
\end{proof}

The growth gap is a consquence of the following criterion in \cite{DY24}.

\begin{lem}\label{CritGapLem}
Let $A\subseteq G$ be a subset, $Q$ be a  finite non-empty set, $\varepsilon\in (0,1)$, and $M>0$. Suppose the following conditions hold:

\begin{enumerate}
\item
 Each $g\in A$ admits an $M$-almost geodesic product decomposition $g=s_1\cdots s_m$ with $m\ge \varepsilon d_S(1,g)$ (where $m$ may depend on $g$).
 \item
  For every $g\in A$, the map $\Phi_g: \mathcal P_m\to G$ described in Definition \ref{def:extMap} is  injective where $m$ is the length of the decomposition of $g$ given in (1).  
  \item
The images of the maps $\Phi_g$ are pairwise disjoint: for distinct $g, g'\in A$, we have $\Phi_g(\mathcal P_m)\cap \Phi_{g'}(\mathcal P_{m'})=\emptyset$, where $m,m'$ are the lengths of the decomposition associated to $g, g'$ respectively.
\end{enumerate} Then $\omega_G >\omega_A$. Moreover, the gap $\omega_G - \omega_A$ depends only on $\varepsilon, M, \omega_G$ and $\max\{\LS(\po,q\po): q\in Q\}$ (but not on $A$ itself).
 
\end{lem}
By Lemma~\ref{lem:criterion12}, every element of
$\Phi_g(\mathcal P_m)$ lies in the coset $Hg$, while
Lemma~\ref{lem:criterion3} shows that $\Phi_g$ is injective. Thus each
$g\in A$ produces $2^m$ distinct elements in its coset. Moreover, the images
associated with distinct representatives in $[G/H]$ are disjoint because
they lie in distinct cosets. Hence all the hypotheses of
Lemma~\ref{CritGapLem} are satisfied, and Theorem~\ref{AGroTig} follows.
This completes the proof of Theorem~\ref{thm:growthtightnessAH}.

\section{Application II: Counting conjugacy classes in acylindrically hyperbolic groups}\label{SecConjClas}
We prove \ref{MainThmConjugacyClasses} in this section. Assume that a finitely generated group $(G,S)$ acts on a geodesic metric space $X$ with a basepoint $o$. Choose a set $F\subseteq G$ of three pairwise independent strongly contracting WPD elements, and form the associated WPD $C$-contracting axis  system  $\f$  for some $C>0$.  
Consider  the projection complex $\PC$  for $K\gg 3\theta$. We refer the reader to Section~\ref{sec:anchoredlength} for further details concerning the constants $(C, \theta, \beta)$. 

Let $\mathcal C_n$ denote the set  of conjugacy classes $[g]$ of elements $g\in G$ that intersect the ball $B_n$. Namely, 
$$
\mathcal  C_n =\{[g]: g\in B_n\}
$$
\begin{thm}\label{thm:countConj}
    There exist constants $c_1, c_2>0$ (depending only on $G$, $S$, and $\PC$) such that for every $n\in \mathbb{N}$, $$c_1\cdot \frac{|B_n|}{n}\le |\mathcal  C_n|\le c_2\cdot \frac{|B_n|}{n}$$
\end{thm}
We prove the lower and upper bounds in Propositions~\ref{prop:conj_lower_bound} and \ref{prop:conj_upper_bound}, respectively. We remark that the lower bound does not rely on exponential genericity of WPD elements established in this paper, whereas the upper bound does.  
\subsection{Lower bound of conjugacy classes}
Fix  a fixed basepoint $o\in X$. Let $g\in G$ be an arbitrary element. By the extension lemma (Lemma~\ref{extend3}), there exists a contracting element $f_g\in F$ such that $(\cdots, f_g, g, f_g, g, \cdots)$ labels an $(L, \tau)$-admissible path in $X$:
$$
\gamma_g :=\bigcup_{i\in \mathbb Z} (f_gg)^i[o,f_go][f_go,f_ggo].
$$ By Proposition~\ref{admisProp}, there are constants $c$ and $r$ such that $\gamma_g$ is a $c$-quasi-geodesic with the $r$-fellow travel property. Set $R=3\theta+r+C$, and pick $K>3\theta$ as the projection complex constant. By replacing $f\in F$ with a sufficiently large power, we may assume that $$d(o, fo)>\max\{2r+3C+K+2\theta, 2R+C\},$$ and  the coarse stabilizers  are  finite: $$M_0:=\max\{|\mathrm{Stab}_{12R+2}(o, fo)|: f\in F\}$$ since $f\in F$ acts as WPD elements on $X$ (Definition \ref{defn:WPD}).

Let $L_0 = \max\{d_S(1, f): f\in F\}$. Fix $n>L_0$. Since $|F|=3$, there exists some $f\in F$ so that the set $A_n$ of elements $f_gg$ with $g\in B_{n-L_0}$ and $f_g=f$, i.e., $$A_n=\{f_gg: f_g=f, g\in B_{n-L_0}\}$$
satisfies   $|A_n| \ge |B_{n-L_0}|/3$. Write  $Y=\ax(f)$ for notational simplicity.
\begin{lem}\label{lem:conj_near_axis}
    Let $g, h\in G$ so that $fg, fh\in A_n$ and $p(fg)p^{-1}=fh$ for some $p\in G$. Then there exists $k\in \mathbb{Z}$ such that $$\proj^{\pi}_{Y}(q(fg)^{-1}o, o)\le R,\quad \proj^{\pi}_Y(qfgo, fo)\le R$$ where $q=p(fg)^k$.
\end{lem}
\begin{proof}
    By Proposition~\ref{admisProp},  the shortest projections of $(fg)^{-m}o$ and $(fg)^mo$ to $Y$ are $r$-close to $o$ and $fo$ respectively for each integer $m\ge 1$. Then $\proj^\pi_Y((fh)^{-m}o, (fh)^mo)\ge d(o,fo)-2r$. Further, for all sufficiently large $m$, $$
    \begin{aligned}
    d((fg)^{-m}o, Y)&\ge d((fg)^{-m}o, o) - r\\
    &\ge \left(\frac{md(o, fo)}{c}-c\right)-r>d(o, po),    
    \end{aligned}$$ where the second inequality is due to the $c$-quasi-geodesicity of $\gamma_g$. Similarly, we derive $d((fh)^mo, Y)>d(o, po)$. Therefore, the $C$-contracting property of $Y$ gives $$\proj^\pi_Y((fh)^{-m}o, (fh)^{-m}po)\le C,\quad \proj^\pi_Y((fh)^mo, (fh)^mpo)\le C.$$ By the conjugation assumption, $(fh)^{-m}p = p(fg)^{-m}$ and $(fh)^mp = p(fh)^m$. These together show  
    $$\begin{aligned}
        \proj^\pi_Y(p(fg)^{-m}o, p(fg)^mo)&\ge \proj^\pi_Y((fh)^{-m}o, (fh)^mo) - \proj^\pi_Y(p(fg)^{-m}o, (fh)^{-m}o) - \proj^\pi_Y(p(fg)^{m}o, (fh)^{m}o) \\
        &\ge d(o, fo) - 2r-2C>K+2\theta.
    \end{aligned}$$
    By Eq. (\ref{eq:projdist}) in Theorem~\ref{thm:forcing_seq}, $$Y\in \f_K(p(fg)^{-m}\po, p(fg)^m\po).$$ 
    For each $-m\le k\le m-1$, applying Proposition~\ref{admisProp} to the contracting axis $p(fg)^kY$ of $p\gamma_g$, $$\begin{aligned}
        \proj^\pi_{p(fg)^kY}(p(fg)^{-m}o, p(fg)^mo)&\ge d(o, fo) - \proj^\pi_{p(fg)^kY}(p(fg)^{-m}o, p(fg)^ko) - \proj^\pi_{p(fg)^kY}(p(fg)^{k+1}o, p(fg)^mo)\\
        &\ge d(o, fo)-2r>K+2\theta.
    \end{aligned}$$ So by Eq. \ref{eq:projdist}, we verify that $$\{p(fg)^{-m}Y,p(fg)^{-m+1}Y, \cdots, p(fg)^{m-1}Y\}$$ is contained in the standard path $\f_K(p(fg)^{-m}\po, p(fg)^m\po)$ in the order that $p(fg)^{i}Y<p(fg)^{i+1}Y$ for each $-m\le i\le m-2$. 
    
    Let $k$ be the largest integer with $-m-1\le k\le m$ for which $p(fg)^kY\le Y$ in the standard path $\f_K(p(fg)^{-m-1}\po, p(fg)^{m+1}\po)$. Then $-m\le k\le m-1$. Set $q = p(fg)^k$. Since $q(fg)^{-1}Y = p(fg)^{k-1}Y<Y$, Proposition~\ref{prop:std_path} gives $\proj_Y(p(fg)^{-m-1}\po, q(fg)^{-1}Y)\le \theta$. Therefore, Eq. (\ref{eq:projdist}) and triangle inequality gives $$\begin{aligned}
        \proj^\pi_Y(q(fg)^{-1}o, o)&\le \proj^\pi_Y(q(fg)^{-1}o, p(fg)^{-m-1}o) + \proj^\pi_Y(p(fg)^{-m-1}o, (fh)^{-m-1}o) + \proj^{\pi}_Y((fh)^{-m-1}o, o)\\
        &\le \proj_Y(p(fg)^{-m-1}\po, q(fg)^{-1}Y) + 2\theta + C + r\\
        &\le 3\theta + r+C\le R.
    \end{aligned}$$
    Similarly, since $Y<p(fg)^{k+1}Y=qfgY$, we obtain $\proj^\pi_Y(qfgo, fo)\le 3\theta+r+C\le R.$
\end{proof}
We next show that each conjugacy class contains at most linearly many elements in $A_n$.
\begin{lem}\label{lem:conj_class_in_A}
    Set $M=2\beta M_0$. Then  each conjugacy class of $G$ contains at most $Mn$ elements in $A_n$.
\end{lem}
\begin{proof}
    Fix any $fg\in A_n$ and a segment $\alpha=[(fg)^{-1}o, fgo]$. Note that  $$\len(\alpha)=d((fg)^{-1}o, (fg)o)\le 2\beta(d_S(1, f) + d_S(1, g))\le 2\beta n$$  by the $\beta$-Lipschitz orbital map.  The goal  is to show that $[fg]$ contains  at most $Mn$ different $fh\in A_n$.

    Write $Y = \ax(f)$. By Lemma~\ref{lem:conj_near_axis}, for each $fh\in [fg]\cap A_n$,  there exists $q=q(h)\in G$ such that $fh=qfgq^{-1}$ and $$\proj^{\pi}_Y(q(fg)^{-1}o, o)\le R, \quad \proj^{\pi}_Y(qfgo, fo)\le R.$$ So $\proj^{\pi}_Y(q(fg)^{-1}o, qfgo)\ge d(o, fo) - 2R>C$. Applying Lemma~\ref{BigFive}(3) to $q\alpha$, there exist two points $x=x(q), y=y(q)\in \alpha=[(fg)^{-1}o,fgo]$ such that $d(qx, \pi_Y(q(fg)^{-1}o)), d(qy, \pi_Y(q(fg)o)) \le 2C$. Thus  
    \begin{equation}\label{eq:qxqyofo}
    \begin{aligned}
    d(qx, o) &\le d(qx, \pi_Y(q(fg)^{-1}o)) + \proj^\pi_Y(q(fg)^{-1}o, o)\le 2C+R\le 2R,\\
    d(qy, fo)&\le d(qy, \pi_Y(q(fg)o)) + \proj^\pi_Y(q(fg)o, o)\le 2C+R\le 2R.    
    \end{aligned}    
    \end{equation}
    %Note that $[q(fg)^{-1}o, qfgo] = q[(fg)^{-1}o, fgo]$. Let $x=q^{-1}u$ and $y=q^{-1}v$, then $$d(x, q^{-1}o)\le N, d(y, q^{-1}fo)\le N.$$
    In summary, we produce  a set  denoted by $Q$ of  elements $q=q(h)$ for each $fh\in A\cap [fg]$, and a pair of points $x(q), y(q)\in \alpha$ so that  (\ref{eq:qxqyofo}) holds. It thus suffices to bound $|Q|\le Kn$.
    
   To that end, we divide $\alpha=[(fg)^{-1}o, fgo]$ into $2\beta n$ subsegments of length at most 1. Fix any such subsegment  $\alpha_0$. We claim that there are at most $M$ elements $q\in Q$ so that the point $x=x(q)$ chosen above is contained in $\alpha_0$. 

    Indeed, suppose $q_1\in Q$ and $q_2\in Q$ admit $(x_1, y_1)$ and $(x_2, y_2)$ respectively such that $d(x_1, x_2)\le 1$. Noting that $x_1, x_2, y_1, y_2$ lie on the geodesic $\alpha$, we obtain from (\ref{eq:qxqyofo}) that 
    $$\begin{aligned}
        d(y_1, y_2)&\le d(x_1, x_2) + |d(x_1, y_1) - d(x_2, y_2)|\\
        &\le 1+|d(q_1x_1, q_1y_1)-d(o, fo)| + |d(q_2x_2, q_2y_2) - d(o, fo)|\\
        &\le 1+8R.
    \end{aligned}$$
    So by triangle inequality, $$
    \begin{aligned}
    d(o, q_1q_2^{-1}o) &= d(q_1^{-1}o,q_2^{-1}o) \le d(q_1^{-1}o, x_1) + d(x_1, x_2) + d(x_2, q_2^{-1}o)\le 4R+1,\\
    d(fo, q_1q_2^{-1}fo) &= d(q_1^{-1}fo,q_2^{-1}fo) \le d(q_1^{-1}fo, y_1) + d(y_1, y_2) + d(y_2, q_2^{-1}fo)\le 12R+2.    
    \end{aligned}$$
    This yields $$q_1q_2^{-1}\in Stab_{12R+2}(o, fo)$$ so the claim holds.

    Hence, $|Q|\le 2\beta\cdot |Stab_{12R+2}(o, fo)|\le 2\beta M_0=M$, which concludes the proof.
\end{proof}
\begin{prop}\label{prop:conj_lower_bound}
    There exists a constant $c_1>0$ such that $$|\mathcal C_n|\ge c_1\cdot \frac{|B_n|}{n}$$ for any $n\in \mathbb{N}$.
\end{prop}
\begin{proof}
    We may assume $n>L_0$ up to decreasing the constant $c_1$. By Lemma~\ref{lem:conj_class_in_A}, the set $A_n\subseteq B_n$ contains at least $\frac{|A_n|}{Mn}$ different conjugacy classes. On the other hand, recall that $$|A_n|\ge \frac{|B_{n-L_0}|}{3}\ge \frac{|B_n|}{3|B_{L_0}|},$$ where the submultiplicative inequality on $|B_n|$ is used.  Hence, setting $c_1 = 1/(3M|B_{L_0}|)$   finishes the proof.
\end{proof}
\subsection{Upper bound of conjugacy classes}
We apply our main \ref{MainThmWPDGeneric} to the acylindrical hyperbolic action of $G$ on the projection complex $\PC$. Then exponentially generic elements in $G$ act as loxodromic elements on $X$. Fix a basepoint $\po\in \PC$ and denote $X=\PC$ for simplicity in what follows.  

The next two   facts hold for any loxodromic elements in any acylindrical hyperbolic action.  Recall that $\tau_X(g)$ is the stable translation length of $g$ on $X$.

%The following lemma is a result of Bowditch \cite{Bow08}.% we include the proof here for completeness. Note that $X$ is $\delta$-hyperbolic for some $\delta>0$. Let $$\langle a, b\rangle_c = \frac{1}{2}(d(a, c) + d(b, c) - d(a, b))$$ denote the Gromov product.
\begin{lem}\cite[Corollary 1.5]{Bow08}\label{lem:minimal_stable_length}
    There exists $\tau_0>0$ such that for each loxodromic element $g$, $$\tau_X(g)=\inf\left\{\frac{d(\po,g^n\po)}{n}: n\ge 1\right\}\ge \tau_0.$$ 
\end{lem}
\iffalse
\begin{proof}
    Let $l(g)$ be the minimal translation length for any loxodromic element $g\in G$, which exists since $X$ is endowed with the graph metric. Fix any loxodromic element $g\in G$. For any $x\in X$, let $p$ be the midpoint of $[x, gx]$, then $gp$ is the midpoint of $[gx, g^2x]$. Applying $\delta$-hyperbolicity to the quadrilateral $p, gp, g^2x, x$, we get $$\langle p, gp\rangle_{gx}\ge \min\{\langle gp, g^2x\rangle_{gx}, \langle g^2x, x\rangle_{gx}, \langle x, p\rangle_{gx}\}-4\delta= \min\{\frac{d(x, gx)}{2}, d(x, gx)-\frac{d(x, g^2x)}{2}\} - 4\delta.$$
    By the definition of Gromov product, we have $$d(p, gp)\le \max\{0, d(x, g^2x) - d(x, gx)\} + 8\delta.$$ Note that the triangle inequality gives $d(x, g^2x)\le 2d(x, gx)$. So $d(p, gp)\le d(x, g^2x)/2+8\delta$. Since $x$ is an arbitrary point in $X$, we have $l(g^2)\ge 2l(g)-16\delta$. Hence, for any $k\in \mathbb{N}$, $$l(g^{2^k})\le 2^kl(g) - 16\delta(2^k-1).$$ So $\tau(g)\le l(g)-16\delta$.
\end{proof}
\fi
For any loxodromic element $g\in G$, define $$E^+(g):=\{h\in G: \exists n\in \mathbb Z_{+}, g^nh=hg^n\}.$$ Then $E^+(g)$ is a subgroup of $E(g)$ with index at most 2.

%The following lemma from \cite{GYANG} is useful.
\begin{lem}\cite[Lemma 6.8]{Osin6}\label{lem:exact_axis}
    There exists $N>0$ such that for any loxodromic element $g\in G$ on $X$, there exists an exact sequence $$1\to F_g\to E^+(g)\to \mathbb Z\to 1,$$ where $|F_g|\le N$.
\end{lem}
\iffalse
\begin{proof}
    Let $x\in X$. 
\end{proof}
\fi
\begin{lem}\label{lem:axis_in_ball}
    There exists a constant $E>0$ such that for any loxodromic element $g\in G$ and $n\in \mathbb{N}$, $$|E^+(g)\cap B_n|\le En.$$
\end{lem}
\begin{proof}
    Denote $\varphi: E^+(g)\to \mathbb Z$  the corresponding map given in Lemma~\ref{lem:exact_axis}. Choose $f\in E^+(g)$ such that $\varphi(f)=1$. Then any $h\in E^+(g)$ has a unique normal form $af^k$ for some $a\in F_g$ and $k\in \mathbb Z$. Let $$L_0 = \max_{a\in F_g} d(\po, a\po).$$ Noting that for any $n\in \mathbb N$, $\varphi((af^k)^n)=kn$, so $(af^k)^n(f^k)^{-n}\in F_g$. Since $f$ is loxodromic,  it follows  by Lemma~\ref{lem:minimal_stable_length} that $d(\po, f^{kn}\po) \ge |k|n \tau_0$. Thus for any $n\in \mathbb N$, $$d(\po, (af^k)^n\po) \ge d(\po, f^{kn}\po) - d(\po, (af^k)^nf^{-kn}\po)\ge |k|n\tau_0-L_0,$$ which implies $\tau(af^k) \ge |k|\tau_0$. Note that $\tau_X(h)\le  d(\po, h\po)\le \beta d_S(1, h)$ for any $h\in G$. So we have $$|E^+(g)\cap B_n|\le |E^+(g)\cap \{h: \tau_X(h)\le \beta n\}|\le \sum_{a\in F_g} |\{k\in \mathbb Z: \tau(af^k)\le \beta n\}\le |F_g|\cdot \frac{2\beta n}{\tau_0}.$$ 
    Hence, setting $E = 2N\beta/\tau_0$ finishes the proof.
\end{proof}
\begin{prop}\label{prop:conj_upper_bound}
    There exists a constant $c_2>0$ such that $$|\mathcal C_n|\le c_2\cdot \frac{|B_n|}{n}$$ for any $n\in \mathbb{N}$.
\end{prop}
\begin{proof}
    Let $\mathcal{NL}$ denote the set of non-loxodromic elements in $G$. Then by~\ref{MainThmWPDGeneric}, $\mathcal{NL}$ is growth tight in $G$. We look at the loxodromic elements $g\in B_n$ so that $E^+(g)$ contains some loxodromic element in $B_{n/2}$:  $$\mathcal{SL}(n) := \left \{g\in B_n\setminus \mathcal{NL}: (B_{n/2}\setminus \mathcal{NL})\cap E^+(g)\ne \emptyset.\right\}$$  
    If $h\in E^+(g)$ is loxodromic, then $E^+(h)=E^+(g)$. By Lemma~\ref{lem:axis_in_ball}, $$|\mathcal{SL}(n)|\le \sum_{h\in B_{n/2}\setminus \mathcal{NL}} |E^+(h)\cap B_n|\le En|B_{n/2}|.$$
    Since
$\lim_{n\to\infty}n^{-1}\log|B_n|=\omega(G,S)>0$,
this estimate and the growth tightness of $\mathcal{NL}$ imply
that there exists an integer $n_0\ge2$ such that
\[
|\mathcal{NL}(n)|\le\frac{|B_n|}{n},
\qquad
|\mathcal{SL}(n)|\le\frac{|B_n|}{n}
\]
for every $n\ge n_0$.

    Set $A_n=B_n\setminus \mathcal{NL}(n)\cup \mathcal{SL}(n)$. We claim that for any element $g\in A_n$, $$|[g]\cap B_n|\ge \frac{n}{4N}.$$

    Indeed, noting that $B_{n/2}\subseteq \mathcal{NL}(n)\cup \mathcal{SL}(n)$, we have $g\in B_n\setminus B_{n/2}$. So we can write $g = s_1\cdots s_m$ as a geodesic word where $n/2\le m\le n$. If we denote $g_i := s_1 s_2\cdots s_i$ for  $0\le i\le  n/4$ with $g_0=1$, then $$g_i^{-1}gg_i = s_{i+1} \cdots s_ms_1\cdots s_i$$ so $g_i^{-1}gg_i\in B_n$.

    Suppose that $g_i^{-1} g g_i = g_j^{-1} g g_j$ for some $0\le i\ne j\le n/4$. Then $g_j g_i^{-1} g g_ig_j^{-1}=g$ and so $g_ig_j^{-1}\in E^+(g)$. Note that $$d_S(1, g_ig_j^{-1})\le d_S(1, g_i) + d_S(1, g_j)\le n/2$$ Since $g\not\in \mathcal{SL}(n)$, $g_ig_j^{-1}$ is not loxodromic on $X$: indeed, if $g_ig_j^{-1}$ is loxodromic, then $E(g_ig_j^{-1})=E(g)$ would intersect $B_{n/2}$, contradicting the above definition of $\mathcal{SL}(n)$. Hence, $g_ig_j^{-1}\in F_g$. Therefore, for a fixed $i$, there are at most $N$ different $j$ such that $g_i^{-1} g g_i = g_j^{-1} g g_j$. This in turn implies that $$|\{g_i^{-1}gg_i: 0\le i\le n/4\}\ge \frac{n}{4N},$$ so the claim holds.
    
    Consequently, at most $4N|B_n|/n$ conjugacy classes meet $A_n$. For every $n\ge n_0$, we therefore obtain
\[
|\mathcal C_n|
\le
|\mathcal{NL}(n)|+|\mathcal{SL}(n)|
+\frac{4N|B_n|}{n}
\le
(4N+2)\frac{|B_n|}{n}.
\]
For $1\le n<n_0$, the trivial bound
$|\mathcal C_n|\le |B_n|$ applies.
Thus choosing $c_2:=\max\{4N+2,n_0\}$
gives the required estimate for every $n\ge1$.
\end{proof}
	\bibliography{bibliography}
	\bibliographystyle{alpha}
\end{document}